\documentclass{article}

\usepackage[T1]{fontenc}
\usepackage{lmodern}
\usepackage{amsfonts,amsmath,amssymb,amsthm}
\usepackage{graphicx,color}
\usepackage[margin=1in]{geometry}
\usepackage{tikz}
\usepackage{blkarray}
\usepackage{enumitem}
\usepackage{caption}
\usepackage{subcaption}
\usepackage[dvipsnames]{xcolor}
\usepackage{comment}
\usepackage[colorlinks=true,linkcolor=RoyalBlue,citecolor=PineGreen,urlcolor=RoyalBlue]{hyperref}
\usepackage[nameinlink]{cleveref}
\usepackage{thmtools} 
\usepackage{thm-restate}

\usepackage{kbordermatrix}

\newtheorem{theorem}{Theorem}[section]
\newtheorem{proposition}[theorem]{Proposition}
\newtheorem{lemma}[theorem]{Lemma}
\newtheorem{corollary}[theorem]{Corollary}
\newtheorem{conjecture}[theorem]{Conjecture}
\newtheorem{claim}[theorem]{Claim}

\theoremstyle{definition}
\newtheorem{definition}[theorem]{Definition}

\theoremstyle{remark}
\newtheorem{remark}[theorem]{Remark}

\newcommand\CC{{\mathbb C}}

\newcommand\HH{{\mathbb H}}

\newcommand\NN{{\mathbb N}}
\newcommand\PP{{\mathbb P}}

\newcommand\RR{{\mathbb R}}

\DeclareMathOperator{\arsinh}{arsinh}

\DeclareMathOperator{\id}{id}

\DeclareMathOperator{\rank}{rank}

\DeclareMathOperator{\SL}{SL}
\DeclareMathOperator{\PSL}{PSL}
\DeclareMathOperator{\im}{im}
\DeclareMathOperator{\re}{re}

\title{Rigidity on compact surfaces through hyperbolic symmetries}
\author{Sean Dewar, Alison La Porta, Rebecca Monks, Anthony Nixon,\\ Klara Stokes and Joannes Vermant}
\date{}

\begin{document}

\maketitle

\begin{abstract}
Generically the rigidity of bar-joint structures admits combinatorial characterisations in the Euclidean plane and, more generally, for frameworks on the sphere and the torus. The remaining case of compact surfaces of genus at least two has remained open. Using the hyperbolic geometry of their universal covers, we develop a theory of infinitesimal rigidity for frameworks on compact surfaces of genus at least two. By the uniformisation theorem, every such surface is a quotient of the hyperbolic plane by a surface group, allowing frameworks on the surface to be represented as infinite symmetric frameworks in the hyperbolic plane. Encoding the symmetry through gain graphs, we prove that infinitesimal rigidity is determined entirely by finite combinatorial data. Specifically, a framework is generically rigid if and only if its associated gain graph contains a spanning (2,3,1,0)-gain tight subgraph. This yields the first combinatorial characterisation of generic rigidity for frameworks on compact surfaces of genus at least two.
\end{abstract}

\section{Introduction}
A central problem in rigidity theory is to determine when infinitesimal rigidity admits a purely combinatorial characterisation. Beginning with the work of Geiringer and Laman \cite{Laman1970,HPG27}, generic bar-joint frameworks in the Euclidean plane admit a complete combinatorial description. By projective equivalence, the Geiringer--Laman characterisation extends to the sphere and the hyperbolic plane \cite{I09,saliola_whiteley_2007_equivalence_rigidity}. A combinatorial characterisation is also known for the torus \cite{MR3265750,ross2015inductive}. The corresponding problem for compact orientable surfaces of higher genus has, however, remained open.

By the uniformisation theorem, every compact orientable surface of genus at least two carries a hyperbolic metric and can be realised as a quotient $\mathbb H/\Gamma$ of the hyperbolic plane by a co-compact surface group $\Gamma$. Accordingly, a framework on a surface lifts naturally to a $\Gamma$-symmetric framework in the hyperbolic plane. The central question is whether the rigidity of such an infinite symmetric framework can still be determined from finite combinatorial data. 

The infinitesimal rigidity of symmetric frameworks has attracted considerable attention in recent years, motivated in part by applications in engineering, robotics, biophysics, materials science and computer-aided design (see \cite{cnswpair,dgl,jzt2016,sw2011,tanigawamatroids}). 
We build on this work by developing a theory of symmetric infinitesimal rigidity for frameworks in the hyperbolic plane, with symmetry groups given by \textit{Fuchsian groups}, which are discrete subgroups of $\PSL(2,\mathbb{R})$. In particular, we focus on the case where the symmetry group is a surface group. 

Unlike periodic frameworks on the torus, the symmetry group is no longer generated by translations. Moreover, while projective methods relate generic rigidity in Euclidean and hyperbolic geometry, they do not preserve group actions and therefore cannot be used to study symmetric rigidity. 
Consequently, the existing theory cannot be applied directly. Instead, our approach is based on the orbit rigidity matrix together with gain graph realisations in the hyperbolic plane.

Given a group $\Gamma$, a $\Gamma$-gain graph is a directed multigraph equipped with an assignment of a group element of $\Gamma$ to each directed edge, and can be seen as a quotient graph of a symmetric graph together with the information needed to recover the covering graph (see \Cref{sec: GG}). A $\Gamma$-gain graph is $\Gamma$-isostatic in the hyperbolic plane if, roughly speaking, its lift admits a realisation as a $\Gamma$-symmetric framework with no symmetry-preserving infinitesimal flex and no redundant bars (see \Cref{sec: comb. rig.}). 

Our analysis relies on the intrinsic geometry of the hyperbolic plane, in particular the behaviour of geodesics and the hyperbolic distance. 
We are able to reduce infinitesimal rigidity questions for infinite $\Gamma$-symmetric frameworks in $\mathbb{H}$ to finite combinatorial conditions. Our main result gives a combinatorial characterisation of isostatic $\Gamma$-gain graphs. 

\begin{theorem}\label{thm:main}
For a surface group $\Gamma$ and a $\Gamma$-gain graph $(G,\psi)$, the following are equivalent:
\begin{enumerate}
\item $(G,\psi)$ is $\Gamma$-isostatic in the hyperbolic plane $\mathbb{H}$;
\item $(G,\psi)$ is $(2,3,1,0)$-gain tight.
\end{enumerate}
\label{main_theorem}
\end{theorem}

The formal definitions of gain graphs and $(2,3, 1,0)$-gain tight graphs will be given in  \Cref{sec:symgraphs}. We additionally prove a similar result for the case when $\Gamma$ is a cyclic group, which generalises a result in \cite{Schulze2012}. 

Since every compact orientable surface of genus $g\geq 2$ is realised as a quotient $\mathbb{H}/\Gamma$, \Cref{main_theorem} yields a combinatorial characterisation of rigidity for graphs on such surfaces, 
viewed through their symmetric lifts to the universal cover.
Previous work on intrinsic rigidity of frameworks on surfaces has focused primarily on
surfaces with non-trivial continuous isometry groups, such as the sphere and the torus. 
Whiteley \cite{doi:10.1137/0401025} considered graphs on quotients of the Euclidean plane under a group, where the bars of the framework are geodesics on the surface. Ross \cite{MR3265750,ross2015inductive} studied and characterised the infinitesimal rigidity of frameworks on tori in this way. The same result is implicit in Malestein and Theran's analysis of periodic frameworks \cite{MR2995673}.
In contrast, infinitesimal rigidity for surfaces of genus $\geq 2$ has not been studied in this context. We note also prior work on extrinsic surface rigidity \cite{NOP1,NOP2} which also focussed on surfaces with non-trivial continuous isometry groups but differs from the intrinsic view taken in this paper.

The proof of \Cref{main_theorem} is an inductive proof based on the characterisation of forced-symmetric rigidity in the plane with respect to dihedral groups of order $2n$ for odd $n$ by Jordán, Kaszanitzky, and Tanigawa \cite{jzt2016}. In \cite{jzt2016}, an inductive characterisation is given for the $(2, 3, 1, 0)$-gain tight graphs, which still holds for the graphs symmetric with respect to a Fuchsian group. They show that each $(2, 3, 1, 0)$-gain tight graph can be built from a finite family of base graphs using Henneberg-type operations. Given this combinatorial characterisation, we need to show that the graphs from the finite family of base graphs are all $\Gamma$-isostatic, and that applying an operation to an isostatic graph yields an isostatic graph. We formulate most of the arguments for general Fuchsian groups so that the results could potentially be generalised in the future.

The paper is organised as follows. In \Cref{sec:symgraphs}, we introduce gain graphs and their combinatorial properties. In \Cref{sec:hypgeom}, we review the necessary background on hyperbolic geometry and Fuchsian groups. \Cref{sec:hypsymrig} defines the basic concepts of hyperbolic symmetric rigidity and develops the orbit rigidity matrix. In \Cref{sec:combchar} we prove the characterisation for cyclic Fuchsian groups, where we will already prove some of the combinatorial moves preserve isostaticity. In \Cref{sec: irr.}, we show that irreducible tight graphs are $\Gamma$-isostatic and in \Cref{sec:2ext} we prove that the remaining Henneberg moves, the $2$-extension and the loop $2$-extension, preserve isostaticity. The results in \Cref{sec: irr.,sec:2ext} will conclude the proof of our main result, \Cref{thm:main}. Finally, in \Cref{sec:riemann}, we interpret our results in the setting of compact Riemann surfaces.

\section{Symmetric graphs}\label{sec:symgraphs}
The set of all automorphisms of a graph $\tilde{G}$ forms a group under composition, known as the \textit{automorphism group} of $\tilde{G}$, and denoted $\text{Aut}(\tilde{G})$. Let $\Gamma$ be a group. We say $\tilde{G}$ is \textit{$\Gamma$-symmetric} if there is a homomorphism $\theta:\Gamma\rightarrow\text{Aut}(\tilde{G})$ with finitely many vertex orbits and edge orbits. We call $\theta$ the \textit{action} of $\Gamma$ on $\tilde{G}$. We will assume, throughout the paper, that the action $\theta$ is \textit{free} on $V(\tilde{G})$, i.e.~that $\theta(g)v\neq v$ for all non-identity $g\in\Gamma$ and all $v\in V(\tilde{G})$. When the action $\theta$ is clear from context, we drop it from the notation, i.e.~we write $g v$ instead of $\theta(g)v$ for $g\in\Gamma,v\in V(\tilde{G})$. 

\subsection{Gain graphs}
\label{sec: GG}

Gain graphs are labelled multigraphs introduced in \cite{Gross} and further developed in \cite{ZASLAVSKY198932}. They are often used as a technical tool to concisely represent symmetric graphs, and were first introduced in rigidity theory in \cite{jzt2016}. In this subsection, we define gain graphs and show how gain graphs correspond to symmetric graphs, since this construction will be used throughout. First, we define a gain graph.

Throughout the paper we will use id to denote the identity of a group. We will be consistent with this even when we are concretely analysing matrices in later sections.

\begin{definition}
\label{def:GG}
Let $\Gamma$ be a group. A \emph{$\Gamma$-gain graph} $(G, \psi)$ consists of a directed multigraph $G$ paired with a function $\psi: E(G)\rightarrow \Gamma$ such that the following are satisfied:
\begin{itemize}
    \item[(i)] for all loops $e\in E(G)$, $\psi(e)\neq\text{id};$
    \item[(ii)] for all parallel edges $e,f\in E(G)$ with the same orientation, $\psi(e)\neq\psi(f)$; and
    \item[(iii)] for all parallel edges $e,f\in E(G)$ with the opposite orientation, $\psi(e)\neq\psi(f)^{-1}$.
\end{itemize}
The elements $\psi(e)$ are called the \textit{gains} of the edges, and $\psi$ is called the \textit{gain map} of $(G,\psi)$. We will denote a directed edge $e=u \rightarrow v$, with $\psi(e)= g$ using the notation  $u\xrightarrow{g}v$. 
\end{definition}

Let $\tilde{G}$ be a $\Gamma$-symmetric graph with respect to a group action $\theta$ and $G$ be the $\theta(\Gamma)$-quotient of $\tilde{G}$. To obtain a $\Gamma$-gain graph $(G,\psi)$ from $G$, we fix an orientation on the edges of $G$, and for each $v=\theta(\Gamma)v^{\star}\in V(G)$, choose a vertex representative $v^{\star}\in V(\tilde{G})$. For each $e=u\rightarrow v\in E(G)$, there is a unique $g\in\Gamma$ such that $u^{\star}v^{\star}_{\theta(g)}\in E(\tilde{G})$, where $v^{\star}_{\theta(g)}:=\theta(g)v^{\star}$. We let $\psi(e)=g$.

In this process, each edge $e=u\xrightarrow{g}v\in E(G)$ can be substituted with $v\xrightarrow{g^{-1}}u\in E(G)$. The choice of vertex orbit representatives can also lead to different gains on the edges of the gain graph. Let $u\xrightarrow{g}v\in E(G)$, and suppose we had chosen different representatives $u'=\theta(g_1)u^{\star}$ and $v'=\theta(g_2)v^{\star}$ ($u^{\star}$ may be equal to $v^{\star}$). Then we see that 
\begin{align*}
u^{\star}v^{\star}_{\theta(g)}\in E(\tilde{G}) \quad 
\iff 
\quad
u^{\star}_{\theta(g_1)}v^{\star}_{\theta(g_1 gg_{2}^{-1}g_2)} \in E(\tilde{G})
\quad 
\iff 
\quad
 u'v'_{\theta(g_1gg_{2}^{-1})} \in E(\tilde{G}),
\end{align*} 
where $w_{\theta(h)}:=\theta(h)w$ for all $h\in\Gamma,w\in V(\tilde{G})$. The change of choice for vertex orbit representatives is captured by a gain graph operation known as switching.

\begin{definition}\label{def:gainswitch}
    Let $(G,\psi)$ be a $\Gamma$-gain graph, $v\in V(G)$ and $g\in\Gamma$. A \emph{switching} at $v$ with gain $g$ is an operation which generates a new gain $\psi'$ given by
\begin{equation*}
    \psi'(e)=\begin{cases}
        g\psi(e)g^{-1} & \text{if }e\text{ is a loop incident to }v\\
        g\psi(e) & \text{if }e\text{ is a non-loop edge directed from }v\\
        \psi(e)g^{-1} & \text{if }e\text{ is a non-loop edge directed to }v\\
        \psi(e) & \text{otherwise}
    \end{cases}
\end{equation*}
for all $e\in E(G)$. Two $\Gamma$-gain graphs $(G,\psi),(G,\psi')$ are \emph{equivalent} if one can be obtained from the other by a sequence of switchings. 
\end{definition}

Up to equivalence, and up to the orientation on the edges, this well-known process generates a unique $\Gamma$-gain graph for each $\Gamma$-symmetric graph. Conversely, for any gain graph $(G, \psi)$, we can define a graph $\tilde{G}$, called the \textit{$\Gamma$-lift} of $(G,\psi)$. It can be defined by letting
 \begin{align*}
V(\tilde{G}) &= (\Gamma\times V(G)) \mbox{ and}\\
E(\tilde{G}) &= \{(g,u)(gh,v)  ~|~  u\xrightarrow{h}v\in E(G),g\in\Gamma\}.
 \end{align*}
 The action on $\tilde{G}$ is defined by letting $g(h,u)=(gh,u)$ for all $g\in\Gamma,(h,u)\in V(\tilde{G})$. 

 \subsection{Sparsity conditions of gain graphs}
 \label{sec: sparsity}
Next we introduce combinatorial conditions on gain graphs in terms of sparsity counts. These will be used in later sections to describe infinitesimally rigid symmetric frameworks in the hyperbolic plane. We start with the following definition.

 \begin{definition}
Let $\Gamma$ be a group and let $(G,\psi)$ be a $\Gamma$-gain graph. Let $W=e_1,e_2,\dots,e_n$ be a walk in $G$, where $e_i$ has end-vertices $v_i,v_{i+1}\in V(G)$ for all $1\leq i\leq n$. We define the \emph{gain} of $W$ as 
\begin{equation}
 \psi(W)=\psi(e_1)^{\text{sign}(e_1)} \cdots \psi(e_n)^{\text{sign}(e_n)},
 \end{equation}
 where $\text{sign}(e_i)=1$ if $e_i=v_i\xrightarrow{\psi(e_i)} v_{i+1}$ and $\text{sign}(e_i)=-1$ if $e_i=v_{i+1}\xrightarrow{\psi(e_i)} v_i$. Suppose $G$ is connected. We define the \emph{gain of} $G$ with \textit{base vertex} $v\in V(G)$ and \textit{gain map} $\psi$ to be $$\langle G\rangle_{\psi,v}=\langle \psi(W) \in \Gamma~|~W \text{ is a closed walk in $G$ starting at }v \rangle.$$
 \end{definition}

 \begin{proposition}[{\cite[Corollary 1 of Theorem 2.5.1]{gross2001topological},\cite[Proposition 2.1]{jzt2016}}]
    \label{prop:<H>_u=<H>_v}
    For a group $\Gamma$, let $(G,\psi),(G,\psi')$ be two equivalent connected $\Gamma$-gain graphs. For all $v\in V(G)$, $\left<G\right>_{\psi,v}$ and $\left<G\right>_{\psi',v}$ are conjugate. For all $u,v\in V(G)$,  $\left<G\right>_{\psi,u}$ and $\left<G\right>_{\psi,v}$ are conjugate.
\end{proposition}

Let $\Gamma$ be a group and $(G,\psi)$ be a connected $\Gamma$-gain graph. \Cref{prop:<H>_u=<H>_v} implies that $\left<G\right>_{\psi,v}$ is cyclic (respectively, the trivial group) for some $v\in V(G)$ if and only if it is cyclic (respectively, the trivial group) for all $v\in V(G)$. Similarly, given $v\in V(G)$, $\left<G\right>_{\psi,v}$ is cyclic (respectively, the trivial group) if and only if $\left<G\right>_{\psi',v}$ is cyclic (respectively, the trivial group) for all $(G,\psi')$ equivalent to $(G,\psi)$. Therefore, we define the following.

\begin{definition}
    Let $\Gamma$ be a group, $(G,\psi)$ be a connected $\Gamma$-gain graph, and $v\in V(G)$. We say $(G,\psi)$ is \textit{cyclic} if $\left<G\right>_{\psi,v}$ is cyclic. We say $(G,\psi)$ is \textit{balanced} if $\left<G\right>_{\psi,v}=\{\text{id}\}$. We say a connected subgraph $H$ of $G$ is \textit{cyclic} (respectively, \textit{balanced}) if $(H,\psi|_{E(H)})$ is cyclic (respectively, balanced).
\end{definition}

The following two results show that one can choose the gains in a way that is often convenient for computations.

\begin{lemma}[{\cite[Lemma 5.3]{ZASLAVSKY198932}}]
\label{lem:treegain}
    Let $\Gamma$ be a group and $(G,\psi)$ be a connected $\Gamma$-gain graph. Let $T\subseteq E(G)$ be a spanning tree of $G$. Then there is always a gain map $\psi'$ equivalent to $\psi$ such that $\psi'(e)=\text{id}$ for all $e\in E(T)$. 
\end{lemma}

\begin{proposition}[{\cite[Proposition 2.3]{jzt2016}}]
\label{prop:balgain}
    Let $\Gamma$ be a group and $(G,\psi)$ be a connected $\Gamma$-gain graph. Let $T\subseteq E(G)$ be a spanning tree of $G$, and assume that $\psi(e)=\text{id}$ for all $e\in E(T)$. Then, for all $v\in V(G)$, $\left<G\right>_{\psi,v}=\left<\psi(e):e\in E(G-T)\right>$.
\end{proposition}

Let $\Gamma$ be a group and $(G,\psi)$ be a $\Gamma$-graph graph with non-empty edge set. For $\ell\in\{0,1,3\}$ we say $G$ (equivalently, $(G,\psi)$) is \textit{$(2,\ell)$-sparse} if, for all subgraphs $H\subseteq G$ with $E(H)\neq\emptyset$, $|E(H)|\leq2|V(H)|-\ell$.
If $G$ is $(2,\ell)$-sparse and $|E(G)|=2|V(G)|-\ell$, then we say that $G$ is \emph{$(2,\ell)$-tight}.
There are two families of graphs we are interested in throughout:

\begin{definition}
    We define a $\Gamma$-gain graph $(G,\psi)$ to be \emph{$(2,3,1)$-gain sparse} if it is $(2,1)$-sparse and all balanced subgraphs of $G$ with non-empty edge set are $(2,3)$-sparse.
    We say $(G,\psi)$ is \emph{$(2,3,1)$-gain tight} if it is \emph{$(2,3,1)$-gain sparse} and, in addition, $|E(G)|=2|V(G)|-1$.
\end{definition}

\begin{definition}
    We define a $\Gamma$-gain graph $(G,\psi)$ to be \emph{$(2,3,1,0)$-gain sparse} if it is $(2,0)$-sparse, all cyclic subgraphs of $G$ with non-empty edge set are $(2,1)$-sparse, and all balanced subgraphs of $G$ with non-empty edge set are $(2,3)$-sparse. 
    We say $(G,\psi)$ is \emph{$(2,3,1,0)$-gain tight} if it is \emph{$(2,3,1,0)$-gain sparse} and, in addition,  $|E(G)|=2|V(G)|$.
\end{definition}

 \subsection{Extension and reduction operations}\label{sec:ext}
 Now we introduce extension and reduction operations which we will use throughout.
\begin{definition}
\label{def:0ext}
Let $\Gamma$ be a group. A $\Gamma$-gain graph $(G',\psi')$ is obtained from another $\Gamma$-gain graph $(G,\psi)$ by a \emph{0-extension} if it is obtained by adding a vertex with two incident edges. The new edges are labelled such that, if they are parallel, they satisfy conditions (ii) and (iii) in \Cref{def:GG}. The inverse operation is called a \emph{$0$-reduction} (see \Cref{fig:0ext}).
\end{definition}

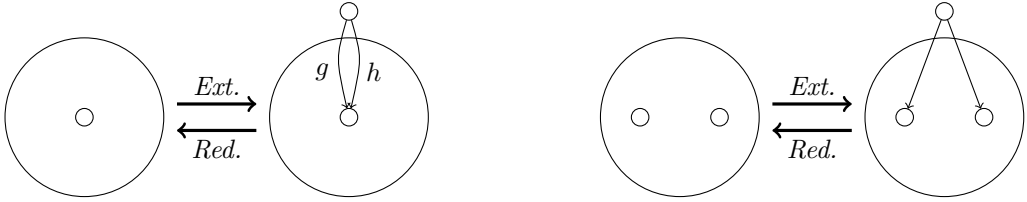
\begin{figure}[ht]
    \centering
    \begin{subfigure}[scale=0.5]{0.47\textwidth}
        \begin{tikzpicture}
  [scale=.7,auto=left]
  \node[draw, circle, scale=0.7](u) at (2.75,0){};
  \draw (2.75,0) circle (1.5cm);

  \draw[->, very thick] (4.5,0.25) -- (6,0.25);
  \node[above] at (5.25, 0.25) {\emph{Ext.}};
  \draw[->, very thick] (6,-0.25) -- (4.5,-0.25);
  \node[below] at (5.25, -0.25) {\emph{Red.}};

  \node[draw, circle, scale=0.7](u) at (7.75,0){};
  \node[draw, circle, scale=0.7](v) at (7.75,2){};  
  \draw (7.75,0) circle (1.5cm);
  \draw[->] (v) .. controls (7.5,1.115) .. (u);
  \draw[->] (v) .. controls (8,1.115) .. (u);
    \node[left] at (7.55,0.85) {$g$};
    \node[right] at (7.9,0.85) {$h$};
\end{tikzpicture}
    \end{subfigure}
    \begin{subfigure}[scale=0.5]{0.45\textwidth}
        \begin{tikzpicture}
  [scale=.7,auto=left]
  \draw (2.75,0) circle (1.5cm);
  \node[draw, circle, scale=0.7](u1) at (2,0){};
  \node[draw, circle, scale=0.7](u2) at (3.5,0){};

  \draw[->, very thick] (4.5,0.25) -- (6,0.25);
  \node[above] at (5.25, 0.25) {\emph{Ext.}};
  \draw[->, very thick] (6,-0.25) -- (4.5,-0.25);
  \node[below] at (5.25, -0.25) {\emph{Red.}};

  \node[draw, circle, scale=0.7](u1) at (7,0){};
  \node[draw, circle, scale=0.7](u2) at (8.5,0){};
  \node[draw, circle, scale=0.7](v) at (7.75,2){};
  \draw (7.75,0) circle (1.5cm);
  \draw[->] (v) -- (u1);
  \draw[->] (v) -- (u2);
\end{tikzpicture}
    \end{subfigure}
    \caption{Examples of a 0-extension. Here, $g$ and $h$ cannot coincide.} 
    \label{fig:0ext}
\end{figure}

\begin{definition}
\label{def:loop1ext}
Let $\Gamma$ be a group. A $\Gamma$-gain graph $(G',\psi')$ is obtained from another $\Gamma$-gain graph $(G,\psi)$ by a \emph{loop-1-extension} if it is obtained by adding a looped vertex of degree 3\footnote{Throughout the paper, we assume that loops add two to the degrees of vertices.}. The new edges are labelled freely, provided the loop has non-identity gain. The inverse operation is called a \emph{loop-1-reduction} (see \Cref{fig:loop1ext}).
\end{definition}

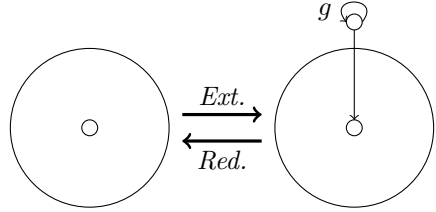
\begin{figure}[ht]
    \centering
    \begin{tikzpicture}
  [scale=.7,auto=left]
  \draw (2.75,0) circle (1.5cm);
  \draw(2.75,0) circle (0.15cm);

  \draw[->, very thick] (4.5,0.25) -- (6,0.25);
  \node[above] at (5.25, 0.25) {\emph{Ext.}};
  \draw[->, very thick] (6,-0.25) -- (4.5,-0.25);
  \node[below] at (5.25, -0.25) {\emph{Red.}};

  \draw(7.75,0) circle (0.15cm);  
  \draw (7.75,0) circle (1.5cm);
  \draw (7.75,2) circle (0.15cm);
  \draw[->] (7.75, 1.85) -- (7.75, 0.15);
  \draw[->] (7.9,2) .. controls (8.3,2.5) and (7.2,2.5) .. (7.6,2); 
  \node[left] at (7.5,2.2) {$g$};
\end{tikzpicture}
    \caption{Example of a loop-1-extension. Here, $g$ cannot be $\text{id}$.}
    \label{fig:loop1ext}
\end{figure}

\begin{definition}
\label{def:1ext}
Let $\Gamma$ be a group. A $\Gamma$-gain graph $(G',\psi')$ is obtained from another $\Gamma$-gain graph $(G,\psi)$ by a \emph{1-extension} if it is obtained by choosing a vertex $u\in V(G)$ and an edge $e=v_1\rightarrow v_2\in E(G)$, removing $e$ and adding a new vertex $v$ to $V(G),$ together with three edges $e_1=v\rightarrow v_1,e_2=v\rightarrow v_2,e_3=v\rightarrow u$. The edges $e_1,e_2$ are labelled such that $\psi'(e_1)^{-1}\psi'(e_2)=\psi(e)$ and every two-cycle $e_ie_j^{-1},$ if it exists, is unbalanced. The inverse operation is called a \emph{$1$-reduction} (see \Cref{fig:1ext}).
\end{definition}

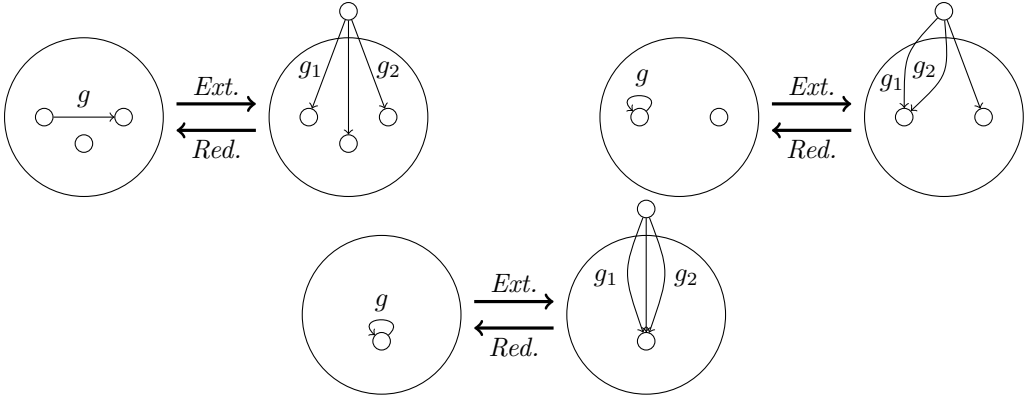
\begin{figure}[ht]
    \centering
    \begin{subfigure}[scale=0.5]{0.47\textwidth}
        \begin{tikzpicture}
  [scale=.7,auto=left]
  \node[draw, circle, scale=0.7] (u1) at (2,0){};
  \node[draw, circle, scale=0.7] (u2) at (3.5,0){};
  \node[draw, circle, scale=0.7] (u3) at (2.75,-0.5){};
  \draw (2.75,0) circle (1.5cm);
  \draw[->] (u1) -- (u2);
  \node[above] at (2.75, 0.01){$g$};

  \draw[->, very thick] (4.5,0.25) -- (6,0.25);
  \node[above] at (5.25, 0.25) {\emph{Ext.}};
  \draw[->, very thick] (6,-0.25) -- (4.5,-0.25);
  \node[below] at (5.25, -0.25) {\emph{Red.}};

  \node[draw, circle, scale=0.7] (u1) at (7,0){};
  \node[draw, circle, scale=0.7] (u2) at (8.5,0){};
  \node[draw, circle, scale=0.7] (u3) at (7.75,-0.5){};
  \node[draw, circle, scale=0.7] (v) at (7.75,2){};
  \draw (7.75,0) circle (1.5cm);
  \draw[->] (v) -- (u1);
  \draw[->] (v) -- (u2);
  \draw[->] (v) -- (u3);
  \node[left] at (7.45,0.9) {$g_1$};
  \node[right] at (8.1,0.9) {$g_2$};
\end{tikzpicture}
    \end{subfigure}
    \begin{subfigure}[scale=0.5]{0.47\textwidth}
        \begin{tikzpicture}
  [scale=.7,auto=left]
  \node[draw, circle, scale=0.7] (u1) at (2,0){};
  \node[draw, circle, scale=0.7] (u2) at (3.5,0){};
  \draw (2.75,0) circle (1.5cm);
  \draw[->] (u1) .. controls (2.55,0.5) and (1.45,0.5) .. (u1);
  \node[above] at (2.05, 0.4){$g$};

  \draw[->, very thick] (4.5,0.25) -- (6,0.25);
  \node[above] at (5.25, 0.25) {\emph{Ext.}};
  \draw[->, very thick] (6,-0.25) -- (4.5,-0.25);
  \node[below] at (5.25, -0.25) {\emph{Red.}};

  \node[draw, circle, scale=0.7] (u1) at (7,0){};
  \node[draw, circle, scale=0.7] (u2) at (8.5,0){};
  \node[draw, circle, scale=0.7] (u3) at (7.75,2){};  
  \draw (7.75,0) circle (1.5cm);
  \draw[->] (v) -- (u2);
  \draw[->] (v) .. controls (7.85,0.8) .. (u1);
  \draw[->] (v) .. controls (7,1.2) .. (u1);
  \node[left] at (7.18,0.65) {$g_1$};
  \node[left] at (7.8,0.9) {$g_2$};
\end{tikzpicture}
    \end{subfigure}
    \begin{subfigure}[scale=0.5]{0.47\textwidth}
        \begin{tikzpicture}
  [scale=.7,auto=left]
  \node[draw, circle, scale=0.7] (u) at (2.75,-0.5){};
  \draw (2.75,0) circle (1.5cm);
  \draw[->] (u) .. controls (3.3,0) and (2.2,0) .. (u);
  \node[above] at (2.75, -0.1){$g$};

  \draw[->, very thick] (4.5,0.25) -- (6,0.25);
  \node[above] at (5.25, 0.25) {\emph{Ext.}};
  \draw[->, very thick] (6,-0.25) -- (4.5,-0.25);
  \node[below] at (5.25, -0.25) {\emph{Red.}};
  
  \node[draw, circle, scale=0.7] (u) at (7.75,-0.5){};
  \node[draw, circle, scale=0.7] (v) at (7.75,2){};
  \draw (7.75,0) circle (1.5cm);
  \draw[->] (v) -- (u);
  \draw[->] (v) .. controls (8.2,0.8) .. (u);
  \draw[->] (v) .. controls (7.3,0.8) .. (u);
  \node[left] at (7.4,0.7) {$g_1$};
  \node[right] at (8.1,0.7) {$g_2$};
\end{tikzpicture}
    \end{subfigure}
    \caption{Examples of 1-extensions. Here, $g=g_1^{-1}g_2$.}
    \label{fig:1ext}
\end{figure}

\begin{definition}
    \label{def:loop2ext}
    Let $\Gamma$ be a group. A $\Gamma$-gain graph $(G',\psi')$ is obtained from another $\Gamma$-gain graph $(G,\psi)$ by a \emph{loop-2-extension} if it is obtained by choosing an edge $e=v_1\rightarrow v_2\in E(G)$, removing $e$ and adding a looped vertex $v$ to $V(G)$, together with the edges $e_1=v\rightarrow v_1,e_2=v\rightarrow v_2$. The loop is labelled with any non-identity gain, and $e_1,e_2$ are labelled in a way such that $\psi'(e_1)^{-1}\psi'(e_2)=\psi(e)$. The inverse operation is called a \textit{loop-2-reduction} (see \Cref{fig:loop2ext}).
\end{definition}

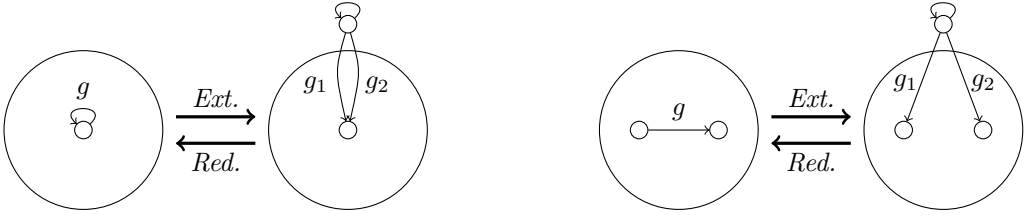
\begin{figure}[ht]
    \centering
    \begin{subfigure}[scale=0.5]{0.47\textwidth}
        \begin{tikzpicture}
  [scale=.7,auto=left]
  \draw (2.75,0) circle (1.5cm);
  \node[draw, circle, scale=0.7] (u) at (2.75,0){};
  \draw[->] (u) .. controls (3.3,0.5) and (2.2,0.5) .. (u);
  \node[above] at (2.75,0.4){$g$};

  \draw[->, very thick] (4.5,0.25) -- (6,0.25);
  \node[above] at (5.25, 0.25) {\emph{Ext.}};
  \draw[->, very thick] (6,-0.25) -- (4.5,-0.25);
  \node[below] at (5.25, -0.25) {\emph{Red.}};

  \node[draw, circle, scale=0.7] (u) at (7.75,0){};
  \node[draw, circle, scale=0.7] (v) at (7.75,2){}; 
  \draw (7.75,0) circle (1.5cm);
  \draw[->] (v) .. controls (7.5,1.115) .. (u);
  \draw[->] (v) .. controls (8,1.115) .. (u);
  \draw[->] (v) .. controls (8.3,2.5) and (7.2,2.5) .. (v);
    \node[left] at (7.55,0.85) {$g_1$};
    \node[right] at (7.9,0.85) {$g_2$};
\end{tikzpicture}
    \end{subfigure}
    \begin{subfigure}[scale=0.5]{0.45\textwidth}
        \begin{tikzpicture}
  [scale=.7,auto=left]
  \node[draw, circle, scale=0.7] (u1) at (2,0){};
  \node[draw, circle, scale=0.7] (u2) at (3.5,0){};
  \draw (2.75,0) circle (1.5cm);
  \draw[->] (u1) -- (u2);
  \node[above] at (2.75,0){$g$};

  \draw[->, very thick] (4.5,0.25) -- (6,0.25);
  \node[above] at (5.25, 0.25) {\emph{Ext.}};
  \draw[->, very thick] (6,-0.25) -- (4.5,-0.25);
  \node[below] at (5.25, -0.25) {\emph{Red.}};

  \node[draw, circle, scale=0.7] (u1) at (7,0){};
  \node[draw, circle, scale=0.7] (u2) at (8.5,0){};
  \node[draw, circle, scale=0.7] (v) at (7.75,2){};
  \draw (7.75,0) circle (1.5cm);
  \draw[->] (v) .. controls (8.3,2.5) and (7.2,2.5) .. (v);
  \draw[->] (v) -- (u1);
  \draw[->] (v) -- (u2);
  \node[left] at (7.45,0.9) {$g_1$};
  \node[right] at (8.1,0.9) {$g_2$};
\end{tikzpicture}
    \end{subfigure}
    \caption{Examples of a loop-2-extension. Here, $g=g_1^{-1}g_2$. } 
    \label{fig:loop2ext}
\end{figure}

\begin{definition}
    \label{def:2ext}
    Let $\Gamma$ be a group. A $\Gamma$-gain graph $(G',\psi')$ is obtained from another $\Gamma$-gain graph $(G,\psi)$ by a \emph{2-extension} if it is obtained by choosing two edges $f_1=v_1\rightarrow v_2,f_2=v_3\rightarrow v_4\in E(G)$, removing $f_1,f_2$ and adding a vertex $v$ to $V(G)$, together with the edges $e_1=v\rightarrow v_1,e_2=v\rightarrow v_2,e_3=v\rightarrow v_3,e_4=v\rightarrow v_4$. The edges $e_1,e_2,e_3,e_4$ are labelled in a way such that $\psi'(e_1)^{-1}\psi'(e_2)=\psi(f_1)$, $\psi'(e_3)^{-1}\psi'(e_4)=\psi(f_2)$ and $G[\{v_1,v_2,v_3,v_4\}]$ is $(2,3,1,0)$-sparse. The inverse operation is called a \textit{2-reduction} (see \Cref{fig:2ext}).
\end{definition}

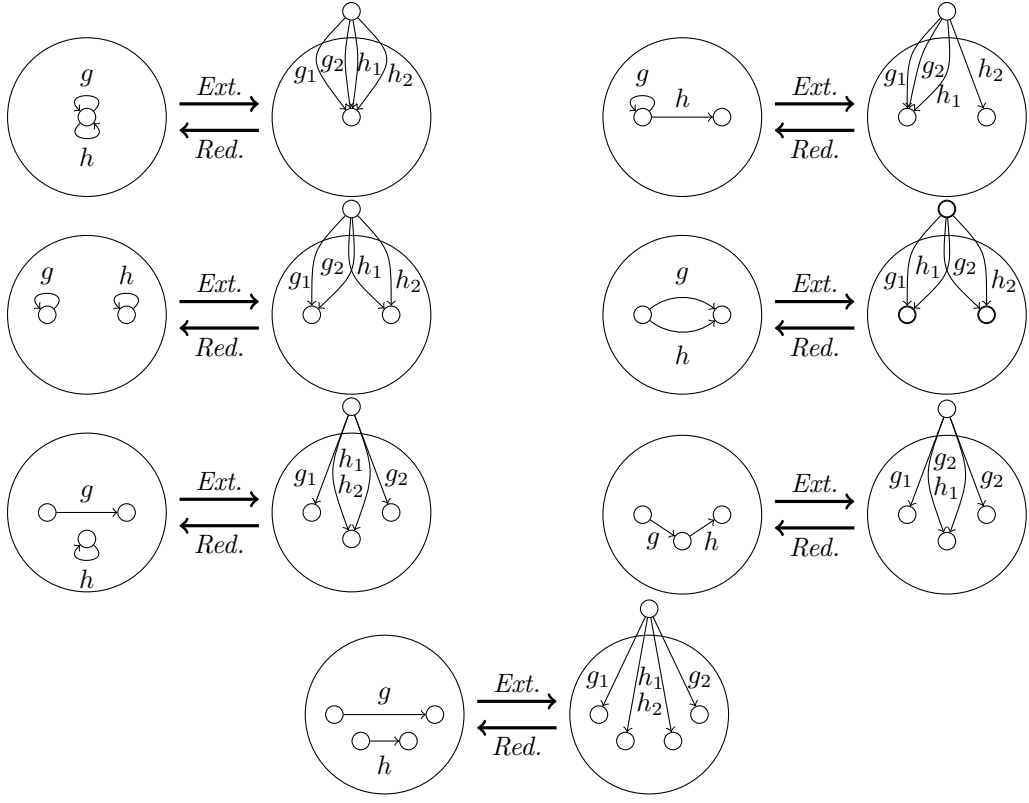
\begin{figure}[ht]
    \centering
    \begin{subfigure}[scale=0.5]{0.47\textwidth}
        \begin{tikzpicture}
  [scale=.7,auto=left]
  \node[draw, circle, scale=0.7] (u) at (2.75,0){};
  \draw (2.75,0) circle (1.5cm);
  \draw[->] (u) .. controls (3.3,0.5) and (2.2,0.5) .. (u);
  \draw[->] (u) .. controls (2.2,-0.5) and (3.3,-0.5) .. (u);
  \node[above] at (2.75,0.4) {$g$};
  \node[below] at (2.75,-0.4) {$h$};

  \draw[->, very thick] (4.5,0.25) -- (6,0.25);
  \node[above] at (5.25, 0.25) {\emph{Ext.}};
  \draw[->, very thick] (6,-0.25) -- (4.5,-0.25);
  \node[below] at (5.25, -0.25) {\emph{Red.}};

  \node[draw, circle, scale=0.7] (u) at (7.75,0){};
  \node[draw, circle, scale=0.7] (v) at (7.75,2){}; 
  \draw (7.75,0) circle (1.5cm);
  \draw[->] (v) .. controls (6.9,1.115) .. (u);
  \draw[->] (v) .. controls (8.6,1.115) .. (u);
  \draw[->] (v) .. controls (7.6,1.115) .. (u);
  \draw[->] (v) .. controls (7.9,1.115) .. (u);
  \node[left] at (7.3,0.8) {$g_1$};
  \node at (7.4,1) {$g_2$};
  \node at (8.1,1) {$h_1$};
  \node[right] at (8.25,0.8) {$h_2$};
\end{tikzpicture}
    \end{subfigure}
    \begin{subfigure}[scale=0.5]{0.47\textwidth}
        \begin{tikzpicture}
  [scale=.7,auto=left]
  \node[draw, circle, scale=0.7] (u1) at (2,0){};
  \node[draw, circle, scale=0.7] (u2) at (3.5,0){};
  \draw (2.75,0) circle (1.5cm);
  \draw[->] (u1) .. controls (2.55,0.5) and (1.45,0.5) .. (u1);
  \draw[->] (u1) -- (u2);
  \node[above] at (2,0.4) {$g$};
  \node[above] at (2.75,0.01) {$h$};

  \draw[->, very thick] (4.5,0.25) -- (6,0.25);
  \node[above] at (5.25, 0.25) {\emph{Ext.}};
  \draw[->, very thick] (6,-0.25) -- (4.5,-0.25);
  \node[below] at (5.25, -0.25) {\emph{Red.}};

  \node[draw, circle, scale=0.7] (u1) at (7,0){};
  \node[draw, circle, scale=0.7] (u2) at (8.5,0){};
  \node[draw, circle, scale=0.7] (v) at (7.75,2){};
  \draw (7.75,0) circle (1.5cm);
  \draw[->] (v) .. controls (7.25,1.2) .. (u1);
  \draw[->] (v) -- (u2);
  \draw[->] (v) .. controls (7.85,0.8) .. (u1);
  \draw[->] (v) .. controls (7,1.2) .. (u1);
  \node[left] at (7.2,0.75) {$g_1$};
  \node at (7.5,0.9) {$g_2$};
  \node at (7.8,0.4) {$h_1$};
  \node[right] at (8.15,0.9) {$h_2$};
\end{tikzpicture}
    \end{subfigure}
    \begin{subfigure}[scale=0.5]{0.47\textwidth}
        \begin{tikzpicture}
  [scale=.7,auto=left]
  \node[draw, circle, scale=0.7] (u1) at (2,0){};
  \node[draw, circle, scale=0.7] (u2) at (3.5,0){};
  \draw (2.75,0) circle (1.5cm);
  \draw[->] (u1) .. controls (2.55,0.5) and (1.45,0.5) .. (u1);
  \draw[->] (u2) .. controls (4.05,0.5) and (2.95,0.5) .. (u2);
  \node[above] at (2,0.4) {$g$};
  \node[above] at (3.5,0.4) {$h$};

  \draw[->, very thick] (4.5,0.25) -- (6,0.25);
  \node[above] at (5.25, 0.25) {\emph{Ext.}};
  \draw[->, very thick] (6,-0.25) -- (4.5,-0.25);
  \node[below] at (5.25, -0.25) {\emph{Red.}};

  \node[draw, circle, scale=0.7] (u1) at (7,0){};
  \node[draw, circle, scale=0.7] (u2) at (8.5,0){};
  \node[draw, circle, scale=0.7] (v) at (7.75,2){};
  \draw (7.75,0) circle (1.5cm);
  \draw[->] (v) .. controls (7.85,0.8) .. (u1);
  \draw[->] (v) .. controls (7,1.2) .. (u1);
  \draw[->] (v) .. controls (7.65,0.8) .. (u2);
  \draw[->] (v) .. controls (8.5,1.2) .. (u2);
  \node[left] at (7.2,0.65) {$g_1$};
  \node at (7.4,0.9) {$g_2$};
  \node at (8.1,0.9) {$h_1$};
  \node[right] at (8.4,0.65) {$h_2$};
\end{tikzpicture}
    \end{subfigure}
    \begin{subfigure}[scale=0.5]{0.47\textwidth}
        \begin{tikzpicture}
  [scale=.7,auto=left]
  \node[draw, circle, scale=0.7] (u1) at (2,0){};
  \node[draw, circle, scale=0.7] (u2) at (3.5,0){};
  \draw (2.75,0) circle (1.5cm);
  \draw[->] (u1) .. controls (2.5,0.4) and (3,0.4) .. (u2);
  \draw[->] (u1) .. controls (2.5,-0.4) and (3,-0.4) .. (u2);
  \node[above] at (2.75,0.4) {$g$};
  \node[below] at (2.75,-0.4) {$h$};

  \draw[->, very thick] (4.5,0.25) -- (6,0.25);
  \node[above] at (5.25, 0.25) {\emph{Ext.}};
  \draw[->, very thick] (6,-0.25) -- (4.5,-0.25);
  \node[below] at (5.25, -0.25) {\emph{Red.}};

  \node[draw, circle, scale=0.7] (u1) at (7,0){};
  \node[draw, circle, scale=0.7] (u2) at (8.5,0){};
  \node[draw, circle, scale=0.7] (v) at (7.75,2){};
  \draw(7,0) circle (0.15cm);  
  \draw (7.75,0) circle (1.5cm);
  \draw(8.5,0) circle (0.15cm);
  \draw (7.75, 2) circle (0.15cm);
  \draw[->] (v) .. controls (7.85,0.8) .. (u1);
  \draw[->] (v) .. controls (7,1.2) .. (u1);
  \draw[->] (v) .. controls (7.65,0.8) .. (u2);
  \draw[->] (v) .. controls (8.5,1.2) .. (u2);
  \node[left] at (7.2,0.65) {$g_1$};
  \node at (7.4,0.9) {$h_1$};
  \node at (8.1,0.9) {$g_2$};
  \node[right] at (8.4,0.65) {$h_2$};
\end{tikzpicture}
    \end{subfigure}
    \begin{subfigure}[scale=0.5]{0.47\textwidth}
        \begin{tikzpicture}
  [scale=.7,auto=left]
  \node[draw, circle, scale=0.7] (u1) at (2,0){};
  \node[draw, circle, scale=0.7] (u2) at (3.5,0){};
  \node[draw, circle, scale=0.7] (u3) at (2.75,-0.5){};
  \draw (2.75,0) circle (1.5cm);
    \draw[->] (u3) .. controls (3.3,-1) and (2.2,-1) .. (u3);
    \draw[->] (u1) -- (u2);
    \node[above] at (2.75,0.01){$g$};
    \node[below] at (2.75,-0.9){$h$};

  \draw[->, very thick] (4.5,0.25) -- (6,0.25);
  \node[above] at (5.25, 0.25) {\emph{Ext.}};
  \draw[->, very thick] (6,-0.25) -- (4.5,-0.25);
  \node[below] at (5.25, -0.25) {\emph{Red.}};

  \node[draw, circle, scale=0.7] (u1) at (7,0){};
  \node[draw, circle, scale=0.7] (u2) at (8.5,0){};
  \node[draw, circle, scale=0.7] (u3) at (7.75,-0.5){};
  \node[draw, circle, scale=0.7] (v) at (7.75,2){}; 
  \draw (7.75,0) circle (1.5cm);
    \draw[->] (v) .. controls (8.2,0.75) .. (u3);
  \draw[->] (v) .. controls (7.3,0.75) .. (u3);
  \draw[->] (v) -- (u1);
  \draw[->] (v) -- (u2);
  \node[left] at (7.3,0.65) {$g_1$};
  \node[right] at (7.32,1) {$h_1$};
  \node[left] at (8.18,0.5) {$h_2$};
  \node[right] at (8.2,0.65) {$g_2$};
\end{tikzpicture}
    \end{subfigure}
    \begin{subfigure}[scale=0.5]{0.47\textwidth}
        \begin{tikzpicture}
  [scale=.7,auto=left]
  \node[draw, circle, scale=0.7] (u1) at (2,0){};
  \node[draw, circle, scale=0.7] (u2) at (3.5,0){};
  \node[draw, circle, scale=0.7] (u3) at (2.75,-0.5){};
  \draw (2.75,0) circle (1.5cm);
  \draw[->](u1) -- (u3);
  \draw[->](u3) -- (u2);
  \node at (2.2,-0.5) {$g$};
  \node at (3.3,-0.5) {$h$};

  \draw[->, very thick] (4.5,0.25) -- (6,0.25);
  \node[above] at (5.25, 0.25) {\emph{Ext.}};
  \draw[->, very thick] (6,-0.25) -- (4.5,-0.25);
  \node[below] at (5.25, -0.25) {\emph{Red.}};

  \node[draw, circle, scale=0.7] (u1) at (7,0){};
  \node[draw, circle, scale=0.7] (u2) at (8.5,0){};
  \node[draw, circle, scale=0.7] (u3) at (7.75,-0.5){};
  \node[draw, circle, scale=0.7] (v) at (7.75,2){}; 
  \draw (7.75,0) circle (1.5cm);
    \draw[->] (v) .. controls (8.2,0.75) .. (u3);
  \draw[->] (v) .. controls (7.3,0.75) .. (u3);
  \draw[->] (v) -- (u1);
  \draw[->] (v) -- (u2);
  \node[left] at (7.3,0.65) {$g_1$};
  \node[right] at (7.32,1) {$g_2$};
  \node[left] at (8.18,0.5) {$h_1$};
  \node[right] at (8.2,0.65) {$g_2$};
\end{tikzpicture}
    \end{subfigure}
    \begin{subfigure}[scale=0.5]{0.47\textwidth}
        \begin{tikzpicture}
  [scale=.7,auto=left]
  \node[draw, circle, scale=0.7] (u1) at (1.8,0){};
  \node[draw, circle, scale=0.7] (u2) at (3.7,0){};
  \node[draw, circle, scale=0.7] (u3) at (2.3,-0.5){}; 
  \node[draw, circle, scale=0.7] (u4) at (3.2,-0.5){};
  \draw (2.75,0) circle (1.5cm);
  \draw[->] (u1) -- (u2);
  \draw[->] (u3) -- (u4);
  \node[above] at (2.75,0.01) {$g$};
  \node[below] at (2.75,-0.6) {$h$};

  \draw[->, very thick] (4.5,0.25) -- (6,0.25);
  \node[above] at (5.25, 0.25) {\emph{Ext.}};
  \draw[->, very thick] (6,-0.25) -- (4.5,-0.25);
  \node[below] at (5.25, -0.25) {\emph{Red.}};
  
  \node[draw, circle, scale=0.7] (u1) at (6.8,0){};
  \node[draw, circle, scale=0.7] (u2) at (8.7,0){};
  \node[draw, circle, scale=0.7] (u3) at (7.3,-0.5){}; 
  \node[draw, circle, scale=0.7] (u4) at (8.2,-0.5){};
  \node[draw, circle, scale=0.7] (v) at (7.75,2){};
  \draw (7.75,0) circle (1.5cm);
  \draw[->] (v) -- (u1);
  \draw[->] (v) -- (u2);
  \draw[->] (v) -- (u3);
  \draw[->] (v) -- (u4);
  \node[left] at (7.2,0.65) {$g_1$};
  \node[right] at (7.35,0.7) {$h_1$};
  \node[left] at (8.2,0.2) {$h_2$};
  \node[right] at (8.3,0.65) {$g_2$};
\end{tikzpicture}
    \end{subfigure}
    \caption{Examples of 2-extensions. Here, $g=g_1^{-1}g_2$, $h=h_1^{-1}h_2$ and $G[\{v_1,v_2,v_3,v_4\}]$ is $(2,3,1,0)$-sparse.} 
    \label{fig:2ext}
\end{figure}

Note that, in general, 1-reductions, loop-2-reductions and 2-reductions do not preserve sparsity counts: when applying these operations to a $(2,3,1)$-gain tight or $(2,3,1,0)$-gain tight graph, we add at least one edge (two in the case of 2-reductions). The added edge may lead the resulting graph to break the sparsity conditions. Given a $(2,3,1)$-gain tight (respectively, $(2,3,1,0)$-gain tight) graph $(G,\psi)$, we say a reduction is \textit{admissible} if the resulting gain graph is also $(2,3,1)$-gain tight (respectively, $(2,3,1,0)$-gain tight). If there is an admissible reduction at a vertex $v\in V(G)$, then we say that $v$ \textit{admits a reduction}, and that $(G,\psi)$ \textit{admits a reduction}.

\subsection{Constructive characterisation of gain tight gain graphs}\label{subsec:basegraphs}
We conclude the section by presenting a constructive characterisation of $(2,3,1)$-gain tight and $(2,3,1,0)$-gain tight graphs. The results given in this subsection were shown in \cite[Theorem 4.4 and Theorem 7.9]{jzt2016} for the case where $\Gamma$ is a dihedral group. The same proofs extend to the case where $\Gamma$ is a Fuchsian group (see \Cref{appA} for details). We start with presenting the constructive characterisation of $(2,3,1)$-gain tight graphs from a single vertex with an \emph{unbalanced loop}; i.e.~a loop with non-trivial gain.

\begin{theorem}
\label{thm: recur. con. easy}
    Let $\Gamma$ be a cyclic Fuchsian group and $(G,\psi)$ be a $\Gamma$-gain graph. Then, $(G,\psi)$ is $(2,3,1)$-gain tight if and only if it can be built up from a single vertex with an unbalanced loop by applying a series of $0$-extensions, $1$-extensions and loop-$1$-extensions. 
\end{theorem}

To present the constructive characterisation of $(2,3,1,0)$-gain tight graphs, we first need to define the `base graphs' for the recursion. A \emph{base graph} is a graph which is contained in one of the following three classes. 

The first class consists of `trivial graphs', `fancy triangles', and `fancy hats'. 
A \emph{trivial graph} is a gain graph composed of a simple vertex with two loops (see \Cref{fig:base}(a)) labelled so that it is not cyclic. 
A \emph{fancy triangle} is a gain graph whose underlying graph is obtained from a 3-cycle by adding a loop at each vertex (see \Cref{fig:base}(c)). The labels are assigned so that the 3-cycle is balanced and the full graph is not cyclic. A \emph{hat} is a graph obtained from a $K_{2,3}$ by adding an edge to the partite set of cardinality 2. A \emph{fancy hat} is a gain graph obtained from a hat by adding a loop to each degree 2 vertex (see \Cref{fig:base}(d)). The labels are assigned so that the hat is balanced and the full graph is not cyclic. It is easy to check that the trivial graph, a fancy triangle and a fancy hat are all $(2,3,1,0)$-gain tight.

The second class of graphs consists of $(2,3,1,0)$-gain tight graphs whose underlying graphs are `double cycles'. For $n\geq2$, the \emph{double cycle} $C_n^2$ is the graph obtained from the cycle $C_n$ by replacing each edge with two parallel edges. Note, $C_2^2$ is simply composed of two vertices with four parallel edges (see \Cref{fig:base}(b)).

The third class of graphs are the \emph{near-cyclic} graphs: $(2,3,1,0)$-gain tight graphs obtained from a cyclic graph by adding an edge.

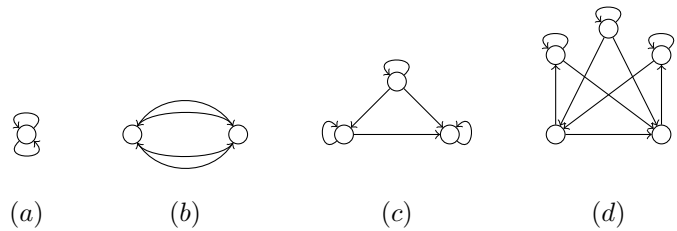
\begin{figure}[ht]
    \centering
    \begin{tikzpicture}
  [scale=.7,auto=left]
  \node(a)[draw, circle, scale = 0.75] at (1,0){};
  \draw[->] (a) .. controls (1.55,0.5) and (0.45,0.5) .. (a);
  \draw[->] (a) .. controls (0.45,-0.5) and (1.55,-0.5) .. (a);
  \node at (1,-1.5){$(a)$};

  \node(b1)[draw, circle, scale = 0.75] at (3,0){};
  \node(b2)[draw, circle, scale = 0.75] at (5,0){};
  \draw[->] (b1) .. controls (3.3, 0.5) and (4.7, 0.5) .. (b2);
  \draw[->] (b2) .. controls (4.5, 0.8) and (3.5, 0.8) .. (b1);
  \draw[->] (b2) .. controls (4.7, -0.5) and (3.3, -0.5) .. (b1);
  \draw[->] (b1) .. controls (3.5, -0.8) and (4.5, -0.8) .. (b2);
  \node at (4,-1.5){$(b)$};

  \node(d1)[draw, circle, scale = 0.75] at (7,0){};
  \node(d2)[draw, circle, scale = 0.75] at (9,0){};
  \node(d3)[draw, circle, scale = 0.75] at (8,1){};
  \draw[->] (d1) -- (d2);
  \draw[->] (d3) -- (d1);
  \draw[->] (d3) -- (d2);
  \draw[->] (d1) .. controls (6.5,-0.45) and (6.5,0.55) .. (d1);
  \draw[->] (d2) .. controls (9.5,-0.45) and (9.5,0.55) .. (d2);
  \draw[->] (d3) .. controls (8.55,1.5) and (7.45,1.5) .. (d3);
  \node at (8,-1.5){$(c)$};

  \node(e1)[draw, circle, scale = 0.75] at (11,0){};
  \node(e2)[draw, circle, scale = 0.75] at (13,0){};
  \node(e3)[draw, circle, scale = 0.75] at (12,2){};
  \node(e4)[draw, circle, scale = 0.75] at (11,1.5){};
  \node(e5)[draw, circle, scale = 0.75] at (13,1.5){};
  \draw[->] (e1) -- (e2);
  \draw[->] (e1) -- (e4);
  \draw[->] (e2) -- (e5);
  \draw[->] (e3) -- (e1);
  \draw[->] (e3) -- (e2);
  \draw[->] (e4) -- (e2);
  \draw[->] (e5) -- (e1);
  \draw[->] (e4) .. controls (11.55,2) and (10.45,2) .. (e4);
  \draw[->] (e5) .. controls (13.55,2) and (12.45,2) .. (e5);
  \draw[->] (e3) .. controls (12.55,2.5) and (11.45,2.5) .. (e3);
  \node at (12,-1.5){$(d)$};
\end{tikzpicture}
    \caption{Examples of base graphs. The unlabelled edges may be labelled freely, with the restrictions that the non-looped edges in (a), (c) and (d) are labelled $\text{id}$, and each graph must maintain the sparsity count.}
    \label{fig:base}
\end{figure}

\begin{theorem}
\label{thm: recur. con.}
    Let $\Gamma$ be a Fuchsian group and $(G,\psi)$ be a $\Gamma$-gain graph. Then, $(G,\psi)$ is $(2,3,1,0)$-gain tight if and only if it can be constructed by applying a series of $0$-extensions, $1$-extensions, loop-$1$-extensions, $2$-extensions, and loop-$2$-extensions to a disjoint union of finitely many base graphs. 
\end{theorem}

\section{Hyperbolic geometry and Fuchsian groups}\label{sec:hypgeom}

There are several isometric models of the hyperbolic plane, each with their own pros and cons. The model we consider in this paper is the upper half-plane model. 
For a general reference on hyperbolic geometry and Fuchsian groups, see \cite{katok1992fuchsian}, and for a general reference on Riemannian geometry, see \cite{lee2018introduction}. 

\subsection{The upper half-plane model}
As a topological space, the \emph{upper half plane} is given by
\begin{align*}
\HH &=\{z \in \CC ~ \vert ~ \im(z) > 0 \}.
\end{align*}
Often we will identify $\CC$ with the real linear space $\RR^2$ by considering each complex point $x+\mathfrak{i}y$ as the real vector $(x,y)$. The tangent space of $\HH$ at a point $P$ is denoted by $T_P\HH$.
Since $\HH$ is a topological submanifold of $\CC$ (but importantly not a Riemannian submanifold since their metrics differ),
each tangent space $T_P\HH$ is simply a copy of $\mathbb{C}$ (which we often consider as $\RR^2$).
We will make use of this identification throughout,
and we will only refer to tangent spaces by $T_P\HH$ if we wish to stress that the vector being considered is a tangent vector at $P$. 

The upper half plane is equipped with the Riemannian metric described as follows for a point $(x,y) \in \HH$:
\begin{equation*}
    \frac{(dx)^2 + (dy)^2}{y^2}.
\end{equation*}
This notation means that for each point $P=(x,y) \in \HH$,
there is an inner product $\langle \cdot , \cdot \rangle_P$ on $T_P(\HH)^2$ so that for any vectors $w = (x_w,y_w),w' = (x_w',y_w') \in T_P(\HH)$ we have
\begin{equation*}
    \langle w, w' \rangle_P= \frac{x_w x_w' + y_w y_w'}{y^2}.
\end{equation*}
The Riemannian metric above induces a distance metric on $\mathbb{H}$, whereby the distance between two points $P_1=(x_1,y_1)$ and $P_2=(x_2,y_2)$ of $\HH$ is given by the following equation (see, e.g., \cite[Theorem 1.2.6]{katok1992fuchsian}):
\begin{equation*}
d_{\HH}(P_1,P_2) =2 \arsinh \left(\frac{\|P_1 - P_2\|}{2\sqrt{y_1 y_2}} \right) ,
\end{equation*}
where $\| \cdot\|$ represents the standard Euclidean norm on $\mathbb{R}^2$,
i.e.~$\|(x,y)\| = \sqrt{x^2 + y^2}$.

\subsection{Geodesics}

A \emph{geodesic} on a Riemannian manifold is a curve (parametrised by arc length) that locally minimises the arc length, in the sense that any variation of the curve increases the arc-length. In the hyperbolic plane, there is a unique unit-speed geodesic between any pair of distinct points. In the upper half plane, these are either given by a vertical line or by a semi-circle with centre on the line $y=0$ \cite[Theorem 1.2.1]{katok1992fuchsian}, and the distance $d_{\mathbb{H}}(x,y)$ is then given by the length of the geodesic segment joining $x$ and $y$. The hyperbolic plane is also \textit{complete}, meaning that one can extend any geodesic infinitely. See \Cref{fig:hyperbolic-geodesics} for an illustration.

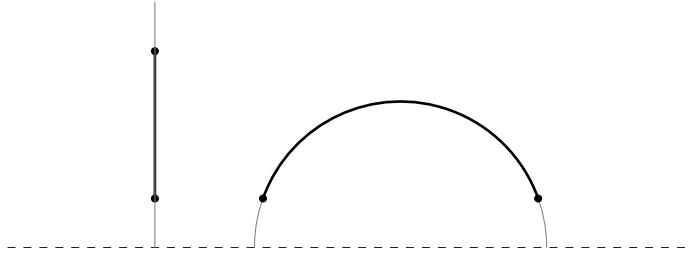
\begin{figure}
    \centering
\begin{tikzpicture}[scale=0.65]
\clip (-5,0) rectangle (10,5);
\draw[dashed] (-5,0) -- (9,0);
\filldraw(-2,4) circle(2pt);
\filldraw(-2,1) circle(2pt);
\draw[black, line width=1pt](-2,4) -- (-2,1);
\draw[gray] (-2,0) -- (-2,5);
\def\c{3}
\def\r{2.974}

\draw[gray] (\c+\r,0) arc[start angle=0,end angle=180,radius=\r];
\draw[black, line width=1pt]
  ({\c+\r*cos(160)},{\r*sin(160)}) arc[start angle=160,end angle=20,radius=\r];

\filldraw (0.2,1) circle(2pt);
\filldraw (5.8,1) circle(2pt);
\end{tikzpicture}
    \caption{Some geodesics in the upper-half plane.}
    \label{fig:hyperbolic-geodesics}
\end{figure}

We introduce the following notation that we use throughout.

\begin{definition}
We denote by $\gamma_{P,Q} : \mathbb{R} \rightarrow \mathbb{H}$ the unique smooth curve that passes through points $P,Q \in \HH$ whose parametrisation satisfies the following properties:
\begin{itemize}
	\item (geodesic): for any $t_1 \in \RR$, there exists $\varepsilon >0$ so that for any $t_1<t_2 < t_1 + \varepsilon$, the length of the domain-restricted curve $\gamma_{P,Q}|_{[t_1,t_2]}$ is exactly $d_{\HH}(\gamma_{P,Q}(t_1),\gamma_{P,Q}(t_2))$.
    \item (unit speed): given
    \begin{equation*}
        \dot{\gamma}_{P,Q}(t_0) := \frac{d}{dt} \gamma_{P,Q}(t) \bigg\vert_{t=t_0} \in \mathbb{R}^2,
    \end{equation*}
    the following holds for all $t_0 \in \mathbb{R}$:
    \begin{equation*}
        \left\langle \dot{\gamma}_{P,Q}(t_0), \dot{\gamma}_{P,Q}(t_0) \right\rangle_{\gamma(t_0)} = 1.
    \end{equation*}
    \item (starts at $P$ and passes through $Q$): $\gamma_{P,Q}(0) =P$ and $\gamma_{P,Q}(T) =Q$ for some $T>0$.
\end{itemize}
\end{definition}

Alternatively, we can define a unit-speed geodesic via its starting point and initial direction.

\begin{definition}\label{def:unitspeedbydirection}
    For $P \in \mathbb{H}$ and $w \in \mathbb{R}^2$ satisfying $\langle w, w \rangle_P=1$,
    we denote by $\gamma_{P:w} : \mathbb{R} \rightarrow \mathbb{H}$ the unique unit-speed geodesic that starts at $P$ with $\dot{\gamma}_{P:w}(0) = w$.
\end{definition}

Occasionally we will not want to concern ourselves with how a given geodesic is parametrised. 
With this in mind, we introduce the following concepts. 

For a geodesic $\gamma: \RR \rightarrow \HH$,
the points $x =\lim_{t \rightarrow -\infty} \gamma(t)$, $y =\lim_{t \rightarrow \infty} \gamma(t)$ (with limits taken with respect to the Euclidean norm $\|\cdot\|_2$), called the \emph{ideal points} of $\gamma$, are contained in the boundary $\RR \,\cup\, \{\infty\}$ of $\HH$. It is clear that we can describe oriented geodesics\footnote{Here we consider an oriented geodesic to be the image of a geodesic $\gamma : \RR \rightarrow \HH$ with ordered ideal points.} by the ideal endpoints $x\neq y \in \mathbb{R}\,\cup\, \{\infty\}$. It is relatively standard to describe geodesics in this way (see for example \cite[Section 8.5]{MR4554426}). We thus define
\begin{equation*}
    \textup{Geod}(\mathbb{H}) := \{(x,y)\in (\mathbb{R} \,\cup\, \{\infty\})^2 ~\vert~ x\neq y \}. 
\end{equation*}
Using the chart $x\mapsto 1/x$ around $\infty$ and the fact that the points of $\textup{Geod}(\mathbb{H})$ are in one-to-one correspondence with oriented geodesics, this gives $\textup{Geod}(\mathbb{H})$ the structure of a smooth manifold. 
Moreover, we can easily pass from geodesics of the form $\gamma_{P,Q}$ for $P,Q\in\HH$ and elements of $\textup{Geod}(\mathbb{H})$ using the following result.

\begin{lemma}\label{continuity-geodesics}
    The map
    \begin{equation*}
        \chi : \left\{(P,Q) \in\mathbb{H}^{2}~\vert~P\neq Q \right\} \longrightarrow \textup{Geod}(\HH), ~ (P, Q) \longmapsto \lim_{t \rightarrow \infty} \Big( \gamma_{P,Q}(t) , \gamma_{P,Q}(-t) \Big)
    \end{equation*}
    is continuous and surjective.
\end{lemma}
\begin{proof}
    Since any two points on the boundary of $\HH$ describe an oriented geodesic, $\chi$ is surjective. If $P=x_1+ \mathfrak{i} y_1$ and $Q=x_2, + \mathfrak{i} y_2$ with $x_1 \neq x_2$, then $\gamma_{P, Q}$ parametrises the semi-circle with centre $(C,0)$ and radius $R$ given by 
    \begin{equation*}
        C = \frac{x_1^2-x_2^2 + y_1^2-y_2^2}{2(x_1- x_2)}, \qquad R= \sqrt{\left( x_1 - C \right)^2 + y_1^2}.
    \end{equation*}
    With this we have
    \begin{equation*}
        \chi(P, Q) =
        \begin{cases}
            (C-R, C+R) &\text{if } x_1 < x_2, \\
            (C+R, C-R) &\text{if } x_1 > x_2, \\
            (x_1, \infty) &\text{if } x_1 = x_2, ~ y_1 < y_2, \mbox{ and} \\
            (\infty,x_1) &\text{if } x_1 = x_2, ~ y_1 > y_2.
        \end{cases}
    \end{equation*}
    By observing the behaviour of $C,R$ as $x_1$ converges to $x_2$, we see that $\chi$ is continuous.
\end{proof}

Suppose that two geodesics $(x_1,y_1),(x_2,y_2)\in\text{Geod}(\HH)$ intersect at a single point $P\in \HH$ and, for $i=1,2$, let $T_i$ denote the tangent space of the geodesic $(x_i,y_i)$ at the point $P$ when considered as a smooth submanifold of $\HH$. We say $(x_1,y_1),(x_2,y_2)\in\text{Geod}(\HH)$ \textit{intersect transversally} if $T_1 + T_2 = \RR^2$.
The following lemma (whose tedious but trivial proof we omit) allows us to determine if geodesics intersect by checking orderings of their ideal points. Note that if two distinct geodesics intersect, then they intersect transversally, which is also reflected in the following lemma.

\begin{lemma}\label{intersecting geodesics}
    Two distinct geodesics $(x_1,y_1),(x_2,y_2) \in \textup{Geod}(\HH)$ intersect (transversally) if and only if
    \begin{equation*}
    x_1 < x_2 < y_1 < y_2 \qquad \textup{ or } \qquad x_2 < x_1 < y_2 < y_1.
    \end{equation*}
\end{lemma}

We can also describe transverse intersection for geodesics of the form $\gamma_{P,P_1}, \gamma_{P,P_2}$ in terms of collinearity.

\begin{restatable}{lemma}{keylemma}\label{lem:linear_independence}
    Let $P_1,P_2$ be distinct points in $\mathbb{H}$.
    Then for a point $P\in \HH$ with $P\neq P_1$ and $P\neq P_2$, the pair of vectors $\dot{\gamma}_{P,P_1}(0), \dot{\gamma}_{P,P_2}(0)$ are linearly dependent if and only if $P$ lies on the geodesic $\gamma_{P_1,P_2}$.
\end{restatable}

\begin{proof}
    Suppose $P$ lies on the geodesic $\gamma_{P_1,P_2}$ with $\gamma_{P_1,P_2}(T) = P_2$ for some $T>0$.
    Then $P = \gamma_{P_1,P_2}(t)$ for some $t \in \mathbb{R} \setminus \{0,T\}$,
    and one of the following cases hold:
    \begin{itemize}
        \item If $t<0$ then $\dot{\gamma}_{P,P_1}(0) = \dot{\gamma}_{P,P_2}(0) = \dot{\gamma}_{P_1,P_2}(t)$.
        \item If $t>T$ then $\dot{\gamma}_{P,P_1}(0) = \dot{\gamma}_{P,P_2}(0) = -\dot{\gamma}_{P_1,P_2}(t)$. 
        \item If $0<t<T$ then $\dot{\gamma}_{P,P_1}(0) = - \dot{\gamma}_{P_1,P_2}(t)$ and $\dot{\gamma}_{P,P_2}(0) = \dot{\gamma}_{P_1,P_2}(t)$. 
    \end{itemize}
    Hence, $\dot{\gamma}_{P,P_1}(0), \dot{\gamma}_{P,P_2}(0)$ are linearly dependent.

    Now suppose $w_1 =\dot{\gamma}_{P,P_1}(0)$ and $w_2 = \dot{\gamma}_{P,P_2}(0)$ are linearly dependent.
    If $w_1=w_2$ then $\gamma_{P:w_1} = \gamma_{P:w_2}$, while if $w_1=-w_2$ then $\gamma_{P:w_1} = \gamma_{P:-w_2}$.
    In either case,
    this implies $P,P_1,P_2$ all lie on the unit-speed geodesic $\gamma_{P:w_1}$.
    Since every unit-speed geodesic is uniquely determined by two distinct points,
    we have that $P,P_1,P_2$ all lie on the geodesic $\gamma_{P_1,P_2}$.
\end{proof}

The derivative of the distance metric $d_{\HH}$ can also be understood using geodesics.
Here we denote the derivative of a map $f$ at a point $x$ by either $D_x f$ or $D_x (f)$. 

\begin{proposition}[{\cite[IX.80]{Bourguignon2022-za}}]\label{metric-metric}
    Let $P = (x_P,y_P),Q= (x_Q,y_Q) \in \mathbb{H}$ with $P\neq Q$, and let $w_P\in T_{P}\mathbb{H}$ and $w_Q\in T_{Q}\mathbb{H}$. Then
    \begin{equation*}
        D_{(P,Q)}(d_\mathbb{H})(w_P, w_Q) = \langle \dot{\gamma}_{P,Q}(0)  , w_P \rangle_{P} + \langle \dot{\gamma}_{Q,P}(0)  , w_Q \rangle_{Q} =  \langle \dot{\gamma}_{P,Q}(0)  , w_P \rangle_{P} -\langle \dot{\gamma}_{P,Q}(T)  , w_Q \rangle_{Q},
    \end{equation*}
    where $T$ is such that $\gamma_{P,Q}(T) = Q$. Identifying each tangent space with $\mathbb{R}^2$, the derivative $D_{(P,Q)}(d_\mathbb{H})$ at points $P\neq Q$ is described by the $1 \times 4$ Jacobian matrix
    \begin{equation*}
    		\left[ \frac{\dot{\gamma}_{P,Q}(0)}{y_P^2}^T ~  \frac{\dot{\gamma}_{Q,P}(0)}{y_Q^2}^T  \right].
    \end{equation*}
\end{proposition}

We can explicitly compute the Jacobian $D_{(P,Q)}(d_\mathbb{H})$  to be the $1 \times 4$ matrix $[A ~ B ~ C ~ D ]$ where
\begin{align*}
    A &= \frac{2}{\sqrt{(\|P-Q\|)^4 +4y_Py_Q \|P-Q\|^2 }} \cdot\left( x_P -x_Q \right), \\
    B &= \frac{2}{\sqrt{(\|P-Q\|)^4 +4y_Py_Q \|P-Q\|^2}} \cdot\left(y_P-y_Q - \frac{\| P-Q\|^2}{2 y_P} \right), \\
    C &= \frac{2}{\sqrt{(\|P-Q\|)^4 +4y_Py_Q \|P-Q\|^2 }} \cdot\left( x_Q - x_P \right) \mbox{ and} \\
    D &= \frac{2}{\sqrt{(\|P-Q\|)^4 +4y_Py_Q \|P-Q\|^2}} \cdot\left(y_Q-y_P - \frac{\| P-Q\|^2}{2 y_Q} \right).
\end{align*}
Importantly,
we see that the map $(P,Q) \mapsto \dot{\gamma}_{P,Q}(0)$ is continuous on the set of distinct ordered pairs $(P,Q) \in \HH^2$ with $P \neq Q$.

\subsection{Isometries of the hyperbolic plane}

The Lie group $\SL(2,\RR)$ is the group of $2 \times 2$ real matrices with determinant 1, and the quotient
$\SL(2, \mathbb{R}) / \{\pm id\}$, is the Lie
group $\PSL(2,\RR)$ which represents the orientation-preserving isometries of the upper half plane. The matrices in $\PSL(2, \RR)$ act on the points $z \in \mathbb{H}$ (considered here as complex points) by the following action:
\begin{equation*}
\begin{bmatrix}
a & b \\
c & d 
\end{bmatrix}
\cdot 
z := \frac{
az+ b}{cz + d} .
\end{equation*}
We will often treat elements of $\textup{PSL}(2, \mathbb{R})$ simply as matrices, though they are technically equivalence classes. By abuse of notation,
we consider each matrix $g \in \PSL(2,\RR)$, as given above, to also be a map $g: \HH \rightarrow \HH$ with derivative at $z \in \HH$ given by the non-singular Jacobian matrix
\begin{equation}\label{eq:dzg}
	D_z(g) =	
	\begin{bmatrix}
        \re \left(\frac{1}{(cz + d)^2} \right)  & \im \left(\frac{1}{(cz + d)^2} \right) \\
       -\im \left(\frac{1}{(cz + d)^2} \right) & \re \left(\frac{1}{(cz + d)^2} \right) 
    \end{bmatrix}.
\end{equation}
An isometry of $\HH$ that does not preserve orientation is the reflection in the $y$-axis, i.e.~the map $\rho: z \mapsto -\overline{z}$.
In fact, every isometry of $\HH$ that does not preserve orientation can be decomposed into a composition $g \circ \rho$ for some orientation-preserving isometry $g$. Since $\PSL(2,\RR)$ is connected, the isometry group of $\HH$ has exactly two connected components.

It follows from $g$ being an isometry of a Riemannian manifold that the above Jacobian satisfies the following equality for any $w_1, w_2 \in \RR^2$:
\begin{equation}\label{eq:isometries_metric_1}
    \langle D_z (g) w_1, D_z (g) w_2 \rangle_{g \cdot z} = \langle w_1, w_2 \rangle_{z}.
\end{equation}
The derivative $D_z(g)$ behaves especially well with geodesics.
Namely,
for any distinct $z,z' \in \mathbb{H}$, we have
\begin{equation}\label{eq:isometries_metric_2}
    D_z(g) \dot{\gamma}_{z,z'}(0) = \dot{\gamma}_{gz, gz'}(0).
\end{equation}

Since $\HH$ sits in the compact Riemann sphere $\CC \,\cup\, \{\infty\}$,
we can extend every isometry $g : \HH \rightarrow \HH$ to a unique continuous map $\tilde{g} : \HH\,\cup\,\RR \,\cup\,\{\infty\} \rightarrow \HH\,\cup\,\RR\,\cup\,\{\infty\}$. This induces an action of $\text{PSL}(2,\RR)$ on $\text{Geod}(\HH)$, given by $g\cdot(x, y)=(\tilde g\cdot x, \tilde g\cdot y).$ Using $\tilde{g}$, we can classify every non-identity isometry as exactly one of the following \cite[Section 2.1]{katok1992fuchsian}:
\begin{itemize}
	\item (elliptic): $g$ has exactly one fixed point, and for some $\theta\in [0,2\pi)$, $g$ is conjugate to
    \begin{equation*}
        \begin{bmatrix}
            \cos(\theta) & -\sin(\theta)\\
            \sin(\theta) & \cos(\theta)\\
        \end{bmatrix};
    \end{equation*}
	\item (parabolic): $g$ has no fixed points, but $\tilde{g}$ has exactly one fixed point, and for some $t\in \mathbb{R}$, $g$ is conjugate to
    \begin{equation*}
    \begin{bmatrix}
        1 & t\\
        0 & 1\\
    \end{bmatrix};
    \end{equation*}
	\item (hyperbolic): $g$ has no fixed points, but $\tilde{g}$ has exactly two distinct fixed points, and for some $\lambda\in \mathbb{R}\setminus\{0\}$, $g$ is conjugate to
    \begin{equation*}
        \begin{bmatrix}
            \lambda & 0\\
            0 & \frac{1}{\lambda}\\
        \end{bmatrix}.
    \end{equation*}
    In this case, $g$ fixes one geodesic which is called the axis of $g$.
\end{itemize}
By abuse of terminology we describe any fixed point of $\tilde{g}$ to be a fixed point of $g$ and then specify if it is contained in $\HH$ or its boundary if needed.

\subsection{Infinitesimal isometries}

The tangent space of $\PSL(2,\RR)$ at the identity is identical to the tangent space of $\SL(2,\RR)$ at the identity. It is the 3-dimensional Lie algebra
\begin{equation*}
    \mathfrak{sl}(2,\RR) := \left\{A\in \RR^{2\times 2} ~\vert~ Tr(A)= 0\right\}.
\end{equation*}
We can move between the Lie algebra $\mathfrak{sl}(2,\mathbb{R})$ and Lie group $\PSL(2,\mathbb{R})$ using the \emph{exponential map}
\begin{equation*}
    \exp : \mathfrak{sl}(2,\mathbb{R}) \longrightarrow \PSL(2,\mathbb{R}), ~ X \longmapsto \exp(X) : = \sum_{n=0}^\infty \frac{1}{n!} X^n.
\end{equation*}
The exponential map is a well-defined local diffeomorphism. Moreover, $\exp$ is also surjective \cite{moskowitz1994surjectivity}.

Consider the group action $(g, z) \mapsto g \cdot z$ for $g \in \PSL(2,\RR)$ and $z \in \HH$.
The partial derivative of this action with respect to $\PSL(2,\RR)$ at the identity gives the smooth map
\begin{equation}\label{eq:U}
    \eta:\mathfrak{sl}(2,\RR) \times \mathbb{H} \longrightarrow \mathbb{C}, ~ 
    \left( 
    \begin{bmatrix}
        a & b \\
        c & -a 
    \end{bmatrix} ,
    z \right) 
    \longmapsto b + 2a z - cz^2.
\end{equation}
The map $\eta$ can be linked to the derivative of $d_{\HH}$ by the following lemma.

\begin{lemma}
    Let $P,Q$ be distinct points in $\HH$ and choose any $A \in \mathfrak{sl}(2,\RR)$.
    Then
    \begin{equation}\label{eq:trivial_inf_edge}
        D_{(P,Q)}(d_{\HH})\Big(\eta(A, P),\eta(A, Q) \Big) = 0.
    \end{equation}
\end{lemma}

\begin{proof}
    Let $g:(-1,1) \rightarrow \PSL(2,\RR)$ be a smooth curve where $g(0) = \textrm{id}$ and $\dot{g}(0) = A$. We will write $g(t) = g_t$.
    By the chain rule we have that 
    \begin{equation*}
        \frac{d}{dt} d_{\HH} ( g_t \cdot P, g_t \cdot Q) \Big|_{t=0} = D_{(P,Q)}(d_{\HH})\left( \frac{d}{dt} g_t \cdot P \Big|_{t=0}, \frac{d}{dt} g_t \cdot Q \Big|_{t=0} \right) = D_{(P,Q)}(d_{\HH})\Big(\eta(A, P),\eta(A, Q) \Big).
    \end{equation*}
    Since $d_{\HH}(g_t \cdot P, g_t \cdot Q) = d_{\HH}(P,Q)$ for each $t \in (-1,1)$, the above equation implies \Cref{eq:trivial_inf_edge} holds.
\end{proof}

\begin{definition}\label{def:hz}
For $z\in\HH$, we define  $\mathfrak{h}_z:=\{X\in\mathfrak{sl}(2,\RR): \eta(X,z)=\textbf{0}\}.$
\end{definition} 

It follows from \Cref{eq:U} that each $\mathfrak{h}_z$ is a linear subspace of $\mathfrak{sl}(2,\RR)$.  
We can be even more specific on the structure of $\mathfrak{h}_z$.
Suppose that $z = x + \mathfrak{i}y \in \HH$.
Then \Cref{eq:U} shows that
\begin{equation*}
    X=\begin{bmatrix}
       a & b \\
       c & -a
   \end{bmatrix}\in \mathfrak{h}_z
   \qquad
   \iff \qquad
   b+ 2a x - c x^{2} + cy^{2} = 2a y - 2c xy= 0.
\end{equation*}
Since $y>0$ (as $z\in\HH$), the right-hand side holds if and only if $a=cx$ and $b=(-x^2-y^2)c$. 
Therefore, we have
\begin{equation}\label{eq:hz}
    \mathfrak{h}_z=\left\{\lambda
    \begin{bmatrix}
        x & -x^2-y^2\\
        1 & -x
    \end{bmatrix}
    :\lambda\in\RR\right\}.
\end{equation}
Note that each $\mathfrak{h}_z$ is a 1-dimensional linear space, and so is uniquely represented by an element of the projective space $\mathbb{P}(\mathfrak{sl}(2, \mathbb{R}))$.

Later, we will require the following technical lemmas regarding $\mathfrak{h}_z$.

\begin{lemma}\label{lem: Lie -algebras-stabilizers0}
    For any $z \in \HH$,
    the set $\{\mathfrak{h}_z \in \mathbb{P}(\mathfrak{sl}(2, \mathbb{R})) ~\vert ~ z\in \mathbb{H}\}$ is a non-empty open subset of the projective space $\mathbb{P}(\mathfrak{sl}(2, \mathbb{R}))$. 
\end{lemma}
\begin{proof}
    We first note that the set 
    \begin{equation*}
        S
        :=
        \left\{ \lambda 
        \begin{bmatrix}
             x & t\\
            1 & -x 
        \end{bmatrix} 
        : \lambda \neq 0, ~ t +x^2 <0  \right\}
    \end{equation*}
    is a non-empty open subset of $\mathfrak{sl}(2,\mathbb{R})$ and the image of the map
    \begin{equation*}
        \HH \times (\RR \setminus \{0\}) \longrightarrow \mathfrak{sl}(2,\mathbb{R}), ~
        (x+\mathfrak{i}y, \lambda) \longmapsto 
        \lambda
        \begin{bmatrix}
           x & -x^2-y^2\\ 1 & -x
        \end{bmatrix}.
    \end{equation*}
    Thus $\{\mathfrak{h}_z \in \mathbb{P}(\mathfrak{sl}(2, \mathbb{R})) ~\vert ~ z\in \mathbb{H}\} = \mathbb{P}(S)$,
    proving the result.
\end{proof}

\begin{lemma}\label{lem: Lie -algebras-stabilizers1}
    For any $z \in \HH$ and $g\in \PSL(2,\RR)$ we have $g\mathfrak{h}_z g^{-1} = \mathfrak{h}_{g\cdot z}$.
\end{lemma}
\begin{proof}
    Choose some $g\in \PSL(2,\RR)$ and some $X\in\mathfrak{h}_z$. Then by the chain rule we have
   \begin{equation*}
       \eta(g X g^{-1},g z)=\frac{d}{dt}(g\exp(tX)g^{-1}g z)=\frac{d}{dt}(g\exp(tX)z)=D_z(g) \frac{d}{dt}(\exp(tX)z)= D_z(g) \eta(X,z),
    \end{equation*}
    where $D_z(g)$ is the Jacobian given in \Cref{eq:dzg}.
    Since $D_z(g)$ is non-singular for any $z$ and $g$, this proves the result.
\end{proof}

\begin{lemma}\label{lem: eqfrom_lem_inf_motions}
    For any $z,w\in \mathbb{H}$ and $X \in \mathfrak{h}_{w}$, we have $\langle \dot{\gamma}_{z,w}(0), \eta(X, z)\rangle_{z} = \textup{\textbf{0}}$.
\end{lemma}

\begin{proof}
    Since $\eta(X,w) = \textbf{0}$, we have that
    \begin{align*}
        \Big\langle \dot{\gamma}_{z,w}(0), \eta(X, z)\Big\rangle_{z} &= \Big\langle \dot{\gamma}_{z,w}(0),  \eta(X, z) \Big\rangle_z + \Big\langle \dot{\gamma}_{w,z}(0),  \eta(X,w) \Big\rangle_w \\
        &= \frac{d}{dt}\Big(d_{\mathbb{H}}\big(\exp(t X) z, \exp(t X) w \big) \Big)\Big\vert_{t=0}\nonumber\\
        &= \frac{d}{dt}\Big(d_{\mathbb{H}}\big( z,w \big) \Big)\Big\vert_{t=0}=0. \qedhere
    \end{align*} 
\end{proof}

\subsection{Fuchsian groups and surface groups}
\label{sec: fuc surf groups}

A \textit{Fuchsian group} is a discrete subgroup of $\mathrm{PSL}(2,\mathbb{R})$.
One class of Fuchsian groups we are interested in is the class of surface groups. We define the \textit{abstract surface group (of genus $g\geq 2$)} to be the group $\Gamma_g$ given by the presentation.
\begin{equation*}
    \Gamma_g=\langle a_1,b_1,\dots,a_g,b_g:[a_1,b_1]\cdots[a_g,b_g]=\text{id}\rangle,
\end{equation*}
where $[a,b]$ denotes the commutator $aba^{-1}b^{-1}$ of $a$ and $b$. These groups are the fundamental groups of genus $g$ compact surfaces. 
By the Uniformisation Theorem \cite{Unif1,Unif2}, every compact Riemann surface of genus $g \geq 2$ with constant curvature is isometric to a quotient $\mathbb{H}/\Gamma$, where $\Gamma$ is an abstract surface group of genus $g$ acting freely and properly discontinuously\footnote{A group $G$ acts properly discontinuously on a topological space $X$ is for every $x\in X$ there is an open neighbourhood $U_x\subset X$ such that $g\cdot U_x\,\cap\,U_x=\emptyset$ for all $g\neq\text{id}$ in $G$.} on $\mathbb{H}.$
We say a subgroup $\Gamma \leq \PSL(2,\RR)$ is a \emph{surface group (of genus $g \geq 2$)} if it is isomorphic to $\Gamma_g$, co-compact (i.e.~$\mathbb{H}/\Gamma$ is compact) and the action on $\HH$ by $\Gamma$ is free and properly discontinuous.
Importantly, if $\Gamma$ is a surface group of genus $g \geq 2$,
then $\HH/ \Gamma$ is a compact Riemann surface of genus $g \geq 2$ with constant curvature: $\HH/\Gamma$ is a Riemann surface with constant curvature by \cite[Corollary 2.29]{lee2018introduction}, and, since the fundamental group of $\HH/\Gamma$ is $\Gamma_g$, the surface $\HH/\Gamma$ has genus $g$ by \cite[Theorem 2.6]{MR2895884}. Additionally, every surface group is a Fuchsian group by \cite[Proposition 2.19]{MR2895884}.

We later require the following two results regarding abelian Fuchsian groups and surface groups.

\begin{lemma}\label{lem:cyclic-iff-fix}
If $\Lambda$ is a subgroup of a Fuchsian group, then the following are equivalent:
\begin{enumerate}
    \item\label{abelian} $\Lambda$ is abelian;
    \item\label{cyclic} $\Lambda$ is cyclic;
    \item\label{fixed points} all non-identity elements of $\Lambda$ have the same fixed points when acting on $\HH \,\cup\, \mathbb{R}\,\cup\, \{\infty\}$; and
    \item\label{max_1_d_subgroup} $\Lambda$ is contained in a subgroup $\{\exp(tX) \in \PSL(2, \mathbb{R})~\vert~t\in \mathbb{R}\}$ for some $X\in \mathfrak{sl}(2, \mathbb{R})$.  
\end{enumerate}
\end{lemma}
\begin{proof}
   \ref{abelian} $\implies$ \ref{cyclic} is \cite[Theorem 2.3.6]{katok1992fuchsian}, \ref{cyclic} $\implies$ \ref{fixed points} follows directly from \cite[Theorem 2.3.2 ]{katok1992fuchsian}, and \ref{fixed points} $\implies$ \ref{cyclic} is \cite[Theorem 2.3.5]{katok1992fuchsian}. Therefore, it suffices to show that \ref{cyclic} $\implies$ \ref{max_1_d_subgroup} and \ref{max_1_d_subgroup} $\implies$ \ref{abelian}. To see that \ref{cyclic} $\implies$ \ref{max_1_d_subgroup}, suppose $\Lambda = \langle g \rangle$ for some $g\in\Lambda$. Since $\exp$ is surjective, there is some $X\in \mathfrak{sl}(2, \mathbb{R})$ such that $\exp(X)= g$. Then $\{\exp(tX) ~\vert ~t\in \mathbb{R}\}$ contains $\Lambda$, and  \ref{max_1_d_subgroup} holds.
   Finally, we show that \ref{max_1_d_subgroup} $\implies$ \ref{abelian}. Let $\Lambda$ be contained in $\{\exp(tX) \in \PSL(2, \mathbb{R})~\vert~t\in \mathbb{R}\}$ for some $X\in \mathfrak{sl}(2, \mathbb{R})$. By \cite[Proposition 20.8(b)]{MR2954043}, 
   \begin{equation*}
        \exp(aX)\exp(bX)=\exp((a+b)X)=\exp(bX)\exp(aX) \qquad \text{ for all } a,b \in \RR.
   \end{equation*}
   From this it follows that $\Lambda$ is abelian. 
\end{proof}

\begin{lemma}\label{lem:noparabolic}
    Let $\Gamma \leq \PSL(2,\RR)$ be a co-compact Fuchsian group isomorphic to an abstract surface group of genus $g\geq 2$.
    Then $\Gamma$ is a surface group of genus $g$. Moreover, $\Gamma$ contains no parabolic or elliptic elements, and has no element of finite order.
\end{lemma}

\begin{proof}
    By \cite[Proposition 2.19]{MR2895884},
    $\Gamma$ acts properly discontinuously on $\HH$.
    By the Uniformisation Theorem,
    the abstract surface group isomorphic to $\Gamma$ can act on $\HH$ both freely and properly discontinuously (although the action of $\Gamma$ as a subgroup of $\PSL(2,\RR)$ may not be).
    Thus $\Gamma$ contains no finite order elements by \cite[Corollary 2.24]{MR2895884}.
    Now suppose for contradiction that $h \in \Gamma$ is elliptic.
    By conjugating $\Gamma$ (which does not change our initial assumptions),
    we may suppose that 
    \begin{equation*}
        h = 
        \begin{bmatrix}
            \cos(\theta) & -\sin(\theta)\\
            \sin(\theta) & \cos(\theta)\\
        \end{bmatrix}
    \end{equation*}
    for some $\theta \in [0, 2 \pi)$.
    Since $h$ cannot have finite order, $2 \pi / \theta$ is not a rational value.
    Since $\Gamma$ is closed,
    we thus have that
    \begin{equation*}
        \overline{\langle h \rangle} = \left\{ \begin{bmatrix}
            \cos(\phi) & -\sin(\phi)\\
            \sin(\phi) & \cos(\phi)\\
        \end{bmatrix} : \phi \in [0,2\pi)  \right\} \subseteq \Gamma.
    \end{equation*}
    However, this contradicts that $\Gamma$ is discrete.
    Hence $\Gamma$ contains no elliptic elements.
    Since elliptic isometries are the only elements of $\PSL(2,\RR)$ that fix a point in $\HH$,
    it follows that $\Gamma$ acts freely on $\HH$ also.
    Hence $\Gamma$ is a surface group of genus $g$.

    Since $\mathbb{H}/\Gamma$ is a compact orientable surface, that $\Gamma$ contains no parabolic elements follows from \cite[Corollary 4.2.7]{katok1992fuchsian}, which states that if $\Gamma'$ is a co-compact Fuchsian group (i.e.~$\mathbb{H}/\Gamma'$ is compact), then $\Gamma'$ has no parabolic elements.    
\end{proof}

\section{Hyperbolic symmetric rigidity}\label{sec:hypsymrig}

Our main objects of study are \emph{$\Gamma$-symmetric frameworks} for some fixed choice of a Fuchsian group $\Gamma$:
pairs $(\tilde{G},\tilde{p})$, where $\tilde{G}=(\tilde{V},\tilde{E})$ is a $\Gamma$-symmetric graph and $\tilde{p}: \tilde{V}\rightarrow \HH$ is a map such that for any $g\in \Gamma$, one has $\tilde{p}(g \cdot v) = g \cdot \tilde{p}(v)$. 
Since $\Gamma$-symmetric frameworks of $\tilde{G}$ are determined by the positions of their vertex orbits, we will directly define rigidity concepts throughout the paper in the language of gain graphs.

\subsection{Rigid joint-configurations}

We begin with the following definition.
\begin{definition}
	Let $\Gamma$ be a Fuchsian group and let $(G,\psi)$ be a $\Gamma$-gain graph. 
    Any map $p: V(G)\rightarrow \HH$ is said to be a \emph{realisation} of $(G,\psi)$.
    A triple $(G, \psi,p)$ consisting of a $\Gamma$-gain graph $(G, \psi)$ and a realisation $p$ of $(G,\psi)$ is called a \emph{$\Gamma$-joint-configuration} (or just \emph{joint-configuration} when the group is clear).
    We identify the space of all realisations of $(G,\psi)$ with $\HH^{V(G)}$. 
\end{definition}

When $\Gamma$ is a Fuchsian group, any joint-configuration $(G, \psi,p)$ can be lifted to a $\Gamma$-symmetric framework $(\tilde{G},\tilde{p})$ by setting $\tilde{G}$ to be the $\Gamma$-lift of $(G,\psi)$ and $\tilde{p}(g\cdot v) = g\cdot p(v)$ for each vertex $v \in V(G)$ and $g \in \Gamma$. 
Note that two different joint-configurations can lift to the same $\Gamma$-symmetric framework;
if they do, then we say that they are \emph{geometrically equivalent}.
If two joint-configurations are geometrically equivalent then their gain graphs are equivalent.

\begin{definition}
Let $\Gamma$ be a Fuchsian group, and let $(G, \psi)$ be a finite $\Gamma$-gain graph. For any edge $e$ given by $u\xrightarrow{g}v$ and any $p \in \HH^{V(G)}$, let 
\begin{equation*}
f_e(p) =d_{\HH}\big(p(u), g\cdot p(v) \big)
\end{equation*}
and define the \emph{rigidity map}
\begin{equation*}
f_{G, \psi} : \mathbb{H}^{V(G)} \longrightarrow \mathbb{R}^{E(G)} \hspace{.6cm} p \longmapsto \big(f_e(p) \big)_{e\in E(H)}.
\end{equation*}
\end{definition}

Note that the rigidity map $f_{G,\psi}$ is not invariant under the action $p \mapsto g\cdot p := (g\cdot p(v))_{v \in V(G)}$ for all isometries $g: \HH \rightarrow \HH$.
The isometries for which this is true can be characterised as follows: 

\begin{proposition}\label{prop:invariant}
    Let $\Gamma$ be a Fuchsian group and let $g: \HH \rightarrow \HH$ be an isometry.
    Then the following are equivalent:
    \begin{enumerate}
        \item For every $\Gamma$-gain graph $(G,\psi)$, we have $f_{G,\psi}(g\cdot p) = f_{G,\psi}(p)$ for each $p \in \HH^{V(G)}$.
        \item $g$ commutes with every element of $\Gamma$.
    \end{enumerate}
\end{proposition}

\begin{proof}
    If $g$ commutes with $\Gamma$, then for an arbitrary edge $u\xrightarrow{h}v$ and any $p \in \HH^{V(G)}$ we compute
    \begin{equation*}
         d_{\HH}\big(g\cdot p(u),  (hg)\cdot p(v)\big) = d_{\HH}\big(g\cdot p(u), (gh)\cdot p(v) \big) = d_{\HH}\big(p(u),h\cdot p(v) \big),
    \end{equation*}
    and hence $f_{G,\psi}$ is invariant under $g$. 

    To prove the converse, suppose for a contradiction there exists $h \in \Gamma$ such that $gh \neq hg$ but 
    \begin{equation}\label{eq:invariant}
        d_{\HH}(x, (g^{-1}hg)\cdot y ) = d_{\HH}(g\cdot x,  (hg)\cdot y) = d_{\HH}(x,h\cdot y )
    \end{equation}
    for all $x,y \in \HH$.
    If we set $x = h\cdot y$, then the right hand side of \Cref{eq:invariant} is zero but the left hand side is not zero.
    This completes the proof. 
\end{proof}

It follows from \Cref{prop:invariant}
that the subgroup of orientation-preserving isometries $$ \left\{g \in \PSL(2,\mathbb{R}): f_{G,\psi} (g\cdot p) = f_{G,\psi}(p) \mbox{ for all } p\in \mathbb{H}^{V(G)} \right\}$$ is exactly the centraliser of $\Gamma$,
which we denote here by $C_{\PSL(2,\mathbb{R})}(\Gamma)$.
\Cref{prop:invariant} also justifies the following definition for rigidity in our setting.

\begin{definition}\label{def:rigiditysymm}
Let $\Gamma$ be a Fuchsian group.
A joint-configuration $(G, \psi, p)$ is \emph{locally $\Gamma$-symmetrically rigid} if there exists an open neighbourhood $U\subseteq \mathbb{H}^{V(G)}$ of $p$ such that, if $q\in f_{G, \psi}^{-1}(f_{G,\psi}(p)) \,\cap\, U$,  then there exists $g \in C_{\PSL(2,\mathbb{R})}(\Gamma)$  such that $g\cdot q =p$. 
Otherwise, we say $(G, \psi, p)$ is \emph{locally $\Gamma$-symmetrically flexible}.
\end{definition}

\subsection{The orbit rigidity matrix}
Given an edge $e=u\xrightarrow{g}v$ of a joint-configuration $(G,\psi,p)$, for any $X \in \prod_{v\in V(G)}T_{p(v)}\mathbb{H}^{V(G)} = (\mathbb{R}^2)^{V(G)}$,
\begin{equation}\label{eq:edgederivative}
    D_p f_e (X) = \Big\langle \dot{\gamma}_{p(u), g\cdot p(v)}(0), X(u) \Big\rangle_{p(u)} + \Big\langle \dot{\gamma}_{p(v), g^{-1}\cdot p(u)}(0), X(v) \Big\rangle_{p(v)}.
\end{equation}
Hence the derivative $D_p(f_{G,\psi})$ of the rigidity map $f_{G,\psi}$ at $p$ (with $p(v) = (x_v,y_v)$ for each vertex $v\in V$) can be represented using a matrix of the form:
\begin{equation*}
\kbordermatrix{
&&&&u&&&&v&&& \\
u\xrightarrow{g}v &0&\cdots&0& \frac{1}{y_u^2} \dot{\gamma}_{p(u), g\cdot p(v)}(0)&0&\cdots 
&0& \frac{1}{y_v^2} \dot{\gamma}_{p(v), g^{-1}\cdot p(u)}(0)&0&\cdots&0 \\
v\xrightarrow{g}v &0&\cdots&0& 0&0&\cdots 
&0&\frac{1}{y_v^2}(\dot{\gamma}_{p(v), g\cdot p(v)}(0) + \dot{\gamma}_{p(v), g^{-1}\cdot p(v)}(0)) &0&\cdots&0}.
\end{equation*}
By abuse of notation, we will often consider $D_p(f_{G,\psi})$ to be its matrix representation.
We can simplify $D_p(f_{G,\psi})$ by multiplying each column associated to $v$ by $y_v^2$ to obtain the \emph{orbit rigidity matrix} $O(G, \psi, p)$ of $(G,\psi,p)$:
\begin{equation*}
    \kbordermatrix{
    &&&&u&&&&v&&& \\
    u\xrightarrow{g}v &0&\cdots&0&  \dot{\gamma}_{p(u), g\cdot p(v)}(0)&0&\cdots 
    &0&\dot{\gamma}_{p(v), g^{-1}\cdot p(u)}(0)&0&\cdots&0 \\
    v\xrightarrow{g}v &0&\cdots&0& 0&0&\cdots 
    &0&\dot{\gamma}_{p(v), g\cdot p(v)}(0) + \dot{\gamma}_{p(v), g^{-1}\cdot p(v)}(0)&0&\cdots&0}.
\end{equation*}
Since column multiplication does not affect the image of a matrix, the matrix representation of $D_p(f_{G,\psi})$ and the orbit rigidity matrix $O(G,\psi,p)$ have the same rank and left kernel.

The rank and left kernel of $O(G, \psi, p)$ are invariant under switching operations.
For example, suppose we perform a switching at $v$ with gain $g$ (see \Cref{def:gainswitch}) to obtain a new $\Gamma$-gain graph $(G,\psi')$.
Let $\phi : \mathbb{H}^{V(G)} \rightarrow \mathbb{H}^{V(G)}$ be the map given by $\phi(p)(v) = g\cdot p(v)$ and $\phi(p)(u) = p(u)$ for each $u \neq v$. Then,
we obtain the following commutative diagram:
\begin{center}       
    \begin{tikzpicture}
        \node[] (r)at (0,0) {$\mathbb{H}^{V(G)}$};
        \node[] (p) at (4,0) {$\mathbb{H}^{V(G)}$};
        \node[] (c) at (4,-2) {$\mathbb{R}^{E(G)}$};
        \draw[->] (r) to node[above] {$\phi$} (p);
        \draw[->] (r) to node[below] {$f_{G,\psi}$} (c);
        \draw[->] (p) to node[right] {$f_{G,\psi'}$} (c);
    \end{tikzpicture}
\end{center}
Using the chain rule, we see that $D_pf_{G,\psi} = D_{\phi(p)} f_{G,\psi'} \circ D_{p} \phi$ for each $p \in \mathbb{H}^{V(G)}$,
and so (since $\phi$ is a diffeomorphism) the image of $D_pf_{G,\psi}$ is the same as the image of $D_{\phi(p)} f_{G,\psi'}$.
Hence, the matrices $O(G, \psi, p)$ and $O(G, \psi', \phi(p))$ have the same rank and left kernel. 

In the same way that there exists a one-to-one map between `trivial motions' of frameworks in $d$-dimensional Euclidean space and continuous paths in the group of $d$-dimensional Euclidean isometries, it follows from \Cref{prop:invariant} that continuous paths in the centraliser $C_{\PSL(2,\mathbb{R})}(\Gamma)$ are the correct notion for `trivial motions' in our current setting.
Similarly, the correct notion for `trivial infinitesimal motions' are the elements of the tangent space of the Lie group $C_{\PSL(2,\mathbb{R})}(\Gamma)$ at the identity, which we now denote by $\mathcal{T}_\Gamma$.
By considering the constraint conditions for smooth paths in $C_{\PSL(2,\mathbb{R})}(\Gamma)$ we obtain the following equality:
\begin{equation}\label{eq:tgamma}
    \mathcal{T}_\Gamma = \{ A \in \mathfrak{sl}(2, \mathbb{R}) : Ag = gA \mbox{ for all } g \in \Gamma \}.
\end{equation}
From this, we define the \emph{space of trivial infinitesimal motions of $(G,\psi,p)$} to be the set $$\mathcal{T}_\Gamma(p) = \left\{W_{v} \in \mathbb{R}^{2|V(G)|} : W_v= \eta(X, p(v)) \textup{ for all }X\in \mathcal{T}_\Gamma \right\}.$$ 

\begin{lemma}\label{lem:centraliser_dim}
    Let $\Gamma$ be a non-trivial Fuchsian group with centraliser $C_{\PSL(2,\mathbb{R})}(\Gamma)$.
    Then either $\Gamma$ is cyclic and $\dim C_{\PSL(2,\mathbb{R})}(\Gamma) = 1$,
    or $\dim C_{\PSL(2,\mathbb{R})}(\Gamma) = 0$.
\end{lemma}

\begin{proof}
    We first show that for any $g \in \PSL(2,\mathbb{R})$ such that $g\neq \text{id}$, $\dim \mathcal{T}_{\langle g \rangle} = 1$ (so $\dim C_{\PSL(2,\mathbb{R})}(\langle g \rangle) = 1$).
    Choose any pair
    \begin{equation*}
        g =
        \begin{bmatrix}
            g_{11} & g_{12} \\
            g_{21} & g_{22}
        \end{bmatrix} \in \PSL(2,\mathbb{R}) , \qquad 
        A =
        \begin{bmatrix}
            A_{11} & A_{12} \\
            A_{21} & -A_{22}
        \end{bmatrix}\in\mathfrak{sl}(2,\mathbb{R}).
    \end{equation*}
    By removing any repeated equations, we see that $Ag-gA=\mathbf{0}_{2 \times 2}$ if and only if
    \begin{equation*}
        - g_{21} A_{12} + g_{12} A_{21}  = 0, \qquad 2 g_{21} A_{11} + (g_{22} -g_{11}) A_{21} = 0, \qquad - 2g_{12} A_{11} + (g_{11} - g_{22}) A_{12} = 0.
    \end{equation*}
    It follows that $\dim \mathcal{T}_{\langle g \rangle}$ is exactly the nullity of the matrix
    \begin{equation*}
        M(g) :=
        \begin{bmatrix}
            0 & -g_{21} & g_{12} \\
            2 g_{21} & 0 & -(g_{11} - g_{22}) \\
            -2 g_{12} & g_{11} - g_{22} & 0
        \end{bmatrix}.
    \end{equation*}
    Since $g \in \PSL(2,\mathbb{R})$ and
    \begin{equation}\label{eq:hg is a}
        A = 
        \begin{bmatrix}
            g_{11}-g_{22} & 2g_{12} \\
            2g_{21} & -g_{11}+g_{22}
        \end{bmatrix}
    \end{equation} 
    is a solution we have $0 <\rank M(g) <  3$.
    As either $g_{12}\neq 0$, $g_{21} \neq 0$ or $g_{11} \neq g_{22}$,
    it is easy to check that $\rank M(g) \neq 1$.
    Hence $\rank M(g) = 2$ and $\dim C_{\PSL(2,\mathbb{R})}(\langle g \rangle) = 1$.

    Now suppose $\Gamma$ is not cyclic. 
    Assume, for contradiction, that $\dim C_{\PSL(2,\mathbb{R})}(\Gamma )=1$.
    By \Cref{lem:cyclic-iff-fix}, there exist $g,h \in \Gamma$ such that $\langle g,h \rangle$ is non-abelian.
    Given any $X_g \in \mathfrak{sl}(2,\mathbb{R})$ that satisfies $\exp(X_g) = g$ (guaranteed to exist by the surjectivity of $\exp$ in our setting),
    the connected subgroup $\{\exp(tX_g) : t \in \mathbb{R}\}$ contains both $\textrm{id}$ and $g$.
    Hence the connected component of $C_{\PSL(2,\mathbb{R})}(\langle g \rangle )$ containing id also contains $g$.
    Similarly, the connected component of $C_{\PSL(2,\mathbb{R})}(\langle h \rangle )$ containing id also contains $h$.
    As $C_{\PSL(2,\mathbb{R})}(\Gamma )$ is 1-dimensional and closed in both $C_{\PSL(2,\mathbb{R})}(\langle g \rangle )$ and $C_{\PSL(2,\mathbb{R})}(\langle h \rangle )$,
    its connected component containing id must also contain both $g$ and $h$.
    However, this contradicts the fact that neither $g$ nor $h$ are contained in $C_{\PSL(2,\mathbb{R})}(\Gamma )$.
\end{proof}

\begin{remark}
    We note that \Cref{lem:centraliser_dim} is not true for discrete isometry groups for the Euclidean plane.
    For example, if we choose $\Gamma$ to be the discrete subgroup of translations generated by the vectors $(1,0),(0,1)$,
    then its centraliser contains the 2-dimensional subgroup of all translations.
\end{remark}

\begin{lemma}\label{lem:centraliser_triv}
    Let $\Gamma$ be a Fuchsian group and let $(G,\psi,p)$ be a $\Gamma$-joint-configuration.
    Then $\mathcal{T}_\Gamma(p)$ is a linear subspace of $\ker D_p f_{G,\psi}$.
    Moreover, $\dim \mathcal{T}_\Gamma(p)=\dim C_{\PSL(2,\mathbb{R})}(\Gamma)$ if and only if one of the following holds:
    \begin{itemize}
        \item $\Gamma$ is non-trivial and there exist $v \in V(G)$ such that $p(v)$ is not a fixed point of $C_{\PSL(2,\mathbb{R})}(\Gamma)$,
        \item there exist $u,v \in V(G)$ such that $p(u) \neq p(v)$, or
        \item $\Gamma$ is not cyclic.
    \end{itemize}
\end{lemma}

\begin{proof}
    It follows from \Cref{eq:U} that $\mathcal{T}_\Gamma(p)$ is a linear space.
    Fix an arbitrary $A \in \mathcal{T}_\Gamma$ and choose a smooth path $A:(-1,1) \rightarrow C_{\PSL(2,\mathbb{R})}(\Gamma)$ such that $\frac{d}{dt}A(t)|_{t=0} = A$ and $A(0) = id$. We will write $A(t) = A_t$. 
    If $e=u\xrightarrow{g}v\in E(G)$, then 
    
    \begin{align*}
        D_p f_e \left(\left(\eta(A,p(u))\right)_{u \in {V(G)}}\right)& =\frac{d}{dt} d_{\mathbb{H}}\big(A_t\cdot p(u), \left(gA_t\right)\cdot p(v) \big)\Big|_{t=0} = \frac{d}{dt} d_{\mathbb{H}}\big(A_t\cdot p(u), \left(A_t g\right)\cdot p(v) \big)\Big|_{t=0} \\ &= \frac{d}{dt} d_{\mathbb{H}}\big(p(u), g\cdot p(v) \big)\Big|_{t=0} =0
    \end{align*} 
    with the final equality following from \Cref{prop:invariant}. 
    Hence $\mathcal{T}_\Gamma(p)$ is a linear subspace of $\ker D_p f_{G,\psi}$.

    Since the linear space $\mathcal{T}_\Gamma(p)$ is the tangent space of the orbit of $p$ under $C_{\PSL(2,\mathbb{R})}(\Gamma)$,
    it follows that $\dim \mathcal{T}_\Gamma(p) < \dim C_{\PSL(2,\mathbb{R})}(\Gamma)$ if and only if there exists a non-zero $A \in \mathcal{T}_\Gamma$ such that 
    $p(v)$ is a fixed point of $\exp(tA)$ for each $v\in V(G)$ and $t \in \mathbb{R}$. 
    If $p(u) \neq p(v)$ then, since any isometry can have at most one fixed point in $\mathbb{H}$, we have that at least one of $p(u)$ and $p(v)$ is not a fixed point of $C_{\PSL(2,\mathbb{R})}(\Gamma)$.
    If $\Gamma$ is not cyclic, then it has no fixed points (see \Cref{lem:cyclic-iff-fix}). 
    Finally, if $\Gamma$ is trivial and $p(u)=p(v)$ for all $u,v \in V(G)$, then we can find $A$ as described previously, and so $\dim \mathcal{T}_\Gamma(p) < \dim C_{\PSL(2,\mathbb{R})}(\Gamma)$.   
\end{proof}

\subsection{Infinitesimal rigidity and independence}

Our aim for introducing the orbit rigidity matrix is to describe a variant of rigidity.
A combination of \Cref{lem:centraliser_dim} and \Cref{lem:centraliser_triv} motivate the following definition.

\begin{definition}
    Let $\Gamma$ be a Fuchsian group and let $(G,\psi,p)$ be a $\Gamma$-joint-configuration.
    We say that $(G,\psi,p)$ (equivalently, $p$) is \emph{regular} if $\rank O(G,\psi,p) \geq \rank O(G,\psi,q)$ for all $q \in \mathbb{H}^{V(G)}$.
    Furthermore $(G,\psi,p)$ is \emph{infinitesimally $\Gamma$-symmetrically rigid} if either $|V(G)|=1$ and $\Gamma$ is trivial, or
    \begin{equation*}
        \rank O(G,\psi,p)  =
        \begin{cases}
            2|V(G)| - 3 &\text{if } \Gamma \text{ is trivial,}\\
            2|V(G)| - 1 &\text{if } \Gamma \text{ is non-trivial cyclic,}\\
            2|V(G)|  &\text{otherwise}.
        \end{cases}
    \end{equation*}
    Otherwise, we say $(G, \psi, p)$ is \emph{infinitesimally $\Gamma$-symmetrically flexible}.
\end{definition}

Infinitesimal $\Gamma$-symmetric rigidity and local $\Gamma$-symmetric rigidity are linked by the following result.

\begin{proposition}\label{prop:asimowroth}
    Let $\Gamma$ be a Fuchsian group and let $(G,\psi,p)$ be a $\Gamma$-joint-configuration.
    \begin{enumerate}
        \item If $(G,\psi,p)$ is infinitesimally $\Gamma$-symmetrically rigid, then $(G,\psi,p)$ is locally $\Gamma$-symmetrically rigid.
        \item If $(G,\psi,p)$ is regular and locally $\Gamma$-symmetrically rigid, then $(G,\psi,p)$ is infinitesimally $\Gamma$-symmetrically rigid.
    \end{enumerate}
\end{proposition}

\begin{proof}
    Since infinitesimally $\Gamma$-symmetrically rigid frameworks are regular, we may assume that $(G,\psi,p)$ is regular.
    If $\dim \mathcal{T}_\Gamma(p) = \dim C_{\PSL(2,\mathbb{R})}(\Gamma)$, then the result is a natural consequence of the constant rank theorem, and follows from a similar proof to that given in \cite{Asimow1978}. So assume $\dim \mathcal{T}_\Gamma(p) < \dim C_{\PSL(2,\mathbb{R})}(\Gamma)$. By \Cref{lem:centraliser_triv},
    $\Gamma = \langle g \rangle$ for some (possibly trivial) $g\in\Gamma$, and
    there exists some $x \in \HH$ such that $g\cdot x = x$ and $p(v) = x$ for all $v \in V(G)$. This implies that $O(G,\psi,p)$ is composed of zero rows, and therefore has rank zero. By regularity of $p$, for all $q\in\mathbb{H}^{V(G)}$, $\rank O(G,\psi,q)=0$. Therefore, $|V(G)|=1$ and $\Gamma$ is trivial. 
    Since any such framework is both infinitesimally and locally $\Gamma$-symmetrically rigid, the result follows.
\end{proof}

We can study certain rigidity properties using `equilibrium stresses', i.e.~left kernel elements of $O(G,\psi,p)$.

\begin{definition}
    Let $\Gamma$ be a Fuchsian group and $(G,\psi,p)$ be a $\Gamma$-joint-configuration. A vector $\omega \in \mathbb{R}^{E(G)}$ is an \emph{equilibrium stress of $(G,\psi,p)$} if for all $v \in V(G)$,
    \begin{equation}
    \label{bal. cond.}
        \sum_{e= v \xrightarrow{g}u, ~ v \neq u } \omega(e) \dot{\gamma}_{p(v), g\cdot p(u)}(0)  + \sum_{l= v \xrightarrow{g} v} \omega(l) \left(\dot{\gamma}_{p(v), g\cdot p(v)}(0) + \dot{\gamma}_{p(v), g^{-1}\cdot p(v)}(0) \right) = \textbf{0}.
    \end{equation}
\end{definition}

\begin{definition}
    Let $\Gamma$ be a Fuchsian group and let $(G,\psi,p)$ be a joint-configuration.
    We say $(G,\psi,p)$ is \emph{$\Gamma$-symmetrically independent} if $\rank O(G,\psi,p) =|E(G)|$ and \emph{$\Gamma$-symmetrically dependent} if it is not $\Gamma$-symmetrically independent. For brevity, we sometimes simply write \textit{$\Gamma$-independent} and \textit{$\Gamma$-dependent}.
    If $(G,\psi,p)$ is both $\Gamma$-symmetrically independent and infinitesimally $\Gamma$-symmetrically rigid,
    then $(G,\psi,p)$ is \emph{$\Gamma$-isostatic}.
\end{definition}

Notice that $(G,\psi,p)$ is $\Gamma$-symmetrically independent if and only if the rows of $O(G,\psi,p)$ are linearly independent. Equivalently, if and only if every equilibrium stress of $(G,\psi,p)$ is \emph{trivial}, i.e.~equal to the all-zeroes vector.
Moreover, $(G,\psi,p)$ is $\Gamma$-isostatic if and only if it is $\Gamma$-symmetrically infinitesimally rigid but $\Gamma$-symmetrically infinitesimally flexible when we remove an edge.
Therefore, for every $\Gamma$-symmetrically infinitesimally rigid $(G,\psi,p)$, there exists a spanning subgraph $H$ of $G$ such that $(H,\psi|_{E(H)},p)$ is $\Gamma$-isostatic. 

\subsection{Combinatorial rigidity}
\label{sec: comb. rig.}

We first note that similarly to the Euclidean case, since the expressions in the rigidity matrix are analytic, one can deduce that the set of regular realisations of a $\Gamma$-gain graph $(G,\psi)$ is a dense open subset of $\mathbb{H}^{V(G)}$.
(This follows from the locus of a real analytic function with connected domain being a closed nowhere dense subset; see, for example, \cite{Mityagin2020}.)
If we combine this with \Cref{prop:asimowroth}, then we can see that either almost all realisations of a $\Gamma$-gain graph are locally and infinitesimally $\Gamma$-symmetrically rigid,
or almost all realisations of a $\Gamma$-gain graph are locally and infinitesimally $\Gamma$-symmetrically flexible.
Moreover, the existence of a single infinitesimally $\Gamma$-symmetrically rigid realisation implies all regular realisations are also $\Gamma$-symmetrically infinitesimally rigid.

\begin{definition}
\label{defn: rigid graph}
    Let $\Gamma$ be a Fuchsian group and let $(G,\psi)$ be a $\Gamma$-gain graph.
    Then $(G,\psi)$ is \emph{$\Gamma$-symmetrically rigid in $\mathbb{H}$} if there exists an infinitesimally $\Gamma$-symmetrically rigid joint-configuration $(G,\psi,p)$.
    Otherwise, $(G,\psi)$ is \emph{$\Gamma$-symmetrically flexible in $\mathbb{H}$}.
    Furthermore, $(G,\psi)$ is \emph{$\Gamma$-symmetrically independent in $\mathbb{H}$} if there exists a $\Gamma$-symmetrically independent joint-configuration $(G,\psi,p)$.
    If $(G,\psi)$ is both $\Gamma$-symmetrically independent and rigid in $\mathbb{H}$,
    then we say $(G,\psi)$ is \emph{$\Gamma$-isostatic in $\mathbb{H}$}. 
\end{definition}

Since we only work with frameworks in $\mathbb{H}$, we often drop the `in $\mathbb{H}$' from the terminology in \Cref{defn: rigid graph}.

\begin{lemma}\label{lem:independence_graphs}
    Let $\Gamma$ be a Fuchsian group and $(G,\psi)$ be a $\Gamma$-gain graph. Suppose $(G,\psi)$ is $\Gamma$-symmetrically independent. Then $(G,\psi)$ is $(2,3,1,0)$-gain sparse.
\end{lemma}

\begin{proof}
    If $|V(G)|=1$ and $\Gamma$ is trivial,
    then $G$ has no edges and the result is immediate.
    If $G$ is disconnected, then we consider each connected component: a simple combinatorial argument shows that the gain graph $(G,\psi)$ is $(2,3,1,0)$-gain sparse if and only if each connected component of $G$ is $(2,3,1,0)$-gain sparse.
    Hence we may assume $G$ is connected and either $|V(G)|\geq 1$ or $\Gamma$ is non-trivial.
    Fix $\Gamma' = \langle G \rangle_{v, \psi}$ for any choice of $v \in V(G)$.
    Since either $|V(G)|\geq 1$ or $\Gamma$ is non-trivial,
    we may choose a realisation $p$ such that for some $u,v \in V(G),g \in \Gamma$ we have $p(u) \neq g\cdot p(v)$.
    By \Cref{lem:centraliser_triv}, $|E(G)| \leq 2|V(G)|-\dim C_{\PSL(2,\mathbb{R})}(\Gamma')$.
    Hence, by \Cref{lem:centraliser_dim}, we have $|E(G)| \leq 2|V(G)|-k$,
    where $k=3$ if $\Gamma'$ is trivial, $k = 1$ is $\Gamma'$ is non-trivial and cyclic and $k = 0$ is $\Gamma'$ is not cyclic.
    Thus, $(G,\psi)$ is $(2,3,1,0)$-sparse.
\end{proof}

This implies the following result.

\begin{theorem}\label{thm:necessary}
    Let $\Gamma$ be a Fuchsian group and let $(G,\psi)$ be a $\Gamma$-gain graph which is $\Gamma$-isostatic.
    \begin{enumerate}
        \item If $\Gamma$ is trivial, then $G$ is $(2,3)$-tight.
        \item If $\Gamma$ is non-trivial and cyclic, then $(G,\psi)$ is $(2,3,1)$-gain tight.
        \item If $\Gamma$ is not cyclic, then $(G,\psi)$ is $(2,3,1,0)$-gain tight.
    \end{enumerate}
\end{theorem}

It was shown in \cite{I09,saliola_whiteley_2007_equivalence_rigidity} that generic rigidity in $\HH$ is equivalent to generic rigidity in $\RR^2$. Since generic rigidity is combinatorially characterised in $\RR^2$ \cite{Laman1970,HPG27}, combining gives the following result which shows the converse to \Cref{thm:necessary} holds in case (i). 
In the rest of the paper we show \Cref{thm:necessary} also holds in case (ii) and case (iii) when $\Gamma$ is a surface group.

\begin{theorem}\label{thm:Hlaman}
    Let $\Gamma$ be the trivial group. Then, a $\Gamma$-gain graph $(G,\psi)$ is $\Gamma$-isostatic if and only if it is $(2,3)$-tight. 
\end{theorem}

\section{Combinatorial characterisation for cyclic Fuchsian groups}\label{sec:combchar}
In this section we prove the following result.

\begin{theorem}\label{thm:simple_extension}
Let $\Gamma$ be a Fuchsian group and $(G,\psi)$ be a $\Gamma$-isostatic $\Gamma$-gain graph. Suppose $(G',\psi')$ is obtained from $(G,\psi)$ by applying a 0-extension, 1-extension or loop-1-extension.
Then $(G',\psi')$ is $\Gamma$-isostatic.
\end{theorem}

Using \Cref{thm:simple_extension},
we can prove the cyclic group variant of \Cref{main_theorem}. 
For this section, we additionally make extensive use of \Cref{lem:linear_independence}, which we restate below.

\keylemma*

We first require the following.

\begin{lemma}\label{lem: bad-geodesics}
    For every non-trivial $g \in \PSL(2,\RR)$, there exists a dense open set of points $z \in \HH$ such that
    \begin{equation*}
        \dot{\gamma}_{z, g\cdot z}(0)+ \dot{\gamma}_{z, g^{-1}\cdot z}(0) \neq \textup{\textbf{0}}.
    \end{equation*}
\end{lemma}

\begin{proof}
    We recall that the zero set of a real analytic function with open connected domain is either the entire domain or nowhere dense (see, for example, \cite{Mityagin2020}).
    Since the map $z \mapsto \dot{\gamma}_{z, g\cdot z}(0)+ \dot{\gamma}_{z, g^{-1}\cdot z}(0)$ is a real analytic function on a connected open set,
    it suffices to find one such point where the map is non-zero.
    
    If $g$ has order two, then $\dot{\gamma}_{z, g\cdot z}(0) = \dot{\gamma}_{z, g^{-1}\cdot z}(0)$, and choosing $z$ to be any non-fixed point of $g$, we have $\dot{\gamma}_{z, g\cdot z}(0)+\dot{\gamma}_{z, g^{-1}\cdot z}(0)\neq\textbf{0}$. So assume that $g$ does not have order two.
    By \Cref{lem:linear_independence},
    if $z$ is not a fixed point of $g$, i.e.~$g\cdot z\neq z$, then
    \begin{equation*}
        \dot{\gamma}_{z, g\cdot z}(0)+ \dot{\gamma}_{z, g^{-1}\cdot z}(0) = \textbf{0} \qquad \implies \qquad z,g\cdot z,g^{-1}\cdot z \text{ lie on a common geodesic.}
    \end{equation*}
    So it suffices to show that $z$ can be chosen such that $z, g\cdot z$ and $g^{-1}\cdot z$ do not lie on a common geodesic. Suppose, for a contradiction, that $g^{-1}\cdot z$ lies on the geodesic $\gamma_{z,g\cdot z}$ for each $z \in \HH$ that is not a fixed point of $g$;
    equivalently, $\gamma_{z,g\cdot z}$ and $\gamma_{z,g^{-1}\cdot z}$ contain the same points.
    Then by \Cref{continuity-geodesics} -- combined with the observation that a sequence $(g\cdot z_n)_{n \in \NN}$ in $\HH$ converges to the boundary if $(z_n)_{n \in \NN}$ converges to the boundary --
    we have that the geodesic $(x,g\cdot x) \in \textup{Geod}(\HH)$ is equal to either $(g^{-1}\cdot x,x)$ or $(x,g^{-1}\cdot x)$ for every boundary point $x \in \RR \,\cup\, \{\infty\}$.
    Suppose that $(x,g\cdot x) = (g^{-1}\cdot x,x)$.
    Then 
    \begin{equation*}
        g \cdot (x,g\cdot x) = g \cdot (g^{-1}\cdot x,x) = (x,g\cdot x),
    \end{equation*}
    and so $x$ is a fixed point of $g$.
    Since $g$ is non-trivial, we thus have that $(x,g\cdot x) = (x,g^{-1}\cdot x)$ for all points $x \in \RR \,\cup\, \{\infty\}$ that are not fixed points of $g$. Then, $g\cdot g x = g^{2}\cdot x = x$ for all $x\in \RR \,\cup\, \{\infty\}$, i.e.~$g^2=\text{id}$, contradicting the fact that $g$ does not have order two. This concludes our proof.
\end{proof}

\begin{theorem}\label{cyclic_main_theorem}
    Let $\Gamma$ be a cyclic Fuchsian group and $(G,\psi)$ be a $\Gamma$-gain graph.
    Then $(G,\psi)$ is $\Gamma$-isostatic if and only if it is $(2,3,1)$-gain tight. 
\end{theorem}

\begin{proof}
Necessity follows from \Cref{thm:necessary}. By \Cref{thm: recur. con. easy}, $(G,\psi)$ is obtained from a single vertex with an unbalanced loop, say $(G',\psi')$, by applying a series of 0-extensions, 1-extensions and loop-1-extensions. By \Cref{lem: bad-geodesics}, there exists some $x\in \mathbb{H}$ with $\dot{\gamma}_{x,g\cdot x}(0) + \dot{\gamma}_{x,g^{-1}\cdot x}(0) \neq \textbf{0}$, and so $(G',\psi')$ is $\Gamma$-isostatic. Therefore, by \Cref{thm:simple_extension}, $(G,\psi)$ is $\Gamma$-isostatic.
\end{proof}

\subsection{0-extensions and loop-1-extensions}
We first show 0-extensions preserve both $\Gamma$-independence and $\Gamma$-dependence. 

\begin{lemma}\label{lem:0ext}
Let $\Gamma$ be a Fuchsian group, $(G,\psi)$ be a $\Gamma$-gain graph,
and $(G',\psi')$ be obtained from $(G,\psi)$ by applying a 0-extension.
Then $(G,\psi)$ is $\Gamma$-independent if and only if $(G',\psi')$ is $\Gamma$-independent.
\end{lemma}
\begin{proof}
Since $(G,\psi)$ is contained in $(G',\psi')$ as a $\Gamma$-gain subgraph, if $(G',\psi')$ is $\Gamma$-independent then $(G,\psi)$ is $\Gamma$-independent.
Choose a regular realisation $p$ of $(G,\psi)$ such that $g_1\cdot p(u_1)\neq g_2\cdot p(u_2)$, which can be achieved since there is a dense open subset of regular configurations and $g_1\cdot p(u_1)\neq g_2\cdot p(u_2)$ defines a dense open subset of $p \in \mathbb{H}^{|V(G)|}$. Suppose that our $0$-extension adds the vertex $v$ and edges $e_1 = v \xrightarrow[]{g_1} u_1$ and $e_2 = v \xrightarrow[]{g_2} u_2$ with $g_1 \neq g_2$ if $u_1=u_2$. 
Take $p':V(G')\rightarrow \HH$ such that $p'(w)=p(w)$ for all $w\in V(G)$ and $p'(v)$ does not lie on the geodesic between $g_1\cdot p(u_1)$ and $g_2\cdot p(u_2)$. 
Then
\begin{equation*}
O(G',\psi',p')=\left[ 
    \begin{array}{c|c}
        \begin{matrix}
        \dot{\gamma}_{p'(v), g_1\cdot p(u_1)}(0)\\
        \dot{\gamma}_{p'(v), g_2\cdot p(u_2)}(0)
        \end{matrix}

        &\begin{matrix}
            0 ~ \cdots ~ 0 & \dot{\gamma}_{p(u_1), g_1^{-1}\cdot p'(v)}(0) & 0 ~ \cdots ~ 0 & \textbf{0}  & 0 ~ \cdots ~ 0 \\
           0 ~ \cdots ~ 0 & \textbf{0} & 0 ~ \cdots ~ 0 & \dot{\gamma}_{p(u_2), g_1^{-1}\cdot p'(v)}(0) & 0 ~ \cdots ~ 0 
        \end{matrix} \\
        \hline \\
        \mathbf{0}_{|E(G)| \times 2} & O(G,\psi,p) \\ \\
    \end{array}
    \right].
\end{equation*}
By \Cref{lem:linear_independence}, the vectors $\dot{\gamma}_{p'(v), g_1\cdot p(u_1)}(0), \dot{\gamma}_{p'(v), g_2\cdot p(u_2)}(0)$ are linearly independent.
Hence,
\begin{equation*}
    \rank O(G',\psi',p') = \rank O(G,\psi,p) + 2.
\end{equation*}
Since $p$ is regular and $|E(G')|= |E(G)|+2$,
it follows that $(G',\psi',p')$ is $\Gamma$-independent whenever $(G,\psi)$ is $\Gamma$-independent.
\end{proof}

The proof of \Cref{lem:0ext} can be adapted to show that loop-1-extensions also preserve preserve both $\Gamma$-independence and $\Gamma$-dependence.

\begin{lemma}\label{lem:loop1ext}
Let $\Gamma$ be a Fuchsian group, $(G,\psi)$ be a $\Gamma$-gain graph,
and $(G',\psi')$ be obtained from $(G,\psi)$ by applying a loop-1-extension.
Then $(G,\psi)$ is $\Gamma$-independent if and only if $(G',\psi')$ is $\Gamma$-independent.
\end{lemma}

\begin{proof}
Since $(G,\psi)$ is contained in $(G',\psi')$ as a $\Gamma$-gain subgraph, if $(G',\psi')$ is $\Gamma$-independent then $(G,\psi)$ is $\Gamma$-independent.
Suppose that $(G,\psi)$ is $\Gamma$-independent.
Choose a regular realisation $p$ of $(G,\psi)$.
Suppose that our loop-1-extension adds the vertex $v$, the edge $e = v \xrightarrow[]{h} u$, and the loop $l = v \xrightarrow[]{g} v$ with $g \neq\text{id}$. 
By \Cref{lem:treegain}, we may assume $h = \text{id}$.

For each point $x \in \mathbb{H}$,
define a realisation $p_x:V(G')\rightarrow \HH$ such that $p_x(w)=p(w)$ for all $w\in V(G)$ and $p_x(v) = x$. 
Then
\begin{equation*}
O(G',\psi',p_x)=\left[ 
    \begin{array}{c|c}
    \begin{matrix}
    \dot{\gamma}_{x, p(u)}(0)\\
    \dot{\gamma}_{x, g\cdot x}(0) + \dot{\gamma}_{x, g^{-1}\cdot x}(0)
    \end{matrix}
    &\begin{matrix}
        0 ~ \cdots ~ 0 & \dot{\gamma}_{p(u),x}(0) & 0 ~ \cdots ~ 0 \\
        0 ~ \cdots ~ 0 & \textbf{0} & 0 ~ \cdots ~ 0 \\
    \end{matrix} \\
    \hline \\
    \mathbf{0}_{|E(G)| \times 2} & O(G,\psi,p) \\
    \end{array}
    \right].
\end{equation*}
Since $p$ is regular and $(G,\psi)$ is $\Gamma$-independent, $(G',\psi',p_x)$ is $\Gamma$-independent if, for some $x \in \HH$, the vectors 
\begin{equation*}
    z_1(x) := \dot{\gamma}_{x, p(u)}(0), \qquad  z_2(x) := \dot{\gamma}_{x, g\cdot x}(0) + \dot{\gamma}_{x, g^{-1}\cdot x}(0)
\end{equation*}
are defined and linearly independent. We may assume that $z_2(p(u))\neq \textbf{0}$, since $z_2$ is defined on any point in $\HH$ that is not a fixed point for $g$, of which there are at most one, and the set of regular frameworks is dense and open. 
Now choose $w \in \RR^2$ satisfying $\langle w, w \rangle_{p(u)}=1$ which is linearly independent from $z_2(p(u))$.
Given $\gamma_{p(u):w}$ is the geodesic described in \Cref{def:unitspeedbydirection},
we have that the limits
\begin{equation*}
    \lim_{t \rightarrow 0^+} z_1\left( \gamma_{p(u):w} (t) \right) = -\dot{\gamma}_{p(u):w} (0)= - w, \qquad \lim_{t \rightarrow 0^+} z_2\left( \gamma_{p(u):w} (t) \right) = z_2(p(u))
\end{equation*}
are linearly independent.
Hence, there exists $\varepsilon >0$ such that, if we set $x=\gamma_{p(u):w}(\varepsilon)$,
then $z_1(x), z_2(x)$ are defined and linearly independent.
This concludes the proof.
\end{proof}

\subsection{1-extension}

We now prove that 1-extensions preserve $\Gamma$-isostaticity. We first need the following lemma.
\begin{lemma} \label{key_lemma_loop2ext}
    If $x, y, z\in \mathbb{H}$ are distinct points lying on a common geodesic,
    then the matrix
    \begin{align*}
    \begin{bmatrix}
        \dot{\gamma}_{x, y}(0) & \dot{\gamma}_{y, x}(0) & \textup{\textbf{0}}\\
        \dot{\gamma}_{x, z}(0) & \textup{\textbf{0}} & \dot{\gamma}_{z, x}(0)\\
        \textup{\textbf{0}} & \dot{\gamma}_{y, z}(0) &  \dot{\gamma}_{z, y}(0)\\
        \end{bmatrix}
    \end{align*}
    has a non-trivial row-dependence supported on all rows.
    Similarly, if $x,y \in \HH$ and $g\in \PSL(2, \mathbb{R})$ is an isometry for which $x,y,g\cdot x\in\HH$ are distinct points lying on a common geodesic,
    then the matrix
    \begin{align*}
    \begin{bmatrix}
        \dot{\gamma}_{x, y}(0) & \dot{\gamma}_{y, x}(0)\\
        \dot{\gamma}_{x, g^{-1}\cdot y}(0) & \dot{\gamma}_{y, g\cdot x}(0)\\
        \dot{\gamma}_{x, g\cdot x}(0) +  \dot{\gamma}_{x, g^{-1}\cdot x}(0) &  \textup{\textbf{0}}\\
        \end{bmatrix}
    \end{align*}
    has a non-trivial row dependence supported on all rows.
\end{lemma}
\begin{proof}
    We prove only the second case since the first case is similar (yet easier). We use similar reasoning as in the proof of \Cref{lem:linear_independence}. If $x,y,g\cdot x$ lie on a geodesic, then $\dot{\gamma}_{x, g\cdot x}(0) = \varepsilon_1 \dot{\gamma}_{x, y}(0)$, $\dot{\gamma}_{g\cdot x, y}(0) =\varepsilon_2 \dot{\gamma}_{g\cdot x, x}(0)$ and $\dot{\gamma}_{y, x}(0) =\varepsilon_3 \dot{\gamma}_{y, g\cdot x}$, where $\varepsilon_i \in \{\pm 1\}$, and exactly one of $\varepsilon_1, \varepsilon_2, \varepsilon_3$ is equal to $-1$. We thus also have $\dot{\gamma}_{x, g^{-1}\cdot y}(0) =\varepsilon_2 \dot{\gamma}_{x, g^{-1}\cdot x}(0)$. Denote the rows of the matrix by $R_1, R_2$ and $R_3$. One may then verify that $\varepsilon_1 R_1 + \varepsilon_2 R_2 - R_3 = \textbf{0}$.
\end{proof}

\begin{lemma}\label{lem:1ext}
Let $\Gamma$ be a Fuchsian group, $(G,\psi)$ a $\Gamma$-gain graph,
and $(G',\psi')$ be obtained from $(G,\psi)$ by applying a 1-extension.
If $(G,\psi)$ is $\Gamma$-isostatic, then $(G',\psi')$ is $\Gamma$-isostatic.
\end{lemma}

\begin{proof}
Suppose that the 1-extension removes an existing edge $e = u_1 \xrightarrow[]{g} u_2$ (allowing for $u_1= u_2$) and adds a new vertex $v$, together with the edges $e_1,e_2,e_3$, where $e_i = v \xrightarrow[]{g_i} u_i$ for some $u_3\in V(G)$.
By \Cref{def:1ext}, $g=g_1^{-1}g_2$.

\begin{claim}\label{claim: lem-1ext}
    There is a choice of $p\in\mathbb{H}^{V(G)}$ such that the points $g_1\cdot p(u_1),g_2\cdot p(u_2),g_3\cdot p(u_3)$ do not lie on a common geodesic. 
\end{claim}

\begin{proof}
    First suppose that $u_1$ is distinct from $u_2,u_3$.
    If there is some $p\in\mathbb{H}^{V(G)}$ such that $g_2\cdot p(u_2) \neq g_3\cdot p(u_3)$, then $p(u_1)$ can be chosen so that $g_1\cdot p(u_1)$ does not lie on the geodesic through $g_2\cdot p(u_2)$ and $g_3\cdot p(u_3)$. 
    So assume that $g_2\cdot p(u_2)=g_3\cdot p(u_3)$ for all choices of $p\in\mathbb{H}^{V(G)}$.
    Then $u_2=u_3$ and $(g_2^{-1}g_3)\cdot p(u_2)=p(u_2)$ for all choices of $p\in\mathbb{H}^{V(G)}$.
    As non-trivial isometries have at most one fixed point, we have $g_1=g_2$, contradicting \Cref{def:1ext}. 
    
    Now suppose $u_1=u_2=u_3$. By \Cref{lem:treegain}, we may assume that $g_3=\text{id}$. Assume, for a contradiction, that $z,g_1\cdot z,g_2\cdot z$ lie on a common geodesic for all $z\in\mathbb{H}$.
    It follows by \Cref{continuity-geodesics} that for any $z \in \RR \,\cup\, \{\infty\}$, we have $g_1\cdot z =z$, $g_2\cdot z = z$ or $g_1\cdot z= g_2\cdot z$.
    If for some $i \in \{1,2\}$ the equality $g_i\cdot z= z$ holds for more than three boundary points $z$, then, since an isometry can fix at most two boundary points, we have $g_i=\text{id}$.
    Similarly, if $g_1 \cdot z = g_2 \cdot z$ for more than three choices of $z$, then $g_1^{-1} g_2 = \text{id}$. In both cases, we contradict \Cref{def:1ext}.
\end{proof}

Now choose $p':V(G)\rightarrow \HH$ such that $p'(w)=p(w)$ for all $w\in V(G)$, and $p'(v)$ lies on the geodesic through $g_1\cdot p(u_1)$ and $g_2\cdot p(u_2)$ but is neither $g_1\cdot p(u_1)$ nor $g_2\cdot p(u_2)$.
We also define $G_1 := G' + e - e_2$ and $G_2 := G' + e$, $\psi_1;E(G_1)\rightarrow\Gamma$ to be the map with $\psi_1(f) = \psi'(f)$ for each $f \neq e$ and $\psi_1(e) = \psi(e)$, and $\psi_2:E(G_2)\rightarrow\Gamma$ to be the map with $\psi_2(f) = \psi_1(f)$ for each $f \neq e$ and $\psi_2(e) = \psi(e)$.
With this, we observe that $(G_1,\psi_1)$ is formed by applying a 0-extension to $(G,\psi)$, and $(G_2,\psi_2)$ is formed from $(G',\psi')$ by adding back in the edge $e$.

By \Cref{lem:linear_independence}, $(G_1,\psi_1,p')$ is $\Gamma$-isostatic.
Thus $(G_2,\psi_2)$ is $\Gamma$-symmetrically infinitesimally rigid and $O(G_2,\psi_2,p')$ has a 1-dimensional left kernel. By \Cref{key_lemma_loop2ext}, the set $\{e,e_1,e_2\}$ has a linear dependence, and thus it is the unique circuit of the row matroid of $O(G_2,\psi_2,p')$. It follows that removing the row corresponding to $e$ from $O(G_2,\psi_2,p')$ does not reduce the rank.
Hence $\rank O(G',\psi',p') =2|V(G')|-1$ and so $(G',\psi')$ is $\Gamma$-isostatic.
\end{proof}

\section{Irreducible (2,3,1,0)-gain graphs are \texorpdfstring{$\Gamma$}{Gamma}-isostatic}
\label{sec: irr.}
A $(2,3,1,0)$-gain tight graph is \textit{irreducible} if no reduction operation applied to it yields a $(2,3,1,0)$-gain tight graph. In this section we show that, for certain Fuchsian groups $\Gamma$, any irreducible $(2,3,1,0)$-gain tight graph is $\Gamma$-isostatic.

\begin{theorem}\label{main_irreducible}
    Let $\Gamma$ be a Fuchsian group that is isomorphic to either a surface group or a non-cyclic finitely generated free group.
    If $(G,\psi)$ is an irreducible $(2,3,1,0)$-gain tight $\Gamma$-gain graph, then $(G,\psi)$ is $\Gamma$-isostatic.
\end{theorem}

When $\Gamma$ is a non-cyclic Fuchsian group, \Cref{thm: recur. con.} implies that every irreducible $(2,3,1,0)$-gain tight $\Gamma$-gain graph is a base graph. However this is not the end of the story.
Recall, from \Cref{subsec:basegraphs}, that there are three classes of base graph: the class consisting of trivial graphs, fancy triangles and fancy hats (see \Cref{fig:base}(a),(c) and (d)); the class consisting of double cycles (see, e.g., \Cref{fig:base}(b)); and the class of near-cyclic graphs. 
The first and third classes are irreducible and \Cref{subsec:triv,{subsec:nearcyclic}} show that each $\Gamma$-gain graph in these classes is $\Gamma$-isostatic. However the class of double cycles is not, and the proof of \Cref{main_irreducible} is completed in \Cref{sec: DC} by showing that all such gain graphs are reducible.

\subsection{Trivial graphs, fancy triangles and fancy hats}
\label{subsec:triv}

We require the following result.
\begin{lemma}\label{lem:loopequation1}
    For each $g \in \PSL(2,\mathbb{R}) \setminus \{\textup{id}\}$, $X_g\in \mathcal{T}_{\langle g \rangle}$ and $x \in \mathbb{H}$ such that $x\neq g\cdot x$, we have
    \begin{equation*}
        \Big\langle \dot{\gamma}_{x,g\cdot x}(0) + \dot{\gamma}_{x,g^{-1}\cdot x}(0) ~ , ~ \eta (X_g,x) \Big\rangle_{x} = 0.
    \end{equation*}
\end{lemma}
\begin{proof}
    Since $g$ commutes with $\exp(t X_g)$ for any $t \in \mathbb{R}$, we have that
    \begin{equation*}
        d_{\mathbb{H}}(\exp(tX_g) x, g\cdot \exp(tX_g) x) = d_{\mathbb{H}}(\exp(tX_g) x, \exp(t X_g) g\cdot x) = d_{\mathbb{H}}(x, g\cdot x).
    \end{equation*}
    It follows that
    \begin{equation*}
        \Big\langle \dot{\gamma}_{x,g\cdot x}(0) + \dot{\gamma}_{x,g^{-1}\cdot x}(0) ~ , ~  \eta (X_g,x)  \Big\rangle_x = \frac{d}{dt} d_{\mathbb{H}}\Big(\exp(tX_g) \cdot x, g \exp(tX_g)\cdot x \Big) \Big|_{t=0} = \frac{d}{dt} d_{\mathbb{H}}(x, g\cdot x) \Big|_{t=0} = 0. \qedhere
    \end{equation*}
\end{proof}

We now show that every trivial $\Gamma$-gain graph is $\Gamma$-isostatic.

\begin{lemma}\label{lem:trivial}
    Let $\Gamma$ be a non-cyclic Fuchsian group. Then every trivial $\Gamma$-gain graph is $\Gamma$-isostatic.
\end{lemma}

\begin{proof}
    By sparsity, $G$ has two (loop) edges $e=v\xrightarrow[]{g}v$ and $f=v\xrightarrow[]{h}v$. Using \Cref{lem: bad-geodesics} and \Cref{lem: Lie -algebras-stabilizers0}, we choose $p$ such that 
\begin{equation*}
    \gamma_{p(v),g\cdot p(v)}(0)+\dot\gamma_{p(v),g^{-1}\cdot p(v)}(0)\neq\textbf{0},\qquad \gamma_{p(v),h\cdot p(v)}(0)+\dot\gamma_{p(v),h^{-1}\cdot p(v)}(0)\neq\textbf{0}
\end{equation*}
and $\mathfrak{h}_{p(v)}\neq\mathcal{T}_{\langle g\rangle},\mathcal{T}_{\langle h\rangle}$. Fix $k\in\{g,h\}$.
It is easy to verify that, for any choice of $p(v)$, the map 
\begin{equation*}
    f: \mathfrak{sl}(2, \mathbb{R}) \longrightarrow T_{p(v)}\mathbb{H},\qquad X\longmapsto \eta(X, p(v))
\end{equation*}
is surjective and has kernel $\mathfrak{h}_{p(v)}$. Moreover, by our choice of $p(v)$, the map
\begin{equation*}
    \alpha_k : T_{p(v)}\HH\longrightarrow\RR,\qquad u\longmapsto \Big\langle \dot\gamma_{p(v),k\cdot p(v)}(0)+\dot\gamma_{p(v),k^{-1}\cdot p(v)}(0)~,~u\Big\rangle_{p(v)}
\end{equation*}
is surjective. Therefore, the composition $\Phi_k = \alpha_k \circ f:\mathfrak{sl}(2,\RR)\rightarrow\RR$ is a surjective map and therefore $\text{null}\,\Phi_k=2$.
We observe here that the kernel of $D_p f_{G,\psi}$ is exactly the intersection $\ker \alpha_g \,\cap\, \ker \alpha_h$,
and the preimage of $\ker \alpha_g $ (respectively, $\ker \alpha_h$) under $f$ is exactly $\ker \Phi_g $ (respectively, $\ker \Phi_h$).
Since $\ker f=\mathfrak{h}_{p(v)}$, it now suffices to show that $\ker \Phi_g \,\cap\, \ker \Phi_h = \mathfrak{h}_{p(v)}$.

Since $\text{null}\,\Phi_k=2$ is surjective, any 2-dimensional subspace of $\ker\Phi_k$ is exactly $\ker\Phi_k$. 
As $\ker f \subset \ker\Phi_k$,
we have $\mathfrak{h}_{p(v)} \subset \ker\Phi_k$.
Additionally, $\mathcal{T}_{\langle k \rangle} \subset \ker\Phi_k$ by \Cref{lem:loopequation1}.
Thus $\mathcal{T}_{\langle k \rangle}+ \mathfrak{h}_{p(v)}$ is contained in $\ker\Phi_k$.
By our choice of $p$, 
$\mathcal{T}_{\langle k \rangle}+ \mathfrak{h}_{p(v)}$ is 2-dimensional, and therefore it is exactly $\ker \Phi_k$.

Since $\langle g,h\rangle$ is non-cyclic,
the linear space $\mathcal{T}_{\langle g, h\rangle}$ is 0-dimensional by \Cref{lem:centraliser_dim}.
It is simple to check that $\mathcal{T}_{\langle g, h\rangle} = \mathcal{T}_{\langle g\rangle} \cap \mathcal{T}_{\langle h\rangle}$.
Therefore,
\begin{equation*}
    \ker \Phi_g \,\cap\, \ker \Phi_h=\left(\mathcal{T}_{\langle g \rangle}+ \mathfrak{h}_{p(v)}\right)\,\cap\,\left(\mathcal{T}_{\langle h \rangle}+ \mathfrak{h}_{p(v)}\right) = \mathfrak{h}_{p(v)},
\end{equation*}
as required. 
\end{proof}

In the rest of this subsection, we provide an analogous result for fancy hats and fancy triangles. Note that, in both cases, we have a graph obtained from a $(2,3)$-tight graph by adding loops $e_1, e_2$ and $e_3$ at three distinct vertices. We consider both cases together in \Cref{base_case_lemma}. The general strategy is as follows. We know that the trivial infinitesimal motions of the balanced graph are all given by $\eta(X, p)$. Hence, it suffices to consider the subsets $M_i \subseteq \mathfrak{sl}(2, \mathbb{R})$ of these infinitesimal motions which are also infinitesimal motions of the balanced graph with one added loop $e_i$. If $M_1 \,\cap\, M_2\,\cap\, M_3 =\{\textbf{0}_{2\times 2}\}$, then the resulting graph will be independent. We will use similar a strategy for near-cyclic graphs (see \Cref{lem:near}). 

Before we begin the next lemma, we recall the following terminology.
Given a real vector space $V$,
we define the projectivisation $\mathbb{P}(V)$ of $V$ to be the quotient of $V\setminus \{\mathbf{0}\}$ by the equivalence relation $x \sim y$ if and only if $x = \lambda y$ for some real non-zero scalar $\lambda$.
We note that $\mathbb{P}(V)$ can also be considered to be the set of 1-dimensional lines in $V$: any line $\{t x : t \in \RR\}$ can be uniquely identified with the equivalence class $[x]$ in $\PP(V)$.
If $V = \mathbb{R}^n$, then we use $\RR \mathbb{P}^{n-1}$ to denote $\PP(\RR^n)$.

\begin{lemma}\label{lin_alg_lemma}
Let $V_1, V_2$ and $V_3$ be one-dimensional subspaces of $\mathbb{R}^{3}$ with 
\begin{equation*}
\bigcap_{i=1}^{3}V_i =\{\textup{\textbf{0}}\}
\end{equation*} 
and $O$ be a non-empty open subset of $(\mathbb{RP}^{2})^3$.
Then there exists a non-empty open subset $O' \subset O$ of triples $(W_1,W_2,W_3)$ such that (when considering each $W_i$ to be a 1-dimensional linear subspace of $\RR^3$) we have $W_i \neq V_i$ for each $i \in \{1,2,3\}$ and 
\begin{equation}\label{eq:lin_alg_lemma}
\bigcap_{i=1}^{3} (V_i + W_i) = \{\textup{\textbf{0}}\}.
\end{equation}
\end{lemma}

\begin{proof}
    We use homogeneous coordinates for $V_i$ and $W_i$, so we may write $V_i = [\alpha_i],W_i=[\beta_i]$, with $\alpha_{i},\beta_i \in \mathbb{R}^{3}$. Note that \Cref{eq:lin_alg_lemma} holds if and only if
    \begin{equation*}
    \det
    \begin{bmatrix}
        (\alpha_{1} \times \beta_{1})^T\\
        (\alpha_{2} \times \beta_{2})^T\\
       (\alpha_{3} \times \beta_{3})^T\\
    \end{bmatrix}\neq 0,
    \end{equation*}
    where $\times$ denotes the cross product.
    This is a tri-homogeneous polynomial $Q(\beta_{1}, \beta_{2}, \beta_{3})$. Thus, if we can find some point at which the polynomial does not vanish, there is a Zariski open subset of $(\mathbb{RP}^{2})^{3}$ where the polynomial does not vanish. Since Zariski open subsets are dense open for the usual topology on $\mathbb{RP}^{2}$, the intersection with $O \,\cap\, \{W_1, W_2, W_3 \in \mathbb{RP}^{2}~\vert~ W_i \neq V_i\}$ is non-empty and open; the lemma then follows.
      We now show that there exists some point where $Q(\beta_1, \beta_2, \beta_3)$ does not vanish. First, suppose that $\mathbb{R}^3 = V_1 + V_2 + V_3$.
    If we pick $W_1=V_2$, $W_2=V_3$, and $W_3 = V_1$, then (since each $V_i$ does not lie in the span of the other two) we have that
    \begin{equation*}
        \bigcap_{i=1}^{3} (V_i + W_i) = (V_1 + V_2) \,\cap\, (V_2 + V_3) \,\cap\, (V_3 + V_1) = \{\textbf{0}\}.
    \end{equation*}
    Now suppose that $V_1 + V_2 + V_3$ is 2-dimensional. Without loss of generality we have 
    \begin{equation*}
    \alpha_1 = (1, 0, 0), \qquad 
        \alpha_2 = (0, 1, 0), \qquad 
        \alpha_3 = (\lambda_1, \lambda_2, 0)
    \end{equation*}
    for some $\lambda_1, \lambda_2 \in \mathbb{R}$, not both zero. Then, picking
        \begin{equation*}
        \beta_1 = (0, 0, 1), \qquad
        \beta_2 = (0, 0, 1), \qquad
        \beta_3 = (-\lambda_2, \lambda_1, 0)
    \end{equation*}
    yields a point with $Q(\beta_1, \beta_2, \beta_3)=\lambda_1^2+\lambda_2^2\neq 0$. This concludes our proof.
\end{proof}

\begin{lemma}\label{base_case_lemma}
Let $\Gamma$ be a non-cyclic Fuchsian group and $(G,\psi)$ be a $(2,3)$-tight balanced $\Gamma$-gain graph. Let $(G',\psi')$ be obtained from $(G,\psi)$ by choosing three distinct vertices $v_1,v_2,v_3\in V(G)$ and adding $v_i \xrightarrow{g_i} v_i$ for $i\in \{1,2,3\}$. If $\left<g_1,g_2,g_3\right>$ is not cyclic, then $(G',\psi')$ is $\Gamma$-isostatic.
\end{lemma}

\begin{proof}
By \Cref{lem:treegain} and \Cref{prop:balgain}, we may assume that each edge of $(G,\psi)$ has gain $\text{id}$. 
Since $G$ is $(2,3)$-tight, it follows from \Cref{thm:Hlaman} that there a exists dense open subset $U \subseteq \mathbb{H}^{V(G)}$ such that for all $p\in U$ and $u\in (\mathbb{R}^2)^{V(G)}$,  $u\in\ker D_p f_{G,\psi}$ if and only if $u = \eta(X,p)$ for some $X \in \mathfrak{sl}(2,\mathbb{R})$.
For $1\leq i\leq3$, define $(G_i,\psi_i)$ to be the gain graph obtained from $(G,\psi)$ by adding the loop $v_i \xrightarrow{g_i} v_i$. 

For each $p \in \mathbb{H}^{V(G)}$ and each $1\leq i\leq3$, we set $\mathfrak{h}_{p(v_i)}$ as in \Cref{def:hz}, and 
we see that $X\in \mathcal{T}_{\langle g_i \rangle} + \mathfrak{h}_{p(v_i)}$ gives $\eta(X, p)\in \ker O(G_i, \psi_i, p)$. By \Cref{cyclic_main_theorem}, and since $ \mathfrak{h}_{p(v_i)} \neq \mathcal{T}_{\langle g_i \rangle}$ if and only if $g_i\cdot p(v_i) \neq p(v_i)$, it follows that there exists a dense open subset $U_i\subseteq U$ such that for all $p\in U_i$, one has
$$\ker O(G_i, \psi_i,p) = \left\{\eta(X, p)~:~ X\in \left( \mathcal{T}_{\langle g_i \rangle} + \mathfrak{h}_{p(v_i)}\right) \right\}.$$
Hence, a realisation $p \in \bigcap_{i=1}^3 U_i$ describes a $\Gamma$-isostatic framework $(G',\psi',p)$ if and only if
\begin{equation}\label{eq:base_case_lemma}
    \bigcap_{i=1}^3 \left( \mathcal{T}_{\langle g_i \rangle} + \mathfrak{h}_{p(v_i)} \right) = \{\textbf{0}_{2\times 2}\}.
\end{equation} 
By \Cref{lem: Lie -algebras-stabilizers0}, $\{\mathfrak{h}_{p(v_i)}~ \vert~ p \in U_i \}$ is a non-empty open subset of $\mathbb{P}(\mathfrak{sl}(2, \mathbb{R}))\cong \mathbb{RP}^2$. As $\langle g_1,g_2,g_3 \rangle$ is non-cyclic, it follows from \Cref{lem:centraliser_dim} that
\begin{equation*}
    \bigcap_{i=1}^3 \mathcal{T}_{\langle g_i \rangle} = \mathcal{T}_{\langle g_1,g_2,g_3 \rangle} =  \{\textbf{0}_{2\times 2}\}.
\end{equation*}
Now fix the set
\begin{equation*}
    O := \left\{ \big(\mathfrak{h}_{p(v_1)},\mathfrak{h}_{p(v_2)},\mathfrak{h}_{p(v_3)} \big)  : p \in  \bigcap_{i=1}^3 U_i  \right\} \subset \Big( \mathbb{P}\big(\mathfrak{sl}(2,\mathbb{R}) \big) \Big)^3.
\end{equation*}
As the vertices $v_1,v_2,v_3$ are pairwise distinct, each of the matrices $h (p(v_1) ), h(p(v_2) ), h (p(v_3) )$ is independent of the choice of the other two.
Hence, the set $O$ is non-empty and open. \Cref{lin_alg_lemma} now implies that there exists some realisation $p$ for which \Cref{eq:base_case_lemma} holds,
so $(G',\psi')$ is $\Gamma$-isostatic.
\end{proof}

\subsection{Near-cyclic graphs}
\label{subsec:nearcyclic}

Let $\Gamma$ be a non-cyclic Fuchsian group and $(G,\psi)$ be a near-cyclic $\Gamma$-gain graph. In this section we show that $(G,\psi)$ is $\Gamma$-isostatic. By \Cref{lem:trivial} we assume that $|V(G)|\geq2$. The following lemma allows us to understand which trivial infinitesimal motions yield infinitesimal motions of a single edge $u\xrightarrow[]{g}v$.

\begin{lemma}\label{inf_motions_edge}
     Let $g \in \PSL(2,\RR)$. 
     For any $P_1,P_2\in\HH$ with $P_1 \neq g\cdot P_2$, define the linear map
     \begin{equation*}
         R_{P_1,P_2}: \mathfrak{sl}(2,\mathbb{R})\longrightarrow\mathbb{R}, ~ X \longmapsto \Big\langle \dot{\gamma}_{P_1, g\cdot P_2}(0) ~ , ~ \eta(X, P_1) \Big\rangle_{P_1} + \Big\langle \dot{\gamma}_{P_2, g^{-1}\cdot P_1}(0) ~ , ~ \eta(X,P_2)  \Big\rangle_{P_2}.
     \end{equation*}
     Then
     $$A_{P_1,P_2}:= \mathcal{T}_{\langle g\rangle} + \left( \left( \mathfrak{h}_{P_1} +\mathfrak{h}_{g\cdot P_2} \right) \,\cap\, \left( \mathfrak{h}_{P_2} +\mathfrak{h}_{g^{-1}\cdot P_1} \right)\right)$$
     is contained in the kernel of $R_{P_1,P_2}$.
     Moreover, if $g \neq \textrm{id}$ and $P_1,P_2,g\cdot P_1,g^{-1}\cdot P_2$ do not lie on a common geodesic,
     then $\dim A_{P_1,P_2}=2$ and $\ker R_{P_1,P_2} = A_{P_1,P_2}$.
\end{lemma}

\begin{proof} 
We first prove $A_{P_1,P_2} \subset \ker R_{P_1,P_2}$ if $P_1 \neq g\cdot P_2$.
Let $X\in \left( \mathfrak{h}_{P_1} +\mathfrak{h}_{g\cdot P_2} \right) \,\cap\, \left( \mathfrak{h}_{P_2} +\mathfrak{h}_{g^{-1}\cdot P_1} \right)$. We can write $X = X_{P_1} + X_{g\cdot P_2} = X_{P_2} + X_{g^{-1}\cdot P_1}$ where $X_{a} \in \mathfrak{h}_{a}$ for $a\in \{P_1, P_2, g\cdot P_2, g^{-1}\cdot P_1\}$. Since $\eta$ is linear and $X_{P_1} \in \mathfrak{h}_{P_1}$, we have
\begin{equation*}
  \eta( X,P_1)  = \eta( X_{P_1}, P_1) + \eta(X_{g\cdot P_2},P_1) =  \eta(X_{g\cdot P_2},P_1).
\end{equation*}
Similarly, $\eta(X,P_2)=\eta(X_{g^{-1}\cdot P_1},P_2)$.
Therefore, by \Cref{lem: eqfrom_lem_inf_motions}
\begin{align*}
    R_{P_1,P_2}(X)&=\langle \dot{\gamma}_{P_1, g\cdot P_2}(0),\eta(X_{g\cdot P_2}, P_1) \rangle_{P_1} + \langle \dot{\gamma}_{P_2, g^{-1}\cdot P_1}(0),\eta(X_{g^{-1}\cdot P_1},P_2)  \rangle_{P_2}=0,
\end{align*}
and so $X\in \ker R_{P_1,P_2}$.
By \Cref{lem:centraliser_triv}, $\mathcal{T}_{\langle g \rangle} \subseteq \ker R_{P_1,P_2}$,
and thus $A_{P_1,P_2}\subseteq\ker R_{P_1,P_2}$.

We now suppose that $g \neq \textrm{id}$ and $P_1,P_2,g\cdot P_1,g^{-1}\cdot P_2$ do not lie on a common geodesic.
To prove that $\dim A_{P_1,P_2}=2$, we need the following.

\begin{claim}\label{claim: collinear-iff-collinear} 
   Three points $x,y,z \in \HH$ lie on a common geodesic if and only if $\mathfrak{h}_{x} \subseteq \mathfrak{h}_{y} +  \mathfrak{h}_{z}$.
\end{claim} 
\begin{proof}
Suppose first $\mathfrak{h}_x \subseteq \mathfrak{h}_y + \mathfrak{h}_{z}$ and choose some non-zero $W_x = W_y + W_z\in\mathfrak{h}_x$ for some $ W_y\in \mathfrak{h}_y$, $W_z\in \mathfrak{h}_z$. Since $\eta(W_x, x)=\textbf{0}$, it follows that $\eta(W_y, x) = \eta(-W_z, x)$. Note $x=z$ whenever $W_y=\textbf{0}_{2\times 2}$, and $x=y$ whenever $W_{z}=\textbf{0}_{2\times 2}$. In both cases, $x$ lies on the geodesic through $y$ and $z$. So we may assume $W_y,W_z\neq \textbf{0}_{2\times 2}$. By \Cref{lem: eqfrom_lem_inf_motions},  $\eta(W_y, x)$ is orthogonal to $\dot{\gamma}_{x, y}(0)$ and $\eta(-W_z, x)$ is orthogonal to $\dot{\gamma}_{x, z}(0)$. It follows that $\dot{\gamma}_{x, y}(0) = \pm \dot{\gamma}_{x, z}(0)$. By \Cref{lem:linear_independence}, $x,y,z$ lie on one geodesic. 

 Conversely, suppose $x,y,z$ lie on a common geodesic. Choose any $W_y\in \mathfrak{h}_{y}$, $W_z\in \mathfrak{h}_{z}$. Note that $W_y \in \mathfrak{h}_x$ whenever $\eta(W_y, x) = \textbf{0}$ and $W_z \in \mathfrak{h}_x$ whenever $\eta(W_z, x) = \textbf{0}$. In the former case we have $\mathfrak{h}_x=\mathfrak{h}_y$ and in the latter we have $\mathfrak{h}_x=\mathfrak{h}_z$. Hence the result holds in both cases and we assume that $\eta(W_y, x),\eta(W_z, x) \neq \textbf{0}_{2\times 2}$. By \Cref{lem: eqfrom_lem_inf_motions}, $\langle \dot{\gamma}_{x, y}(0) , \eta(W_y, x) \rangle_x= 0$ and $\langle \dot{\gamma}_{x, z}(0) , \eta(W_z, x) \rangle_x= 0$. Therefore $\{\dot{\gamma}_{x,y}(0), \eta(W_y, x)\}$ forms a basis at $T_{x}\HH$, as  does $\{\dot{\gamma}_{x,z}(0), \eta(W_z, x)\}$. Since $x,y,z$ all lie on the same geodesic, $\dot{\gamma}_{x,z}(0) = \pm \dot{\gamma}_{x,y}(0)$ by \Cref{lem:linear_independence}, which implies that there exist $\mu, \kappa$ such that $\eta(\mu W_y, x) + \eta(\kappa W_z, x) =\textbf{0}$. So $\mu W_y + \kappa W_z \in \mathfrak{h}_x$. Hence $(\mathfrak{h}_y+\mathfrak{h}_z) \,\cap\, \mathfrak{h}_x\neq \{\textbf{0}_{2\times 2}\}$ and, since $\dim(\mathfrak{h}_x) =1$, the claim holds. 
\end{proof}

It now follows that $\dim\left( \mathfrak{h}_{P_1} + \mathfrak{h}_{g P_2}\right) \,\cap\, \left(\mathfrak{h}_{P_2} +  \mathfrak{h}_{g^{-1}P_1} \right)=2$ as $P_1,P_2, g\cdot P_2, g^{-1}\cdot P_1$ do not lie on a common geodesic.

Finally, we prove that $\ker R_{P_1,P_2} = A_{P_1,P_2}$.
As $A_{P_1,P_2} \subseteq \ker R_{P_1,P_2}$, it suffices to show $\text{null}\, R_{P_1,P_2}=2$, i.e.~that $R_{P_1,P_2}$ is not the zero map.
Suppose for a contradiction that $R_{P_1,P_2} \equiv 0$ and $\ker R_{P_1,P_2}=\mathfrak{sl}(2,\RR)$.
Fix $\Gamma = \langle g \rangle$ and let $(G,\psi,p)$ be the $\Gamma$-joint-configuration with vertices $u,v$, exactly two edges $u \xrightarrow[]{\id} v$ and $u \xrightarrow[]{g} v$, and realisation with $p(u)=P_1$ and $p(v) = P_2$.
By \Cref{lem:centraliser_triv}, the image of the linear map
\begin{equation*}
    T: \mathfrak{sl}(2,\RR) \longrightarrow (\RR^2)^2, ~ X \longmapsto \Big( \eta\big(X,p(u)\big), \eta\big(X,p(v)\big) \Big)
\end{equation*}
is 3-dimensional (specifically, it is $\mathcal{T}_\Gamma(p)$). 
Note that the kernel of $D_p f_{G,\psi}$ is exactly the kernel of $R_{P_1,P_2}$, which, by assumption, is $\mathfrak{sl}(2,\RR)$.
Thus the image of $T$ is contained in the kernel of $D_p f_{G,\psi}$, and so $\rank D_p f_{G,\psi} \leq 1$.
Since $\rank D_p f_{G,\psi} = \rank O(G,\psi,p)$,
we have that the matrix
\begin{equation*}
    O(G,\psi,p)
    =
    \begin{bmatrix}
        \dot{\gamma}_{P_1,P_2}(0) & \dot{\gamma}_{P_2,P_1}(0) \\
        \dot{\gamma}_{P_1,g\cdot P_2}(0) & \dot{\gamma}_{P_2,g^{-1}\cdot P_1}(0)
    \end{bmatrix}
\end{equation*}
has linearly dependent rows.
It now follows from \Cref{lem:linear_independence} that $P_1,P_2, g\cdot P_2$ lie on a common geodesic and $P_1,P_2, g^{-1}\cdot P_1$ lie on a common geodesic.
However, as $P_1 \neq g\cdot P_2$ (and $P_2 \neq g^{-1}\cdot P_1$), this now implies all four points lie on a common geodesic,
a contradiction.
\end{proof}

\begin{lemma}\label{lem:near}
Let $\Gamma$ be a Fuchsian group that is isomorphic to either a surface group or a non-cyclic finitely generated free group.
Then every near-cyclic $\Gamma$-gain graph is $\Gamma$-isostatic.
\end{lemma}

\begin{proof}
    We first observe that $\Gamma$ does not contain any element of order 2:
    it is clear that free groups have no finite-order elements, and surface groups have no finite-order elements by \Cref{lem:noparabolic}. 
    
    Let $(G,\psi)$ be a near-cyclic $\Gamma$-gain graph. Fix an edge $e=u\xrightarrow[]{g}v$ of $G$ for which, given $\psi' := \psi|_{E(G)-e}$, the $\Gamma$-gain graph $(G-e, \psi')$ is cyclic.
    By \Cref{lem:treegain} and \Cref{prop:balgain} we may assume 
    there is a non-trivial subgroup $\Gamma'=\langle h\rangle \subsetneq \Gamma$ for which every edge $f \in E(G-e)$ has gain $\psi(f) \in \Gamma'$. Since $\Gamma$ is not cyclic, $g\not\in\Gamma'$. In particular, $g\neq\text{id}$.
    
    By \Cref{cyclic_main_theorem},
    $(G-e,\psi')$ is $\Gamma'$-isostatic, and so there is a dense open set $W\subseteq \mathbb{H}^{V(G)}$ such that, for all $p\in W$, the space of infinitesimal motions of $(G-e,\psi',p)$ is $\mathcal{T}_{\langle h\rangle}(p)$. 
    Fix some non-zero $X_h \in \mathcal{T}_{\langle h\rangle}$.
    It is now sufficient for us to show there exists $p \in W$ for which $\eta(X_h,p)$ restricted to the vertices $u,v$ is not an infinitesimal motion of $(G[e],\psi|_{\{e\}},p|_{\{u,v\}})$.
    In fact, since the maps $p \mapsto \eta(X_h,p)$ and $p \mapsto D_{p|_{\{u,v\}}} f_{G[e],\psi|_{\{e\}}}$ are continuous on the choices of $p$ where $p(u) \neq g \cdot p(v)$,
    we can relax our assumption that our witness $p$ is contained in $W$ to just being some $p \in \HH^{V(G)}$ with $p(u) \neq g \cdot p(v)$.
    We consider two separate cases.

\medskip
\noindent\textbf{\textit{Case (i)} $u\neq v$:} 
For any points $P_1,P_2 \in \HH$ with $P_1 \neq g \cdot P_2$,
let $A_{P_1,P_2}$ be the linear space described in \Cref{inf_motions_edge} with respect to the group element $g$.
By \Cref{inf_motions_edge},
it now suffices to find $P_1,P_2 \in \HH$ such that $P_1,P_2,g \cdot P_1, g^{-1} \cdot P_2$ do not lie on a common geodesic and $X_h \notin A_{P_1,P_2}$.

We now make the following two observations regarding points $P \in \HH$:
\begin{enumerate}
    \item Since $g$ does not have order two, by~\Cref{claim: lem-1ext} there exists $P \in \HH$ such that the points $P, g\cdot P, g^{-1}\cdot P$ do not lie on a common geodesic.
    It can be easily checked that this implies that there exists a dense open set of such points $P$ satisfying this property.
    \item As $\{\mathfrak{h}_{P}~\vert~ P\in \mathbb{H}\}$ is a non-empty open set (\Cref{lem: Lie -algebras-stabilizers0}),
    there exists a non-empty open set of $P \in \HH$ for which $\mathfrak{h}_{P}\not\subseteq\mathcal{T}_{\langle g \rangle } + \mathcal{T}_{\langle h \rangle }$.
\end{enumerate}
Since the points satisfying both properties form an non-empty open set,
we can now fix a point $P \in \HH$ satisfying the above two points.
With this choice of $P$, fix $P_1=P_2=P$.
By \Cref{inf_motions_edge}, $\dim A_{P_1, P_2} = 2$.
Since $\mathfrak{h}_{P}\not\subseteq\mathcal{T}_{\langle g \rangle }$,
this implies the right-hand side of the inclusion
\begin{equation*}
     \mathfrak{h}_{P} \subseteq \left( \left( \mathfrak{h}_{P} +\mathfrak{h}_{g\cdot P} \right) \,\cap\, \left( \mathfrak{h}_{P} +\mathfrak{h}_{g^{-1}\cdot P} \right)\right)
\end{equation*}
is 1-dimensional,
and hence the inclusion is an equality.
Thus $A_{P_1, P_2} =\mathcal{T}_{\langle g\rangle} +  \mathfrak{h}_{P}$. 
Since $\langle g,h\rangle$ is non-cyclic,
the linear space $\mathcal{T}_{\langle g, h\rangle}$ is 0-dimensional by \Cref{lem:centraliser_dim}.
It is simple to check that $\mathcal{T}_{\langle g, h\rangle} = \mathcal{T}_{\langle g\rangle} \cap \mathcal{T}_{\langle h\rangle}$, and thus $\mathcal{T}_{\langle g \rangle} \neq \mathcal{T}_{\langle h \rangle}$.
Since $\mathfrak{h}_{P}\not\subseteq\mathcal{T}_{\langle g \rangle } + \mathcal{T}_{\langle h \rangle }$, we have $\mathcal{T}_{\langle h\rangle} \,\cap\, A_{P_1, P_1}=\{\textbf{0}_{2\times 2}\}$, which implies that $X_h \notin A_{P_1,P_2}$.

\medskip
\noindent\textbf{\textit{Case (ii)} $u= v$:}
In this case $(G[e],\psi|_{\{e\}},p|_{\{u,v\}})$ has a single vertex and a single (loop) edge.
Fix $(H,\psi_H)$ to be the $\Gamma$-gain graph with a single vertex $v$ and two loops $l_g=v\xrightarrow[]{g}v$ and $l_h=v\xrightarrow[]{h}v$.
We now observe that finding $p$ with $p(u) \neq g \cdot p(v)$ for which $\eta(X_h,p)$ restricted to the vertices $u,v$ is not an infinitesimal motion of $(G[e],\psi|_{\{e\}},p|_{\{u,v\}})$ is equivalent to finding $q \in \HH^{V(H)} = \HH$ where $(H,\psi_H,q)$ is $\Gamma$-isostatic.
As trivial $\Gamma$-gain graphs are $\Gamma$-isostatic (\Cref{lem:trivial}), this now concludes the proof.
\end{proof}

\begin{remark}
    The proof of \Cref{lem:near} given above only requires that no element of $\Gamma$ has order 2.
    The proof can be adapted to allow for order-2 elements (and hence apply for all Fuchsian groups) using the Lorentzian geometry techniques from \Cref{subsub:the hardest lemma on earth}.
    Since we do not require this level of generality to prove \Cref{main_irreducible},
    we have omitted this more technical case.
\end{remark}

\subsection{Double cycles}
\label{sec: DC}
The remaining family of base graphs are double cycles. It will follow, from \Cref{thm:main}, that they are $\Gamma$-isostatic for surface groups $\Gamma$. 
Here we complete the proof of \Cref{main_irreducible} by showing that double cycles are reducible.
By contrast, there exist symmetry groups where this is known to fail in the setting of Euclidean rigidity. The smallest known example is the Dihedral group $\mathcal{D}_{2v}=\langle s,r:s^2=r^2=(sr)^2=\text{id}\rangle$ of order 4. The double cycle $C_2^2$ in \Cref{bottema}(a) is irreducible. This lifts to the Bottema mechanism (see \Cref{bottema}(b)), which is flexible \cite{jzt2016,Whiteley84,bottema}. 
We first need the following.

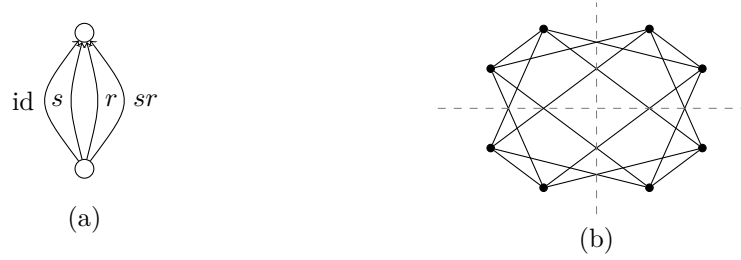
\begin{figure}[ht]
    \begin{subfigure}[scale=0.5]{0.55\textwidth}\centering
        \begin{tikzpicture}
  [scale=.9,auto=left]
            \node(u)[draw, circle, scale=0.75] at (0,0){};
            \node(v)[draw, circle, scale=0.75] at (0,2){};
            \node at (0,-0.75){};
            \draw[->] (u) .. controls (0.25,1) .. (v);
            \draw[->] (u) .. controls (-0.25,1) .. (v);
            \draw[->] (u) .. controls (0.75,1) .. (v);
            \draw[->] (u) .. controls (-0.75,1) .. (v);
            \node at (-0.9,1) {id};
            \node at (-0.4,1) {$s$};
            \node at (0.9,1) {$sr$};
            \node at (0.4,1) {$r$};
            \node at (0,-0.75){(a)};
        \end{tikzpicture}
    \end{subfigure}
    \begin{subfigure}[scale=0.5]{0.25\textwidth}\centering
        \begin{tikzpicture}
  [scale=.7,auto=left]
        \node(u)[inner sep=1pt,circle,draw,fill] at (2,0.75){};
        \node(su)[inner sep=1pt,circle,draw,fill] at (-2,0.75){};
        \node(ru)[inner sep=1pt,circle,draw,fill] at (-2,-0.75){};
        \node(sru)[inner sep=1pt,circle,draw,fill] at (2,-0.75){};

        \node(v)[inner sep=1pt,circle,draw,fill] at (1,1.5){};
        \node(sv)[inner sep=1pt,circle,draw,fill] at (-1,1.5){};
        \node(rv)[inner sep=1pt,circle,draw,fill] at (-1,-1.5){};
        \node(srv)[inner sep=1pt,circle,draw,fill] at (1,-1.5){};
        \draw(u) -- (v) -- (ru) -- (rv) -- (u);
        \draw(su) -- (sv) -- (sru) -- (srv) -- (su);
        \draw(u) -- (sv) -- (ru) -- (srv) -- (u);
        \draw(su) -- (v) -- (sru) -- (rv) -- (su);
        \draw[dashed,gray] (-3, 0) -- (3,0);
        \draw[dashed,gray] (0,-2) -- (0,2);
        \node at (0,-2.5){(b)};
    \end{tikzpicture}\end{subfigure}
    \caption{A double cycle $C_2^2$ and the $\mathcal{D}_{2v}$-lift of a generic $\mathcal{D}_{2v}$-joint-configuration of $C_2^2$.}
    \label{bottema}
\end{figure}

\begin{lemma}\label{lem:dc}
    Let $\Gamma$ be a non-cyclic free group and let $\{g_1,g_2,g_3\}$ be a (not necessarily free) generating set for $\Gamma$. Then at least one of $\Gamma_1=\langle g_1g_2^{-1}, g_3\rangle,\Gamma_2=\langle g_1 g_3^{-1}, g_2\rangle$ and $\Gamma_3=\langle g_1, g_2g_3^{-1}\rangle$ is a rank two free group.
\end{lemma}

\begin{proof}
    By the Nielsen–Schreier Theorem (see, e.g. \cite[Section 7.2]{nielsen}),
    every subgroup of a free group is a free group;
    in particular, each of $\Gamma_1,\Gamma_2,\Gamma_3$ is a free group with rank at most two.
    Suppose for contradiction that each of $\Gamma_1,\Gamma_2,\Gamma_3$ is cyclic, i.e.~has rank at most one.
    Since $\Gamma$ is a free group of rank $k \in \{2,3\}$,
    there exists a surjective group homomorphism $\theta : \Gamma \rightarrow \mathbb{Z}^k$: if $x_1,\ldots,x_k$ is a free generating set for $\Gamma$ and $y_1,\ldots,y_k$ is a free generating set for $\mathbb{Z}^k$, we can define $\theta$ by setting $\theta(x_i) = y_i$.
    Thus $\{\theta(g_1),\theta(g_2),\theta(g_3)\}$ is a generating set for $\mathbb{Z}^k$ and, by our initial assumption,
    each of the subgroups
    \begin{equation*}
        \theta(\Gamma_1) = \langle \theta(g_1) - \theta(g_2) , \theta(g_3) \rangle, \qquad \theta(\Gamma_2) = \langle \theta(g_1) - \theta(g_3) , \theta(g_2) \rangle  \qquad \mbox{and} \qquad \theta(\Gamma_3) = \langle \theta(g_2) - \theta(g_3) , \theta(g_1) \rangle
    \end{equation*}
    is cyclic. 
    Hence,
    $\{\theta(g_1),\theta(g_2),\theta(g_3)\}$ is a spanning subset of $\mathbb{Q}^k$ and each pair
    \begin{equation*}
        \{\theta(g_1) - \theta(g_2) , \theta(g_3) \}, \qquad \{ \theta(g_1) - \theta(g_3) , \theta(g_2) \} \qquad \mbox{and} \qquad \{ \theta(g_2) - \theta(g_3) , \theta(g_1) \}
    \end{equation*}
    is linearly dependent in $\mathbb{Q}^k$.

    If $k=3$, then $\{\theta(g_1),\theta(g_2),\theta(g_3)\}$ is a basis of $\mathbb{Q}^3$, implying that $\theta(g_1) - \theta(g_2) , \theta(g_3)$ are linearly independent, and hence contradicting our original assumption.
    So $k=2$.
    Without loss of generality, we may assume that $x_1,x_2$ are linearly independent.
    As $\theta(g_1) - \theta(g_2) , \theta(g_3)$ are linearly dependent, $\theta(g_3) = \lambda ( \theta(g_1) - \theta(g_2))$ for some $\lambda \in \mathbb{Q}$.
    Since $\theta(g_1) - \theta(g_3) , \theta(g_2)$ are linearly dependent and $\theta(g_1),\theta(g_2)$ are linearly independent,
    the following equality holds for some $\mu \in \mathbb{Q}$:
    \begin{equation*}
        \mu \theta(g_2) = \theta(g_1) - \theta(g_3) = \theta(g_1) - \lambda ( \theta(g_1) - \theta(g_2)) = (1-\lambda) \theta(g_1) + \lambda \theta(g_2).
    \end{equation*}
    Thus, $\lambda = 1$ and $\theta(g_3) = \theta(g_1) - \theta(g_2)$.
    However,
    this implies that 
    \begin{equation*}
        \theta(g_2) - \theta(g_3) = 2\theta(g_2) - \theta(g_1),
    \end{equation*}
    which contradicts the fact that $\theta(g_2) - \theta(g_3) , \theta(g_1)$ are linearly dependent. This concludes our proof.
\end{proof}

\begin{lemma}
\label{l:one non-cyclic}
    Let $\Gamma$ be a non-cyclic free group or a surface group of genus $g \geq 2$.
    Let $g_1,g_2,g_3\neq\text{id}$ be distinct elements of $\Gamma$ such that $\Gamma':=\langle g_1,g_2, g_3\rangle$ is not cyclic. Then one of $\Gamma_1=\langle g_1g_2^{-1}, g_3\rangle,\Gamma_2=\langle g_1 g_3^{-1}, g_2\rangle$ and $\Gamma_3=\langle g_1, g_2g_3^{-1}\rangle$ is not cyclic.
\end{lemma}

\begin{proof}
If $\Gamma$ is a non-cyclic free group, then $\Gamma'$ is also free by the Nielsen-Schreier Theorem \cite[Section I.3.4]{MR1954121}, and the result follows from \Cref{lem:dc}. Therefore, we may assume that $\Gamma$ is a surface group of genus $g\geq2$.

Hurewicz theorem \cite[Chapter IV, Theorem 7.1]{hurewicz} provides an isomorphism between the abelianisation of $\Gamma$ and the homology group $H_1(S_g,\mathbb{Z}) \cong \mathbb{Z}^{2g}$. It follows that there exists a surjective homomorphism $\theta : \Gamma \rightarrow \mathbb{Z}^{2g}$. The subgroup $\theta(\Gamma')$ of $\mathbb{Z}^{2g}$ has rank at most $3 < 2g$, and so $\theta(\Gamma')$ is an infinite index subgroup of $\mathbb{Z}^{2g}$. Since $\theta$ is a surjective map from the left cosets of $\Gamma'$ in $\Gamma$ to the left cosets of $\theta(\Gamma')$ in $\mathbb{Z}^{2g}$, we have
\begin{equation*}
    [\Gamma: \Gamma'] \geq  [\mathbb{Z}^{2g} : \theta(\Gamma')]=  \infty. 
\end{equation*}
Hence $\Gamma'$ is a subgroup of infinite index. The group $\Gamma'$ is thus a free group, since any infinite index subgroup of a surface group is a free group \cite[Theorem 1]{jaco1970certain}, and the result follows from \Cref{lem:dc}.
\end{proof}

\begin{lemma}\label{lem:dcirr}
    Let $\Gamma$ be a Fuchsian group that is isomorphic to either a surface group of genus $g \geq 2$ or a non-cyclic finitely-generated free group. Then every $(2,3,1,0)$-gain tight double cycle admits a reduction.
\end{lemma}

\begin{proof}
    Let $(G,\psi)$ be a $(2,3,1,0)$-gain tight graph with underlying graph $G:=C_n^2$ for some $n\geq2$, and let $v\in V(G)$. Let $v_1,v_2$ be the neighbours of $v$ and $e_1,e_2,e_3,e_4$ be the edges incident to $v$. Assume that $e_i$ is directed from $v$ for $1\leq i\leq4$, that $e_1,e_2$ are incident to $v_1$, and that $e_3,e_4$ are incident to $v_2$. Note, if $n=2$, then $v_1=v_2$ and $e_1,e_2,e_3,e_4$ are all parallel. For $1\leq i\leq4$, let $g_i$ denote the gain of $e_i$ and assume, without loss of generality, $g_4=\text{id}$.

    There are three possible ways of applying a 2-reduction at $v$. Starting from $C_n^2-v$ we can add: (a) a loop at $v_1$ and a loop at $v_2$ with gains $g_1g_2^{-1}$ and $g_3$, respectively; (b) parallel edges between $v_1,v_2$ with gains $g_1g_3^{-1}$ and $g_2$; (c) parallel edges between $v_1,v_2$ with gains $g_1$ and $g_2^{-1}g_3$. Let $(G_1,\psi_1),(G_2,\psi_2)$ and $(G_3,\psi_3)$ be the graphs obtained from applying (a),(b) and (c), respectively. 
    
    If $n\geq3$, then $G_1$ is a path of parallel edges with loops at its end-vertices, and $G_2,G_3$ are double cycles on $n-1$ vertices. If $n=2$, then all three graphs are a vertex with two loops. In each case if one of $\Gamma_1=\langle g_1g_2^{-1}, g_3\rangle,\Gamma_2=\langle g_1 g_3^{-1}, g_2\rangle$ and $\Gamma_3=\langle g_1, g_2g_3^{-1}\rangle$ is not cyclic, then one of $(G_1,\psi_1),(G_2,\psi_2)$ and $(G_3,\psi_3)$ is admissible. By $(2,3,1,0)$-gain tightness of $C_n^2$ and \Cref{l:one non-cyclic}, it follows that one of $\Gamma_1,\Gamma_2$ and $\Gamma_3$ is not cyclic.
\end{proof}

\section{2-extensions and loop-2-extensions}\label{sec:2ext}

In this section we prove the following.

\begin{theorem}\label{thm:deg4ext}
    Let $\Gamma$ be a surface group, $(G,\psi)$ be a $\Gamma$-gain graph,
    and $(G',\psi')$ be obtained from $(G,\psi)$ by applying a 2-extension or a loop-2-extension.
    If $(G,\psi)$ is $\Gamma$-isostatic, then $(G',\psi')$ is $\Gamma$-isostatic.
\end{theorem}

With \Cref{thm:deg4ext}, we can finally prove \Cref{main_theorem}.

\begin{proof}[Proof of \Cref{main_theorem}]
    Necessity follows from \Cref{thm:necessary}. Assume $(G,\psi)$ is $(2,3,1,0)$-gain tight. By \Cref{thm: recur. con.} and \Cref{lem:dcirr}, $(G,\psi)$ can be obtained from the disjoint union of irreducible $(2,3,1,0)$-gain tight graphs $(G_1,\psi_1),\dots,(G_t,\psi_t)$ by applying a series of 0-extensions, 1-extensions, loop-1-extension, 2-extension, and loop-2-extensions. By \Cref{main_irreducible}, each $(G_i,\psi_i)$ is $\Gamma$-isostatic, as is their union. Therefore, by \Cref{thm:simple_extension,thm:deg4ext}, $(G,\psi)$ is $\Gamma$-isostatic.
\end{proof}
 
\subsection{Loop-2-extensions}

\begin{lemma}
Let $\Gamma$ be a non-cyclic Fuchsian group, $(G,\psi)$ be a $\Gamma$-gain graph,
    and $(G',\psi')$ be obtained from $(G,\psi)$ by applying a loop-2-extension.
    If $(G,\psi)$ is $\Gamma$-isostatic, then $(G',\psi')$ is $\Gamma$-isostatic.
\end{lemma}

\begin{proof}
Suppose that $G'$ is obtained from $G$ by removing an edge $e=u_1\xrightarrow{g}u_2$ (here, $u_1,u_2$ are allowed to coincide) and adding a vertex $v$, together with the edges $e_1=v\xrightarrow{g_1}u_1,e_2=v\xrightarrow{g_2}u_2$ and  $l=v\xrightarrow{h}v$. By \Cref{def:loop2ext}, $\psi(e)=g_1^{-1}g_2$.

For any $\Gamma$-isostatic $(G,\psi,p)$ and any $t \in \RR$, define $p_t:V(G')\rightarrow\mathbb{H}$ by letting $p_t(u)=p(u)$ for all $u\in V(G)$ and $p_t(v)=\gamma_{p(u_1),g\cdot p(u_2)}(t)$. 
Using \Cref{key_lemma_loop2ext} applied to rows corresponding to $e_1,e_2,l$, we can apply row operations to $O(G',\psi',p')$ to obtain the matrix
\begin{equation*}
\left[ 
    \begin{array}{c|c}
        \begin{matrix}
        \dot{\gamma}_{p_t(v), h\cdot p_t(v)}(0)+\dot{\gamma}_{p_t(v), h^{-1}\cdot p_t(v)}(0)\\
        \dot{\gamma}_{p_t(v), g_1\cdot p(u_1)}(0)
        \end{matrix}

        &\begin{matrix}
            0 ~ \cdots ~ 0 & \textbf{0}  & 0 ~ \cdots ~ 0 \\
           0 ~ \cdots ~ 0 &  \dot{\gamma}_{p(u_1), g_1^{-1}\cdot p'(v)}(0) & 0 ~ \cdots ~ 0 
        \end{matrix} \\
        \hline \\
        \mathbf{0}_{|E(G)| \times 2} & O(G,\psi,p) \\ \\
    \end{array}
    \right].
\end{equation*}
Hence, $(G',\psi',p_t)$ is $\Gamma$-isostatic if and only if $(G,\psi,p)$ is $\Gamma$-isostatic and the vectors $\dot{\gamma}_{p_t(v), g_1\cdot p(u_1)}(0)$ and $\dot{\gamma}_{p_t(v), h\cdot p_t(v)}(0)+\dot{\gamma}_{p_t(v), h^{-1}\cdot p_t(v)}(0)$ are linearly independent.
Because the set of $p \in \HH^{V(G)}$ where $(G,\psi,p)$ is $\Gamma$-isostatic is dense and open,
it is now sufficient to show that there exists a dense set of $p \in \HH^{V(G)}$ where the two vectors are linearly independent for some choice of $t \in \RR$.
Since the position of $p(u_2)$ can be deduced from the positions of $p_t(v)$ and $p(u_1)$,
it is sufficient to instead show that there exists a dense set of distinct points $(z,w) \in \HH^2$ such that the vectors $\dot{\gamma}_{z, g_1 \cdot w}(0)$ and $\dot{\gamma}_{z, h\cdot z}(0)+\dot{\gamma}_{z, h^{-1}\cdot z}(0)$ are linearly independent.

By \Cref{lem: bad-geodesics},
there exists a dense open set $U \subset \HH$ for which $\dot{\gamma}_{z, h\cdot z}(0)+\dot{\gamma}_{z, h^{-1}\cdot z}(0)$ is non-zero for each $z \in U$.
Now fix $z \in U$.
It is then easy to see that the set $U_z\subset\HH$ for which $\dot\gamma_{z,g_1\cdot w}(0)$ has any direction other than (and so is linearly independent from) $\dot{\gamma}_{z, h\cdot z}(0)+\dot{\gamma}_{z, h^{-1}\cdot z}(0)$ for all $w\in U_z$ is both open and dense.
This now guarantees a dense subset $U' \subset \HH^2$ of pairs $(z,w)$ where he vectors $\dot{\gamma}_{z, g_1 \cdot w}(0)$ and $\dot{\gamma}_{z, h\cdot z}(0)+\dot{\gamma}_{z, h^{-1}\cdot z}(0)$ are linearly independent. This now concludes the proof.
\end{proof}

\subsection{2-extensions}
Next we will prove the following.

\begin{lemma}\label{2-ext}
    Let $\Gamma$ be a surface group, $(G,\psi)$ be a $\Gamma$-gain graph,
    and $(G',\psi')$ be obtained from $(G,\psi)$ by applying a 2-extension.
    If $(G,\psi)$ is $\Gamma$-isostatic, then $(G',\psi')$ is $\Gamma$-isostatic.
\end{lemma}
For the remainder of the section we assume that $\Gamma$ is a non-cyclic Fuchsian group, $(G,\psi)$ is a $\Gamma$-isostatic $\Gamma$-gain graph (and hence $(2,3,1,0)$-gain tight - recall \Cref{thm:necessary}), and $(G',\psi')$ is the $\Gamma$-gain graph obtained from $(G,\psi)$ by applying the 2-extension that removes the edges $e= v_1 \xrightarrow{g} v_2$ and $f= v_3 \xrightarrow{h} v_4$ and adds the vertex $v$, together with the edges $e_{1} = v \xrightarrow{g_1} v_1$, $e_{2}= v \xrightarrow{g_2} v_2$, $e_{3}= v \xrightarrow{g_3} v_3$, $e_{4}= v \xrightarrow{g_4} v_4$. (Here, $g=g_1^{-1}g_2$ and $h=g_3^{-1} g_4$.) 

To prove \Cref{2-ext}, we require a slight alteration to the concept of transverse intersection to keep track of the choice of points defining the lines.
For $P_1,Q_1,P_2,Q_2\in\HH$, we say point-labelled geodesics $\gamma_{P_1,Q_1}, \gamma_{P_2,Q_2}$ \emph{are crossing} if there are $s,t\in\RR$ such that:
\begin{itemize}
    \item[(i)] $\gamma_{P_1,Q_1}(s)=\gamma_{P_2,Q_2}(t)$;
    \item[(ii)] $\dot\gamma_{P_1,Q_1}(s),\dot\gamma_{P_2,Q_2}(t)$ are linearly independent; and
    \item[(iii)] $\gamma_{P_1,Q_1}(s) \notin \{P_1,Q_1\}$ and $\gamma_{P_2,Q_2}(t) \notin \{P_2,Q_2\}$. 
\end{itemize}
Note that, if $\gamma_{P_1,Q_1}, \gamma_{P_2,Q_2}$ are crossing, then they intersect transversally. However, the converse need not be true: if $P,Q_1,Q_2\in\HH$ do not all lie on the same geodesic, then $\gamma_{P,Q_1} ,\gamma_{P,Q_2}$ intersect transversally but are not crossing. 

\begin{lemma}\label{lem:2-extcrossedgeimplies extension}
    Suppose that there exists some $p: V(G)\rightarrow \HH$ such that the geodesics $\gamma_{g_1\cdot p(v_1), g_2\cdot p(v_2)}$ and $\gamma_{g_3\cdot p(v_3), g_4\cdot p(v_4)}$ are crossing. Then $(G',\psi')$ is $\Gamma$-isostatic.
\end{lemma}
\begin{proof} 
We first observe that any small perturbation of $\{p(u):u\in V(G)\}$ preserves the property that the geodesics $\gamma_{g_1\cdot p(v_1), g_2 \cdot p(v_2)}$ and $\gamma_{g_3\cdot p(v_3), g_4\cdot p(v_4)}$ are crossing, so we may assume that $(G,\psi,p)$ is regular, and hence $\Gamma$-isostatic. 
Let $x$ be the intersection of the two geodesics. Define $p'\in\mathbb{H}^{V(G')}$ by letting $p'(u) = p(u)$ for $u\in V(G)$ and $p'(v) = x$. We have 
\begin{align*}
    \dot{\gamma}_{p'(v), g_2\cdot p(v_2)}(0) = \varepsilon_1 \dot{\gamma}_{p'(v), g_1\cdot p(v_1)}(0) &
    \qquad \dot{\gamma}_{p'(v), g_4\cdot p(v_4)}(0) = \varepsilon_4 \dot{\gamma}_{p'(v), g_3\cdot p(v_3)}(0)\\
    \dot{\gamma}_{p(v_1), g_1^{-1}\cdot p'(v)}(0) =\varepsilon_2 \dot{\gamma}_{p(v_1), g\cdot p(v_2)}(0) &\qquad\dot{\gamma}_{p(v_3), g_3^{-1}\cdot p'(v)}(0) = \varepsilon_5 \dot{\gamma}_{p(v_3), h\cdot p(v_4)}(0)\\
    \dot{\gamma}_{p'(v_2), g_2^{-1}\cdot p'(v)}(0) = \varepsilon_3  \dot{\gamma}_{p(v_2), g^{-1}\cdot p(v_1)}(0) & \qquad 
    \dot{\gamma}_{p(v_4), g_4^{-1}\cdot p'(v)}(0) = \varepsilon_6 \dot{\gamma}_{p(v_4), h^{-1}\cdot p(v_3)}(0)
\end{align*}
where $\varepsilon_{i} \in \{\pm1\},$ exactly one of $\varepsilon_1, \varepsilon_2, \varepsilon_3$ is $-1$ and exactly one of $\varepsilon_4, \varepsilon_5, \varepsilon_6$ is $-1$.
Choose an equilibrium stress $\omega'$ of $(G',\psi',p')$. By observing the columns of $(\omega')^T O(G',\psi',p')$ associated to $v$, we can see
\begin{align*}
    &\omega'(e_1) \dot{\gamma}_{p'(v),g_1\cdot p(v_1)}(0) + \omega'(e_2) \dot{\gamma}_{p'(v),g_2\cdot p(v_2)}(0) + \omega'(e_3) \dot{\gamma}_{p'(v),g_3\cdot p(v_3)}(0) + \omega'(e_3) \dot{\gamma}_{p'(v),g_4\cdot p(v_4)}(0)  \\
    = &\Big(\omega'(e) +\varepsilon_1  \omega'(e_2) \Big)\dot{\gamma}_{p'(v),g_1\cdot p(v_1)}(0) + \Big(\omega'(e_3) +\varepsilon_4  \omega'(e_4) \Big)\dot{\gamma}_{p'(v),g_3\cdot p(v_3)}(0)=\textbf{0}.
\end{align*}
Since $\gamma_{g_1\cdot p_{v_1}, g_2\cdot p_{v_2}}$ and $\gamma_{g_3\cdot p_{v_3}, g_4\cdot p_{v_4}}$ are crossing, $\omega'(e_1) = -\varepsilon_1 \omega'(e_2)$ and $\omega'(e_3) = -\varepsilon_4 \omega'(e_4)$.
From this, define the vector $\omega \in \mathbb{R}^{E(G)}$ by letting $\omega(e^{\star})=\omega'(e^{\star})$ for all $e^{\star}\in E(G)\setminus\{e,f\}$, $\omega(e)=\varepsilon_2\omega'(e_1)$ and $\omega(f)=\varepsilon_5\omega'(e_3)$. Since $\omega'$ is an equilibrium stress of $(G',\psi',p')$ and 
\begin{align*}
    &\omega(e)\dot\gamma_{p(v_1),g\cdot p(v_2)}(0)=\varepsilon_2^2\omega'(e_1)\dot\gamma_{p(v_1),g_1^{-1}\cdot p'(v)}(0)=\omega'(e_1)\dot\gamma_{p(v_1),g_1^{-1}\cdot p'(v)}(0)\\
    &\omega(e)\dot\gamma_{p(v_2),g^{-1} \cdot p(v_1)}(0)=-\varepsilon_1\varepsilon_2\varepsilon_3\omega'(e_2)\dot\gamma_{p(v_2),g_2^{-1}\cdot p'(v)}(0)=\omega'(e_2)\dot\gamma_{p(v_2),g_2^{-1}\cdot p'(v)}(0),
\end{align*}
it is easy to see that the balancing condition (that is, \Cref{bal. cond.}) is satisfied by $\omega$ on both $v_1$ and $v_2$.
A similar argument using the edge $f$ shows that the balancing condition is satisfied by $\omega$ on both $v_3$ and $v_4$ also. 
It follows that $\omega$ is an equilibrium stress of $(G,\psi,p)$. Since $(G,\psi,p)$ is $\Gamma$-isostatic,
$\omega = \mathbf{0}$, which implies that $\omega' = \mathbf{0}$. So $O(G,'\psi',p')$ has full-rank. Since $|E(G')| = 2|V(G')|$, this implies $(G',\psi',p')$ is $\Gamma$-isostatic.
\end{proof}

\begin{lemma}\label{lem:easy2-exts1}
    Suppose $v_1 \notin \{v_2, v_3, v_4\}$. Then $(G',\psi')$ is $\Gamma$-isostatic.
\end{lemma}

\begin{proof}
By \Cref{lem:2-extcrossedgeimplies extension}, it suffices to show that there is some $p:V(G)\rightarrow\mathbb{H}$ such that $\gamma_{g_1\cdot p(v_1), g_2\cdot p(v_2)}$ and $\gamma_{g_3\cdot p(v_3) , g_4\cdot p(v_4)}$ are crossing. If, for some distinct $2\leq i,j\leq4$, we have $v_i=v_j$, then $g_i\neq g_j$ by the definition of $\Gamma$-gain graph. So choose $p$ such that $g_2\cdot p(v_2),g_3\cdot p(v_3),g_4\cdot p(v_4)$ are distinct and $g_1\cdot p(v_1)$ lies on $\gamma_{g_2\cdot p(v_2), x}$ for some $x$ on $\gamma_{g_3\cdot p(v_3), g_4\cdot p(v_4)}$. This is our desired $p.$
\end{proof}

By \Cref{lem:easy2-exts1}, we may assume $v_1=v_i$ for some $2\leq i\leq4$. Therefore, without loss of generality, we need only consider the following cases: (a) $v_1=v_2,v_3=v_4$ and $v_2\neq v_3$; (b) $v_1 = v_3,v_2 = v_4$ and $v_2\neq v_3$; and, (c) $v_1=v_2=v_3=v_4$. We will consider each case separately.

\subsubsection{Case (a). \texorpdfstring{$v_1=v_2,v_3=v_4$}{v1=v2,v3=v4} and \texorpdfstring{$v_2\neq v_3$}{v2,v3 distinct}}

In this section and the next, we will use boundary points to determine whether geodesics intersect.

\begin{lemma}\label{lem: cross-boundary-cross-plane0}
    Let $\Gamma$ be a Fuchsian group, and let $g,h \in \Gamma$ satisfy $g,h \neq \textup{id}$. Suppose that there exist $z_\alpha, z_\beta \in \RR \,\cup\, \{\infty\}$ such that $z_{\alpha}\neq g\cdot z_{\alpha}$, $z_{\beta}\neq  h\cdot z_{\beta}$ and such that the geodesics $(z_{\alpha}, g\cdot z_{\alpha})$ and $(z_{\beta}, h\cdot z_{\beta})$ intersect transversally. Then, there exists a non-empty open set of points $(\tau_{1}, \tau_{2})\in \mathbb{H}^{2}$ such that $\tau_1\neq\tau_2$ and such that the geodesics $\gamma_{\tau_{1}, h\cdot\tau_{1}}$ and $\gamma_{\tau_{2}, k\cdot \tau_{2}}$ are crossing. 
\end{lemma}

\begin{proof}
Fix $O \subset \textup{Geod}(\mathbb{H})^2$ to be the set of all pairs of geodesics that intersect transversally.
By \Cref{intersecting geodesics},
$O$ is a non-empty open set.
Given the set
\begin{equation*}
    Z = \left\{\big((z_\alpha,g\cdot z_\alpha),(z_\beta, h\cdot z_\beta)\big) \in  \textup{Geod}(\mathbb{H})^2 : z_\alpha \neq g\cdot z_\alpha , ~ z_\beta \neq h\cdot z_\beta \right\},
\end{equation*}
we have that $Z \,\cap\, O$ is a non-empty open subset of $Z$.
Given $\chi$ is the continuous and surjective map defined in \Cref{continuity-geodesics},
we now define the continuous and surjective map
\begin{equation*}
    f: \left\{ (\tau_1,\tau_2) \in \HH^2 : \tau_1, g\cdot \tau_1, \tau_2 , h \cdot\tau_2 \text{ pairwise distinct} \right\} \longrightarrow Z, ~
    (\tau_1,\tau_2) \longmapsto \big(\chi(\tau_1, g\cdot\tau_1), \chi( \tau_2, h\cdot \tau_2)\big).
\end{equation*}
The set $f^{-1}(Z \,\cap\, O)$ is a non-empty open subset of $\HH^2$;
moreover, if $(\tau_1,\tau_2) \in f^{-1}(Z \,\cap\, O)$ then $\gamma_{\tau_1,g\cdot\tau_1}$ and $\gamma_{\tau_2,h\cdot\tau_2}$ intersect transversally.
We now choose $(\tau_1,\tau_2) \in f^{-1}(Z \,\cap\, O)$ such that $\{\tau_1 , g\cdot \tau_1\} \,\cap\, \{\tau_2 , h\cdot \tau_2\}=\emptyset$ to obtain our crossing pair.
\end{proof}

We now split this case into two subcases. 

\begin{lemma}\label{lem:repeat1}
    Let $\Gamma$ be a Fuchsian group and $g,h \in \Gamma$ be such that $h, g,h g^{-1} \neq\textup{id}$.
    Then there exist points $z, w \in \RR\,\cup\, \{\infty\}$ such that the geodesics $(z,g\cdot z)$ and $(w,h\cdot w)$ intersect transversally.
\end{lemma}

\begin{proof}
    Choose any $z \in \RR\,\cup\,\{\infty\}$ that is not a fixed point of $g$ or $h$.
    First, we suppose $g\cdot z \neq h\cdot z$.
    We may suppose $z < g\cdot z$, as an almost identical method will apply when $g\cdot z<z$.
    If $h\cdot z <z$ or $h\cdot z>g\cdot z$, then there exists suitably small $\varepsilon >0$ so that
    \begin{equation*}
        h\cdot(z+\varepsilon) < z < z+\varepsilon < g\cdot z \qquad \text{ or } \qquad z < z+\varepsilon < g\cdot z <h\cdot(z+\varepsilon)
    \end{equation*}
    respectively.
    Then, setting $w = z + \varepsilon$, the result holds by \Cref{intersecting geodesics}.
    If $z<h\cdot z<g\cdot z$, then there exists suitably small $\varepsilon >0$ so that
    \begin{equation*}
        z - \varepsilon < z < h\cdot(z-\varepsilon) < g\cdot z.
    \end{equation*}
    Then, setting $w = z + \varepsilon$, the result holds by \Cref{intersecting geodesics}.
    
    Suppose for contradiction that neither of the above options hold for any $x$ that is not a fixed point of $g$ or $h$.
    Since $z$ is not a fixed point of $h$, we have $h\cdot x=g\cdot x$ for all but finitely many points in $\RR \,\cup\, \{\infty\}$.
    By continuity, $h\cdot x=g\cdot x$ for all $x \in \RR \,\cup\,\{\infty\}$.
    As an isometry is uniquely determined by its effect on the boundary of $\HH$, we thus have $g=h$.
    However, this contradicts that $hg^{-1} \neq \textup{id}$.
\end{proof}

\begin{lemma}\label{lem:repeat2}
    Let $\Gamma$ be a Fuchsian group and $g \in \Gamma$ be such that $g \neq\text{id}$.
    Then there exist $z, w \in \RR\,\cup\, \{\infty\}$ such that the geodesics $(z,g\cdot z)$ and $(w,g\cdot w)$ intersect transversally.
\end{lemma}

\begin{proof}
    We first note that for any $k\in \PSL(2,\RR)$ and any $z, w\in\RR\,\cup\,\{\infty\}$ with $z\neq g\cdot z$ and $w \neq g\cdot w$, the geodesics $(z,g\cdot z)$ and $(w,g\cdot w)$ intersect transversally if and only if $(z,(kgk^{-1})\cdot z)$ and $(w,(kgk^{-1})\cdot w)$ intersect transversally. 
    Hence by conjugating $g$, we may assume without loss of generality that one of the following holds:
    \begin{itemize}
        \item[(i)] $g$ is elliptic and $g = \begin{bmatrix}
            \cos(\theta) & -\sin(\theta)\\
            \sin(\theta) & \cos(\theta)\\
        \end{bmatrix}$ for some $\theta \in (0, 2\pi)$. 

        First, suppose $\tan (\theta) \neq \infty$.
        If $\tan(\theta) = 0$, then (since $\PSL(2,\RR)$ is a quotient) $g = \textup{id}$, which contradicts our original assumption.
        In this case, pick $z = 0$ and $w = x$ for some small $x \in \mathbb{R}$. Then, since $x \approx 0$, we have 
    \begin{equation*}
        g\cdot z = -\frac{\sin(\theta)}{\cos(\theta)} = -\tan(\theta) , \qquad
        g\cdot w = - \tan(\theta) \left( \frac{ 1 - x\tan(\theta)^{-1} }{1+ x\tan(\theta) } \right) \approx -\tan(\theta).
    \end{equation*}  
    We observe here that for small $x$,
    we have
    \begin{equation*}
        f(x) := \left( \frac{ 1 - x\tan(\theta)^{-1} }{1+ x\tan(\theta) } \right) 
        \begin{cases}
            <1 &\text{if } \tan (\theta) >0 \text{ and } x > 0, \text{ or } \tan (\theta) <0 \text{ and } x < 0  \\
            > 1 &\text{if } \tan (\theta) >0 \text{ and } x < 0, \text{ or } \tan (\theta) <0 \text{ and } x > 0.
        \end{cases}
    \end{equation*}  
    If $\tan(\theta) >0$, then choose $x>0$ with small magnitude so that $g\cdot z < g\cdot w < z <  w$. If, on the other hand, $\tan(\theta) < 0$, then choose $x<0$ small in magnitude so that $g\cdot w < g\cdot z < w < z$. 
    In both cases, $(z,g\cdot z)$ and $(w, g\cdot w)$ intersect transversally by \Cref{intersecting geodesics}.

    Now suppose $\tan(\theta)=\infty$.
    Then $g : z \mapsto -\frac{1}{z}$ for all $z \in \HH\, \cup \,\RR \,\cup\, \{\infty\}$. Pick $z=1$ and $w=2$.
    Then $g\cdot z < g\cdot w < z < w$, and so $(z,g\cdot z)$ and $(w, g\cdot w)$ intersect transversally by \Cref{intersecting geodesics}.
    
    \item[(ii)] $g$ is parabolic and $g=
            \begin{bmatrix}
                1 & t \\
                 0 & 1
            \end{bmatrix}$
            for some $t>0$.

            In this case, pick $z = 0$ and $w = -\frac{t}{2}$. Then $g\cdot z = t$ and $g\cdot w = \frac{t}{2}.$ Therefore, $(z,g\cdot z)$ and $(w, g\cdot w)$ intersect transversally by \Cref{intersecting geodesics}.
        \item[(iii)] $g$ is hyperbolic and $g=
            \begin{bmatrix}
                \lambda & 0 \\
                 0 & \frac{1}{\lambda}
            \end{bmatrix}$ for some $|\lambda | >1$.

            In this case, pick $z = 1$ and $1 < w < \lambda^2$. Then $1 < w < g\cdot z = \lambda^2 < g\cdot w$. Therefore, $(z,g\cdot z)$ and $(w, g\cdot w)$ intersect transversally \Cref{intersecting geodesics}.
    \end{itemize}  
    In each case, we showed that there is some choice of $z,w$ such that $(z,w)$ and $(g\cdot z, g^{-1}\cdot w)$ intersect transversally. 
\end{proof}

\begin{lemma}\label{lem:easy2-exts2}
    Suppose $v_1 = v_2$ and $v_3 = v_4$, but $v_1 \neq v_3$. 
    Then $(G',\psi')$ is $\Gamma$-isostatic.
\end{lemma}

\begin{proof}
    By \Cref{lem:treegain}, we may assume that $g_1=g_3=\text{id}$ and, by definition of gain graph, $g_2,g_4\neq\text{id}$. Then, by \Cref{lem:repeat1,lem:repeat2}, there are some distinct $z,w\in\RR\,\cup\,\{\infty\}$ such that the geodesics $(z,g_2\cdot z)$ and $(w,g_4\cdot w)$ intersect transversally. By \Cref{lem: cross-boundary-cross-plane0}, there is a choice of $p(v_1)=p(v_2)\neq p(v_3)=p(v_4)$ such that $\gamma_{p(v_1), g_2\cdot p(v_1)} = \gamma_{p(v_1), g_2\cdot p(v_2)}$ and $\gamma_{p(v_3),g_4\cdot p(v_3)}=\gamma_{p(v_3),g_4\cdot p(v_4)}$ are crossing. The result then follows from \Cref{lem:2-extcrossedgeimplies extension}.
\end{proof}

\subsubsection{Case (b). \texorpdfstring{$v_1=v_3,v_2=v_4$}{v1=v3,v2=v4} and \texorpdfstring{$v_2\neq v_3$}{v2,v3 distinct}}

We first require the following analogue to \Cref{lem: cross-boundary-cross-plane0}.
We omit its proof as it is almost identical.

\begin{lemma}\label{lem: cross-boundary-cross-plane}
    Let $\Gamma$ be a Fuchsian group, and let $g,h \in \Gamma$. Suppose there exist $z_\alpha, z_\beta \in \RR \,\cup\, \{\infty\}$ such that $z_{\alpha}\neq z_{\beta},$ $g\cdot z_{\alpha}\neq  h\cdot z_{\beta}$ and such that the geodesics $\gamma_1 = (z_{\alpha}, z_{\beta})$ and $\gamma_{2} = (g\cdot z_{\alpha}, h\cdot z_{\beta})$ intersect transversally. Then, there exists a non-empty open set of points $(\tau_{1}, \tau_{2})\in \mathbb{H}^{2}$ such that $\tau_1\neq\tau_2$ and such that the geodesics $\gamma_{\tau_{1}, \tau_{2}}$ and $\gamma_{g\cdot \tau_{1}, h\cdot \tau_{2}}$ are crossing. 
\end{lemma}

We now split this case into two separate subcases. 

\begin{lemma}\label{lem:nongeneric2-extcrossedgeimplies extension1}
    Let $\Gamma$ be a Fuchsian group and $g,h \in \Gamma$ be such that $h, g,hg \neq\text{id}$.
    Then there exist points $z, w \in \RR\,\cup\, \{\infty\}$ such that the geodesics $(z,w)$ and $(g\cdot z,h\cdot w)$ intersect transversally.
\end{lemma}

\begin{proof}
    Since $h,g,hg\neq\text{id}$, we may choose $z,w \in  \RR$ so that the elements in $\{z,w,g\cdot z,h\cdot w\}$ are pairwise distinct and $g\cdot z,h\cdot w,(hg)\cdot z \neq \infty$.
    Assume that $(z,w),(g\cdot z,h\cdot w)$ do not intersect transversally. Choose a continuous injective path $w:[0,1] \rightarrow \RR\,\cup\, \{\infty\}$ where $w(0) = w$, $w(1) = g\cdot z$ and $w(t) \neq z$ for all $0\leq t\leq 1$, and
    suppose that for all $0\leq t\leq 1$ the geodesics $(z, w(t)),(g\cdot z,h\cdot w(t))$ do not intersect transversally. If $( hg)\cdot z \neq z$ and $( hg)\cdot z\neq g\cdot z$ then, setting $w^* := w(1) = g\cdot z$, one of the following inequalities holds:
    \begin{align*}
        h\cdot w^* < g\cdot z = w^* < z, \qquad h\cdot w^* < z < g\cdot z = w^*, \qquad z < h\cdot w^* < w^* = g\cdot z, \\ 
        w^* = g\cdot z < h\cdot w^* < z, \qquad w^* = g\cdot z < z < h\cdot w^*, \qquad z < g\cdot z=w^* < h\cdot w^*.
    \end{align*}
    It follows that we can choose some $w^{**} \in \RR$ close to $w^*$ so that one of the following holds:
    \begin{align*}
        h\cdot w^{**} < g\cdot z < w^{**} < z, \qquad h\cdot w^{**} < z < g\cdot z < w^{**}, \qquad z < h\cdot w^{**} < w^{**} < g\cdot z, \\ 
        w^{**} < g\cdot z < h\cdot w^{**} < z, \qquad w^{**} < g\cdot z < z < h\cdot w^{**}, \qquad z < g\cdot z < w^{**} < h\cdot w^{**}.
    \end{align*}
    Then, by \Cref{intersecting geodesics}, $(z,w^{**})$ and $(g\cdot z,h\cdot w^{**})$ intersect transversally. So assume $(hg)\cdot z = z$ or $(hg)\cdot z = g\cdot z$ for all suitable choices of $z,w$.
    Since both equalities can be seen to be solution sets to rational polynomials, either $(hg)\cdot z = z$ for all $z,w \in \RR\,\cup\, \{\infty\}$, or $(hg)\cdot z = g\cdot z$ for all $z,w \in \RR\,\cup\, \{\infty\}$.
    Since $k\cdot z = z$ for all $z \in \RR\,\cup\, \{\infty\}$ if and only if $k = \text{id}$,
    either $hg = \text{id}$ or $hg = g$, a contradiction. This proves the result.
    \end{proof}

\begin{lemma}\label{lem:nongeneric2-extcrossedgeimplies extension2}
    Let $\Gamma$ be a Fuchsian group and $g,h \in \Gamma$ be such that $h = g^{-1}$ and $h,g \neq\text{id}$.
    Then there exist points $z, w \in \RR\,\cup\, \{\infty\}$ such that the geodesics $(z,w)$ and $(g\cdot z,h\cdot w)$ intersect transversally.
\end{lemma}

\begin{proof}
    We first note that for any $k\in \PSL(2,\RR)$ and any $z\neq w\in\RR\,\cup\,\{\infty\}$ with $g\cdot z\neq h\cdot w$, the geodesics $(z,w)$ and $(g\cdot z,h\cdot w)$ intersect transversally if and only if $(z,w)$ and $((kgk^{-1})\cdot z,(khk^{-1})\cdot w)$ intersect transversally. Since $h=g^{-1}$, this happens if and only if $khk^{-1} = (kgk^{-1})^{-1}$.
    Hence by conjugating $g$, we may assume without loss of generality that one of the following holds:
    \begin{itemize}
        \item[(i)] $g$ is elliptic and $g = \begin{bmatrix}
            \cos(\theta) & -\sin(\theta)\\
            \sin(\theta) & \cos(\theta)\\
        \end{bmatrix}$ for some $\theta \in (0, 2\pi)$. 
        
        In this case, pick $z = 0$ and $w = x$ for some small $x \in \mathbb{R}$. Then, since $x \approx 0$, we have 
    \begin{equation*}
        g\cdot z = -\frac{\sin(\theta)}{\cos(\theta)} = -\tan(\theta) , \qquad
        g^{-1}\cdot w = \tan(\theta) \left( \frac{\sin(\theta) \cos(\theta) + \cos^2(\theta) x}{\sin(\theta) \cos(\theta) - \sin^2(\theta) x} \right) \approx \tan(\theta).
    \end{equation*}   
    If $\tan(\theta) >0$, then choose $x<0$ with small magnitude so that $g\cdot z < w < z < g^{-1} \cdot w$. If, on the other hand, $\tan(\theta) < 0$, then choose $x>0$ small in magnitude so that $g^{-1}\cdot w < z < w < g \cdot z.$ 
    In both cases, $(z,w)$ and $(g\cdot z, g^{-1}\cdot w)$ intersect transversally by \Cref{intersecting geodesics}.
    
    \item[(ii)] $g$ is parabolic and $g=
            \begin{bmatrix}
                1 & t \\
                 0 & 1
            \end{bmatrix}$
            for some $t>0$.

            In this case, pick $z = 0$ and $w = \frac{t}{2}$. Then $g\cdot z = t$ and $g^{-1}\cdot w = -\frac{t}{2}.$ Therefore, $(z,w)$ and $(g\cdot z,g^{-1}\cdot w)$ intersect transversally by \Cref{intersecting geodesics}.
        \item[(iii)] $g$ is hyperbolic and $g=
            \begin{bmatrix}
                \lambda & 0 \\
                 0 & \frac{1}{\lambda}
            \end{bmatrix}$ for some $|\lambda | >1$.

            In this case, pick $z = 1$ and $w = |\lambda|$. Then $g\cdot z = \lambda^2 > 1$ and $g^{-1} w = |\lambda|^{-1} < 1$. Therefore, $(z,w)$ and $(g\cdot z, g^{-1}\cdot w)$ intersect transversally by \Cref{intersecting geodesics}.
    \end{itemize}  
    In each case, we showed that there is some choice of $z,w$ such that $(z,w)$ and $(g\cdot z, g^{-1}\cdot w)$ intersect transversally.
\end{proof}

\begin{lemma}\label{lem:nongeneric2-extcrossedgeimplies extension3}
    Suppose $v_1 = v_3,v_2 = v_4$ and $v_2\neq v_3$.
    Then $(G',\psi')$ is $\Gamma$-isostatic.
\end{lemma}

\begin{proof}
By \Cref{lem:treegain} assume that $g_1=g_2=\text{id}$ and, by the definition of gain graph, $g_3,g_4\neq\text{id}$. Then, by \Cref{lem:nongeneric2-extcrossedgeimplies extension1,lem:nongeneric2-extcrossedgeimplies extension2}, there are distinct $z,w\in\RR\,\cup\,\{\infty\}$ such that the geodesics $(z,w)$ and $(g_3\cdot z,g_4\cdot w)$ intersect transversally. By \Cref{lem: cross-boundary-cross-plane}, there is a choice of $p(v_1)=p(v_3)\neq p(v_2)=p(v_4)$ such that $\gamma_{p(v_1),p(v_2)}$ and $\gamma_{g_3\cdot p(v_1),g_4\cdot p(v_2)}=\gamma_{g_3\cdot p(v_3),g_4p\cdot (v_4)}$ are crossing. The result then follows from \Cref{lem:2-extcrossedgeimplies extension}.
\end{proof}

\subsubsection{Case (c). \texorpdfstring{$v_1=v_2=v_3=v_4$}{v1=v2=v3=v4}}\label{subsub:the hardest lemma on earth}

In this section we prove the following technical result, which applies to the case when $\Gamma$ is a surface group (see \Cref{lem:noparabolic}).

\begin{lemma}\label{lem:2ext2loopsat1vertex}
    Suppose $v_1 = v_2=v_3 = v_4$, and $\Gamma$ has no parabolic elements and no elements of order $2$. 
    Then $(G',\psi')$ is $\Gamma$-isostatic.
\end{lemma}

First, we describe the general strategy behind the proof. The lemma will follow if we show that the gain graph spanned by the vertex set $\{v,v_1\}$ is $\Gamma$-isostatic. 
To do this, we use the fact that the infinitesimal motions of a joint-configuration $(G,\psi,p)$ where $G$ has exactly one edge with trivial gain are given by $\eta(X, p)$ for $X\in \mathfrak{sl}(2,\RR)$. Assuming that $\psi'(e_4)=\text{id}$, we then use \Cref{inf_motions_edge} to describe, for each $i\in\{1,2,3\}$, the space $S_i\subseteq \mathfrak{sl}(2,\RR)$ such that $\eta(X, p)$ defines infinitesimal motions of $(G[e_i],\psi'|_{\{e_i\}},p|_{\{v,v_1\}})$. 
We then use a Lorentzian geometry argument to show that $\bigcap_{i=1}^{3} S_{i} =\{0\}$.

The proof of \Cref{lem:2ext2loopsat1vertex} requires some preliminary notions and results.
First, for any $g\in \text{PSL}(2,\RR)$, we define the \textit{adjoint representation of $\mathfrak{sl}(2,\RR)$ at $g$} to be $$\textup{Ad}(g) : \mathfrak{sl}(2, \mathbb{R}) \rightarrow \mathfrak{sl}(2, \mathbb{R}),\qquad X\mapsto gXg^{-1}$$
The adjoint representation induces a group action of $\text{PSL}(2,\RR)$ on $\mathfrak{sl}(2,\RR)$ given by the map $g\mapsto \text{Ad}(g)(X)$ for $g\in \text{PSL}(2,\RR)$ and $X\in\mathfrak{sl}(2,\RR)$. By noting that the space $\mathcal{T}_{\langle g \rangle}$ -- the tangent space of $C_{\PSL(2,\RR)}(\langle g \rangle)$ -- is the eigenspace of $\textup{Ad}(g)$ associated to the eigenvalue 1, \Cref{l: reformulation} follows directly from \Cref{lem:centraliser_dim} and its proof (see \Cref{eq:tgamma,eq:hg is a}).

\begin{lemma}
    \label{l: reformulation}
    For any non-identity $g\in \PSL(2,\RR)$, the linear map $\textup{Ad}(g)$ has exactly one eigenvector $H_g$ (up to scalar multiplication) associated with the eigenvalue 1. Moreover, we can choose $H_g$ such that \begin{align*}
    g =\begin{bmatrix}
         g_{11} & g_{12} \\
         g_{21} & g_{22}
    \end{bmatrix}\qquad  \implies \qquad 
    H_{g}&= 
        \begin{bmatrix}
         g_{11}-g_{22} & 2g_{12} \\
         2g_{21} & g_{22}-g_{11}
    \end{bmatrix} .
\end{align*}
\end{lemma}

Throughout this section, we use $H_g$ to denote the non-zero eigenvector of $g\in\text{PSL}(2,\RR)$ described in \Cref{l: reformulation}. Note that $\text{Ad}(g)(H_g)=H_g$ for all non-identity $g\in\text{PSL}(2,\RR)$. We define the \textit{Killing form} of $\mathfrak{sl}(2,\RR)$ to be the symmetric bilinear form 
\begin{equation*}
    B: (\mathfrak{sl}(2, \mathbb{R}))^2 \longrightarrow  \mathbb{R}, ~ (X,Y) \longmapsto 4 \textrm{Tr}(XY).
\end{equation*}
 We now revise some properties of the Killing form $B$ (see \cite[section 14.2]{MR1153249} for more details).
Since the trace of a matrix is invariant under conjugation, it is easy to see that the Killing form of $\mathfrak{sl}(2,\RR)$ is \textit{$\text{Ad}$-invariant}, i.e.~for all $g \in \PSL(2,\RR)$ and $X,Y \in \mathfrak{sl}(2,\RR)$, $$B \left(\text{Ad}(g)(X),\text{Ad}(g)(Y)\right) = B(X,Y).$$
Moreover, $B$ is \textit{non-degenerate}, i.e.~for all $X\in\mathfrak{sl}(2,\RR)$ there exists some $Y\in\mathfrak{sl}(2,\RR)$ such that $B(X,Y)\neq0$. For all $X\in\mathfrak{sl}(2,\RR)$ define $$X^\perp := \{ Y : B(X,Y) = 0\}.$$ As a consequence of the non-degeneracy of $B$, for any $X,Y\in\mathfrak{sl}(2,\RR)$, $X^\perp=Y^\perp$ if and only if $X$ and $Y$ are scalar multiples of each other. 
Any $2$-dimensional subspace $\Pi$ of $\mathfrak{sl}(2,\RR)$ can be written as $\{X:B(N,X)=0\}$, for some $N\in\mathfrak{sl}(2,\RR)$, which we call the \textit{normal vector} of $\Pi$. Then clearly $\Pi$ can be represented by its normal vector $N$ as $\Pi = N^\perp$. 
Moreover, since $B$ is $\text{Ad}$-invariant, for all $g\in \PSL(2,\RR)$, $\text{Ad}(g)(\Pi)$ may be represented by $\text{Ad}(g)(N)$.

To prove \Cref{lem:2ext2loopsat1vertex}, we require the following additional property of the Killing form of $\mathfrak{sl}(2,\RR)$.

\begin{lemma}
    \label{l: hg not orth to hg}
    Let $\Gamma$ be a Fuchsian group and $g\in\Gamma$ be non-parabolic. Then, $B(H_g,H_g) \neq 0$.
\end{lemma}

\begin{proof}

By \Cref{l: reformulation},
for $\lambda >1$ and $\theta \in (0,2\pi)$, there exists $x\neq 0$ such that
\begin{align*}
    H_{g}&= \begin{bmatrix}
         0 & -x \\
         x & 0
    \end{bmatrix} \quad \textup{ if } \quad g =\begin{bmatrix}
         \cos(\theta) & -\sin(\theta) \\
         \sin(\theta) & \cos(\theta)
    \end{bmatrix} \quad \implies \quad B(H_{g}, H_{g})= -8x^2 <0 \\
    H_{g}&= \begin{bmatrix}
         x & 0 \\
        0 & -x
    \end{bmatrix} \quad \textup{ if } \quad g =\begin{bmatrix}
         \lambda & 0 \\
          0 & \lambda^{-1} 
    \end{bmatrix} \quad \implies \quad B(H_{g}, H_{g})= 8x^2 >0.
\end{align*}
It can additionally be computed from \Cref{l: reformulation} that $H_{hgh^{-1}}$ is a non-zero scalar multiple of $\text{Ad}(h)(H_g)$.
The result now follows from conjugating the elliptic and hyperbolic elements into the above forms and the fact that $B$ is Ad-invariant.
\end{proof}

By identifying $\mathfrak{sl}(2, \mathbb{R})$ with $\mathbb{R}^{3}$ via the map 
\begin{equation}
        \label{eq: identify}
        \Phi:\mathfrak{sl}(2,\RR)\longrightarrow\RR^3, \,\hspace{2em} \begin{bmatrix}
    a & b \\ c & -a
\end{bmatrix}\longmapsto \begin{bmatrix}
    a \\ \frac{b+c}{2} \\ \frac{b-c}{2}
\end{bmatrix},
\end{equation}
we can identify $\PP(\mathfrak{sl}(2, \mathbb{R}))$ with $\mathbb{RP}^2$.  With this identification, the Killing form of $\mathfrak{sl}(2,\RR)$ is represented by the diagonal matrix $\text{diag}(1,1,-1)$, i.e.~$B:(\RR^3)^2\rightarrow\RR$ is defined by $$B((x_1,y_1,z_1),(x_2,y_2,z_2))=x_1x_2+y_1y_2-z_1z_2$$
for all $(x_1,y_1,z_1),(x_2,y_2,z_2)\in\RR^3$.

We define the \textit{Lorentzian cross product} for $\mathbb{R}^3$ to be the bilinear map $\wedge: \mathbb{R}^{3}\times \mathbb{R}^{3} \rightarrow \mathbb{R}^{3}$ such that for all $(x_1,y_1,z_1),(x_2,y_2,z_2)\in\mathbb{R}^3$,
    \begin{equation*}
        (x_1,y_1,z_1)\wedge(x_2,y_2,z_2)=(y_1z_2-y_2z_1,z_1x_2-x_1z_2,y_1x_2-y_2x_1).
    \end{equation*}
See \cite[Chapter 12.3]{baragar2001survey} for a reference on the Lorentzian cross product.
Using $\Phi$ allows us to define the Lorentzian cross product for $\mathfrak{sl}(2,\RR)$, which (by abuse of notation) we also denote by $\wedge$.
Computations show that the Lorentzian cross product is \textit{$\text{Ad}$-equivariant}, i.e.~$$\text{Ad}(g)(X\wedge Y) = \text{Ad}(g)(X)\wedge \text{Ad}(g)(Y)$$ for any $g \in \PSL(2,\RR)$ and any $X,Y \in \mathfrak{sl}(2,\RR)^3$. Moreover, for all $X,Y,Z \in \mathfrak{sl}(2,\RR)^3$ and $N_1,N_2 \in \mathfrak{sl}(2,\RR)$ the following hold:
    \begin{itemize}
        \item[(P1)] $X \wedge (Y \wedge Z) = B(X,Y) Z - B(X,Z)Y$;
        \item[(P2)] $X \wedge Y = \mathbf{0}$ if and only if $X,Y$ are linearly dependent;
        \item[(P3)] $B(X \wedge Y, X) = B(X \wedge Y,Y) = 0$;
        \item[(P4)] for subspaces $\Pi_1 = N_1^\perp$ and $\Pi_2= N_2^\perp$ of $\mathfrak{sl}(2,\RR)$, $N_1 \wedge N_2 \in \Pi_1 \,\cap\, \Pi_2$.
    \end{itemize}
Each of (P1), (P2) and (P3) can be verified through direct computations, and (P4) follows from (P3).

To prove \Cref{lem:2ext2loopsat1vertex}, we require the following technical lemmas regarding the adjoint representation, the Killing form, and the Lorentzian cross product.

    \begin{lemma}\label{claim:eigen of ad}
        If $g \in \PSL(2,\RR)$ is non-parabolic with eigenvalues\footnote{Technically, the element $g$ is two different matrices; a matrix with eigenvalues $\lambda_1,\lambda_2$, and a matrix with $-\lambda_1,-\lambda_2$. It is, however, clear from the result that the eigenvalues of the operator $\textup{Ad}(g)$ are independent of this choice.} $\lambda_1,\lambda_2$, then $\textup{Ad}(g)$ has eigenvalues $1, \lambda_1\lambda_2^{-1}, \lambda_2\lambda_1^{-1}$.
    \end{lemma}
    \begin{proof}
Suppose $g$ has eigenvalues $\lambda_1,\lambda_2$. Since $g$ is not parabolic, these two eigenvalues are distinct, as can be easily verified.  Taking a basis of $\mathbb{C}^2$ such that $g= \begin{bmatrix}
    \lambda_1 & 0\\
    0 & \lambda_2
\end{bmatrix}$, the following hold:
\begin{align*}
\text{Ad}(g)\left(\begin{bmatrix}
        1 & 0 \\
        0 & -1
    \end{bmatrix}\right) &= 
\begin{bmatrix}
    \lambda_1 & 0\\
    0 & \lambda_2
\end{bmatrix}
    \begin{bmatrix}
        1 & 0 \\
        0 & -1
    \end{bmatrix}
    \begin{bmatrix}
    \lambda_1^{-1} & 0\\
    0 & \lambda_2^{-1}
\end{bmatrix}  = \begin{bmatrix}
        1 & 0 \\
        0 & - 1
    \end{bmatrix},\\
\text{Ad}(g)\left(\begin{bmatrix}
        0 & 1 \\
        0 & 0
    \end{bmatrix}\right) &=\begin{bmatrix}
\lambda_1 & 0\\
0 & \lambda_2
\end{bmatrix}
\begin{bmatrix}
    0 & 1 \\
    0 & 0
\end{bmatrix}
\begin{bmatrix}
\lambda_1^{-1} & 0\\
0 & \lambda_2^{-1}
\end{bmatrix} = \begin{bmatrix}
    0 & \lambda_1\lambda_2^{-1} \\
    0 & 0
\end{bmatrix},\\
\text{Ad}(g)\left(\begin{bmatrix}
        0 & 0 \\
        1 & 0
    \end{bmatrix}\right) &=
\begin{bmatrix}
\lambda_1 & 0\\
0 & \lambda_2
\end{bmatrix}
\begin{bmatrix}
    0 & 0 \\
    1 & 0
\end{bmatrix}
\begin{bmatrix}
\lambda_1^{-1} & 0\\
0 & \lambda_2^{-1}
\end{bmatrix}  = \begin{bmatrix}
    0 & 0 \\
    \lambda_2\lambda_1^{-1} & 0
\end{bmatrix}.\\
\end{align*}
So $\begin{bmatrix}
    1 & 0 \\
    0 & -1
\end{bmatrix},
\begin{bmatrix}
    0 & 1 \\
    0 & 0
\end{bmatrix}, \begin{bmatrix}
    0 & 0 \\
    1 & 0
\end{bmatrix}$
are eigenvectors of $\textup{Ad}(g)$ (over $\mathbb{C}$) with eigenvalues $1,\lambda_1\lambda_2^{-1},\lambda_2\lambda_1^{-1}$, respectively. 
\end{proof}

    \begin{lemma}
    \label{claim:hgdg}
        For $g \in \PSL(2,\RR),A_1,A_2 \in \mathfrak{sl}(2,\RR)^3$, the vector $H_g \wedge (A_1\wedge \text{Ad}(g)(A_2))\wedge(\text{Ad}(g^{-1})(A_1)\wedge A_2)$ equals
         \begin{equation*}
             \lambda \Big( \big( \text{Ad}(g^{-1}) (A_1) \wedge A_2 \big) - \big(A_1 \wedge \text{Ad}(g)( A_2) \big) \Big) = \lambda \Big( \big( \text{Ad}(g^{-1}) (A_1) \wedge A_2 \big) -  \text{Ad}(g) \big(  \text{Ad}(g)^{-1}( A_1) \wedge A_2 \big) \Big)  ,
         \end{equation*}
         where $\lambda = B(H_g, A_1 \wedge  \text{Ad}(g)(A_2))$.
    \end{lemma}

     \begin{proof}
     By (P1), along with the fact that $\wedge$ is $\text{Ad}$-equivariant, $B$ is $\text{Ad}$-invariant, and $\text{Ad}(g)(H_g) = H_g$, the vector $H_g \wedge (A_1\wedge \text{Ad}(g)(A_2))\wedge(\text{Ad}(g^{-1})(A_1)\wedge A_2)$ equals
         \begin{align*}
             &\Big(B(H_g, A_1 \wedge \text{Ad}(g)(A_2)) \Big)( \text{Ad}(g)(A_1) \wedge A_2) -\Big( B(H_g,  \text{Ad}(g^{-1})(A_1) \wedge A_2) \Big) (A_1 \wedge \text{Ad}(g)(A_2)) \\
             =& \Big(B(H_g, A_1 \wedge \text{Ad}(g)(A_2)) \Big)( \text{Ad}(g^{-1})(A_1 \wedge A_2)) -\Big( B(\text{Ad}(g)(H_g),  A_1 \wedge \text{Ad}(g)(A_2)) \Big) (A_1 \wedge \text{Ad}(g)(A_2)) \\
             =& B(H_g, A_1 \wedge \text{Ad}(g)(A_2)) \Big( ( \text{Ad}(g^{-1})(A_1 \wedge A_2)) - (A_1 \wedge \text{Ad}(g)(A_2)) \Big) \\
             =& B(H_g, A_1 \wedge \text{Ad}(g)(A_2)) \Big( ( \text{Ad}(g^{-1})(A_1) \wedge A_2) - \text{Ad}(g)( \text{Ad}(g)(A_1) \wedge A_2) \Big). \qedhere
         \end{align*}
     \end{proof}

\begin{lemma}\label{lem:hard lemma}
    Let $\Gamma=\langle g_1,g_2,g_3\rangle$ be a non-cyclic Fuchsian group with no parabolic elements and no elements of order two. Suppose further that $g_i\neq g_j$ for all $1\leq i\neq j\leq 3$.
    Then, there exists a dense open subset $U \subset \mathfrak{sl}(2,\RR)^2$ such that for all $(A_1, A_2) \in U$ the vectors
    \begin{equation}\label{eq:linwedgevecs2}
        H_{g_1} \wedge \Big(\text{Ad}(g_1^{-1})(A_1)  \wedge A_2 \Big), \qquad H_{g_2} \wedge  \Big(\text{Ad}(g_2^{-1})(A_1)   \wedge A_2 \Big), \qquad H_{g_3} \wedge  \Big(\text{Ad}(g_3^{-1})(A_1)   \wedge A_2 \Big)
    \end{equation}    
    are linearly independent.
\end{lemma}

\begin{proof}
 By \Cref{lem:cyclic-iff-fix}, we may assume that $\langle g_i,g_3\rangle$ is not cyclic for $i=1,2$.
    As each map $(A_1,A_2) \mapsto H_{g_i} \wedge ( \text{Ad}(g_i^{-1})(A_1)  \wedge A_2)$ is a polynomial map, it suffices to find any $A_1,A_2 \in \mathfrak{sl}(2,\RR)$ for which the vectors in \Cref{eq:linwedgevecs2} are linearly independent. 

    For each $i \in \{1,2,3\}$,
    choose some arbitrary non-zero vector $Z_i \in H_{g_i}^\perp$.
    Fix $A_2 = Z_3$ and choose
    \begin{equation*}
        A_1 \in L(Z_1,Z_2,Z_3) := \Big\langle \text{Ad}(g_1) (Z_3) \, , \, \text{Ad}(g_1) (Z_1) \Big\rangle \,\cap\, \Big\langle \text{Ad}(g_2) (Z_3) \, , \, \text{Ad}(g_2) (Z_2) \Big\rangle.
    \end{equation*}
    Let $\lambda_1,\mu_1,\lambda_2,\mu_2 \in \RR$ be scalars such that 
    \begin{equation*}
        A_1 = \lambda_1 \text{Ad}(g_1) (Z_3) + \mu_1 \text{Ad}(g_1) (Z_1) = \lambda_2 \text{Ad}(g_2) (Z_3) + \mu_2 \text{Ad}(g_2) (Z_2).
    \end{equation*}
    By (P1), (P2) and the fact that $Z_i\in H_{g_i}^\perp$, for $i \in \{1,2\}$ we have 
    \begin{align*}
        H_{g_i} \wedge \Big(\text{Ad}(g_i^{-1})(A_1)  \wedge A_2 \Big) &= H_{g_i} \wedge \Big( \big(\lambda_i Z_3 + \mu_i Z_i \big) \wedge Z_3 \Big)  \\
        &=  H_{g_i} \wedge \Big( \lambda_i \big(Z_3 \wedge Z_3  \big)+ \mu_i  \big(Z_i \wedge Z_3 \big)  \Big) \\
        &= \mu_i H_{g_i} \wedge \Big( Z_i  \wedge Z_3  \Big) \\
        &= \mu_i \Big( B( H_{g_i} \, , \, Z_i) Z_3 - B(H_{g_i} \, , \, Z_3) Z_i \Big) \\
        &= - \mu_i B(H_{g_i} \, , \, Z_3) Z_i.
    \end{align*}
    Moreover, by (P1) and the fact that $Z_3\in H_{g_3}^\perp$, we have
    \begin{align*}
        H_{g_3} \wedge \Big(\text{Ad}(g_3^{-1})(A_1)  \wedge A_2 \Big) &= H_{g_3} \wedge \Big(\text{Ad}(g_3^{-1})(A_1)  \wedge Z_3\Big) \\
        &= B \big(H_{g_3} \, , \, \text{Ad}(g_3^{-1})(A_1)\big) Z_3 - B \big(H_{g_3} \, , \, Z_3 \big) \text{Ad}(g_3^{-1})(A_1) \\
        &= B \big(H_{g_3} \, , \, A_1\big) Z_3.
    \end{align*}
    Therefore, there exists a pair $A_1,A_2$ for which the vectors in \Cref{eq:linwedgevecs2} are linearly independent if the following properties hold:
    \begin{enumerate}
        \item $H_{g_i}^\perp \,\cap\, H_{g_3}^\perp$ is 1-dimensional for $i=1,2$: if this is the case, then we can choose $Z_1,Z_2,Z_3$ to be linearly independent. In particular, since $Z_i\in H_{g_i}^\perp$ for $i\in\{1,2,3\}$, this ensures that $Z_3\not\in H_{g_1}^\perp,H_{g_2}^\perp$.

        \item There exists a choice of $Z_1,Z_2,Z_3$ such that $L(Z_1,Z_2,Z_3)$ contains neither $\text{Ad}(g_1) (Z_3)$ nor $\text{Ad}(g_2)(Z_3)$ (and hence is 1-dimensional). This is so we can choose $A_1$ so that $\mu_1,\mu_2\neq0$.
        \item There exists a choice of $Z_1,Z_2,Z_3$ such that $B(H_{g_3},A_1) \neq 0$ for some non-zero $A_1 \in L(Z_1,Z_2,Z_3)$.
    \end{enumerate}
    We now show each property holds individually.

    \medskip\noindent
    \textbf{\textit{Property (i):}}
    Fix some $i\in\{1,2\}$ and recall that $\langle g_i,g_3\rangle$ is non-cyclic. Therefore, $H_{g_i}$ is not a scalar multiple of $H_{g_3}$ by \Cref{lem:centraliser_dim}, and so $H_{g_i}^{\perp}\neq H_{g_3}^{\perp}$. This implies that $\dim(H_{g_i}^\perp\,\cap\, H_{g_3}^\perp)=1$, proving (i). 

    \medskip\noindent
    \textbf{\textit{Property (ii):}} 
    Suppose for contradiction that Property (ii) does not hold.
    For each $i \in \{1,2\}$, the inclusion of $\text{Ad}(g_i) (Z_3)$ in $L(Z_1,Z_2,Z_3)$ is determined by a polynomial equation.
    By applying the same reasoning up until now for $Z_3$ to both $Z_1$ and $Z_2$, we may suppose that $\langle \text{Ad}(g_1) (Z_3) , \text{Ad}(g_1) (Z_1) \rangle$ contains $\text{Ad}(g_2) (Z_3)$ for all choices of $Z_1,Z_3$. 
    This in turn implies that the vectors $Z_1, Z_3, \text{Ad}(g_1^{-1}g_2) (Z_3)$ are linearly dependent for all choices of $Z_1,Z_3$.
    If we choose $Z_3 \notin H_{g_1}^\perp$ (possible by Property (i)) then we can choose $Z_1$ to not lie in $\langle Z_3, \text{Ad}(g_1^{-1}g_2)(Z_3) \rangle$; this is as the latter space is not equal to the 2-dimensional space $H_{g_1}^\perp$.
    Thus, the vectors $Z_3$, $\text{Ad}(g_1^{-1}g_2)(Z_3)$ are linearly dependent. Equivalently, $Z_3$ is an eigenvector of $\text{Ad}(g_1^{-1}g_2)$ for all $Z_3 \in H_{g_3}^\perp\setminus H_{g_1}^\perp$.
    Having a 2-dimensional space of eigenvectors implies that $\text{Ad}(g_1^{-1}g_2)$ has repeated eigenvalues.

    Recall that, by assumption, $g_1^{-1} g_2 \neq \id$. Since conjugation does not effect eigenvalues, either $g_1^{-1}g_2$ is elliptic with eigenvalues $e^{\mathfrak{i} \theta},e^{-\mathfrak{i} \theta}$ for some $\theta \in (0,2\pi)$,
    or $g_1^{-1}g_2$ is hyperbolic with eigenvalues $\lambda, \lambda^{-1}$ for some $\lambda >1$.
    When $g_1^{-1}g_2$ is hyperbolic, \Cref{claim:eigen of ad} gives the eigenvalues of $\textup{Ad}(g_1^{-1}g_2)$ to be $1, \lambda^2, \lambda^{-2}$ which are distinct,
    contradicting that it has repeated eigenvalues.
    When $g_1^{-1}g_2$ is elliptic,
    \Cref{claim:eigen of ad} gives the eigenvalues of $\textup{Ad}(g_1^{-1}g_2)$ to be $1, e^{2\mathfrak{i} \theta}, e^{-2\mathfrak{i} \theta}$.
    As it has repeated eigenvalues, $\theta \in \{\pi/2,\pi\}$.
    However, this implies that $(g_1^{-1} g_2)^4 = \id$, contradicting that $\Gamma$ contains no elements of order 2.
    (The $\theta=\pi$ case corresponds to $g_1^{-1}g_2$ being congruent to the matrix $-\id$, which is modded out in $\PSL(2,\RR)$.) So (ii) holds.

    \medskip\noindent
    \textbf{\textit{Property (iii):}} 
    In the proof of (i), we saw that $H_{g_i}^\perp\neq H_{g_3}^\perp$ for $i\in\{1,2\}$. Applying the adjoint map, we see that for $i\in\{1,2\}$, $\text{Ad}(g_i)(H_{g_i}^\perp)\neq \text{Ad}(g_i)(H_{g_3}^\perp)$, and so $\dim(\text{Ad}(g_i)(H_{g_i}^\perp)\,\cap\, \text{Ad}(g_i)(H_{g_3}^\perp))=1$. Hence, we can choose $Z_3$ such that $\text{Ad}(g_i) (Z_3) \notin \text{Ad}(g_i)(H_{g_i}^\perp)$ for $i\in\{1,2\}$.
    Pick $Y\in \mathfrak{sl}(2,\mathbb{R})$ such that 
    \begin{align*}
        Y\notin \Big\langle \text{Ad}(g_1)(Z_3) ~,~ \text{Ad}(g_2)(Z_3)\Big\rangle \qquad \mbox{ and } \qquad 
        Y\notin \bigcup_{i=1}^{3} H_{g_i}^{\perp}.
    \end{align*}
     The conditions say $Y$ does not belong to a finite union of $2$-dimensional subspaces, so such a $Y$ exists. We will show we can choose $Z_1, Z_2$ such that $Y\in L(Z_1,Z_2,Z_3)$. We have
    \begin{equation*}
        \dim\Big(\Big\langle Y\, , \, \text{Ad}(g_1)(Z_3) \Big\rangle\Big)=\dim\Big(\Big\langle Y\, , \, \text{Ad}(g_2)(Z_3) \Big\rangle\Big)=2.
    \end{equation*}
    Hence, as $\text{Ad}(g_i)(Z_3)\notin \text{Ad}(g_i)(H_{g_i}^{\perp})=H_{g_i}^\perp$ for $i=1,2$, we have
    \begin{align*}
        \dim\left(\Big\langle Y\, , \, \text{Ad}(g_i)(Z_3) \Big\rangle\,\cap\, H_{g_i}^{\perp}\right) &=1,
    \end{align*} 
    and thus we may pick some non-zero     \begin{align*}
        \tilde{Z}_i &\in \left(\Big\langle Y\, , \, \text{Ad}(g_i)(Z_3) \Big\rangle\,\cap\, H_{g_i}^{\perp}\right).
    \end{align*}
    Since $H_{g_i}$ is invariant under conjugation and the Killing form in $\text{Ad}$-invariant, $$B(\text{Ad}(g_i)(X), H_{g_i}) = B(X, H_{g_i}).$$ It follows that the map $\text{Ad}(g_i)$ is a bijection when restricted to $H_{g_i}^{\perp}$. Thus we may pick $Z_1, Z_2$ such that $\textup{Ad}(g_i)(Z_i) =\tilde{Z}_{i}$ for $i=1,2$. With this choice of $Z_1, Z_2$ we have 
    \begin{align*}
        \langle \textup{Ad}(g_i)(Z_i),\textup{Ad}(g_i)(Z_3)\rangle = \langle Y ,\textup{Ad}(g_i)(Z_3)\rangle 
    \end{align*}
    
    Thus, using the definition of $L(Z_1, Z_2, Z_3),$ it follows that $\langle Y \rangle \subseteq L(Z_1, Z_2, Z_3)$
    and we have $Y\notin H_{g_3}^{\perp},$ i.e.~$B(Y, H_{g_3})\neq 0$. 
\end{proof}

     We now have all the background material required to prove \Cref{lem: case_x}, which is the key technical result from which \Cref{lem:2ext2loopsat1vertex} follows.

\begin{lemma}\label{lem: case_x}
    Let $\Gamma=\langle g_1,g_2,g_3\rangle$ be a non-cyclic Fuchsian group with no parabolic elements and no elements of order two.
    Then, a $\Gamma$-gain graph $(G, \psi)$ with exactly two vertices $u,v$ and four edges $u \xrightarrow{\textup{id}} v$, $u \xrightarrow{g_1} v$, $u \xrightarrow{g_2} v$, and $u\xrightarrow{g_3} v$ is $\Gamma$-isostatic. 
\end{lemma}

\begin{proof}
For $g\in\{\text{id},g_1,g_2,g_3\}$, let $e_g$ denote the edge $u\xrightarrow{g}v$ and $G[e_g]$ to be the subgraph of $G$ with the sole edge $e_g$.
If, for some distinct pair $i,j \in \{1,2,3\}$,
the subgroup $\langle g_i, g_j \rangle$ is cyclic,
then $G$ is near-cyclic and the result holds by \Cref{lem:near}.
So assume that each subgroup $\langle g_i, g_j \rangle$, $i \neq j$, is non-cyclic.

By \Cref{inf_motions_edge}, for any $i\in\{1,2,3\}$ and any realisation $p:\{u,v\}\rightarrow\HH$ of $(G,\psi)$ such that $p(u),p(v),g_i\cdot p(v)$ and $g_i^{-1}\cdot p(u)$  do not lie on the same geodesic, the space of infinitesimal motions of $(G[e_{g_i}],\psi|_{\{e_{g_i}\}},p)$ is
\begin{equation*}
    S_i(p) := \mathcal{T}_{\langle g_i\rangle}  + \left(\left( \mathfrak{h}_{p(u)} +\mathfrak{h}_{g_i\cdot p(v)} \right) \,\cap\, \left( \mathfrak{h}_{p(v)} +\mathfrak{h}_{g_i^{-1}\cdot p(u)} \right)\right).
\end{equation*}
Since each $g_i \neq \text{id}$, the set of realisations $p$ where $p(u),p(v),g_i\cdot p(v)$ and $g_i^{-1}\cdot p(u)$ do not lie on a shared geodesic forms a dense open subset of $\HH^{V(G)}$. (This follows from the locus of a real analytic function with connected domain being a closed nowhere dense subset; see, for example, \cite{Mityagin2020}.) So, by \Cref{inf_motions_edge}, for all $1 \leq i\leq 3$, there is a dense open subset $U_i$ of $\HH^2$ such that $S_i(p)$ is a (well-defined) two-dimensional space for all $p\in U_i$. Similar to the proof of \Cref{base_case_lemma}, it suffices to show that there exists a choice of $p$ such that $S_1(p)\,\cap\, S_2(p)\,\cap\, S_3(p) = \{\textbf{0}_{2\times 2}\}$.

Recall from \Cref{lem: Lie -algebras-stabilizers1} that, for $p\in \HH^{\{u,v\}} \cong \HH^2$, we have $\mathfrak{h}_{g_i\cdot p(v)} = g_i \mathfrak{h}_{ p(v)} g_i^{-1} = \textup{Ad}(g_i)(\mathfrak{h}_{p(v)})$. For $A_1,A_2\in\mathfrak{sl}(2,\RR)$ and $1\leq i\leq 3$, define
\begin{equation*}
    S_i'(A_1,A_2) = \mathcal{T}_{\langle g_i\rangle}  + \Big(\big\langle A_1,\text{Ad}(g_i)(A_2)\big\rangle
    \,\cap\, \big\langle \text{Ad}(g_i^{-1})(A_1),A_2 \big\rangle \Big).
\end{equation*}
For any $A_1 \in \mathfrak{h}_{p(u)}$, $A_2 \in \mathfrak{h}_{p(v)}$,
we have
\begin{equation*}
    \big\langle A_1,\text{Ad}(g_i)(A_2)\big\rangle = \mathfrak{h}_{p(u)} +\mathfrak{h}_{g_i\cdot p(v)}, \qquad \big\langle A_1,\text{Ad}(g_i^{-1})(A_2)\big\rangle = \mathfrak{h}_{p(u)} +\mathfrak{h}_{g_i^{-1}\cdot p(v)}.
\end{equation*}
As the set $\{\mathfrak{h}_z \in \PP(\mathfrak{sl}(2, \mathbb{R})) ~\vert ~ z\in \mathbb{H}\}$ is an open subset of $\PP(\mathfrak{sl}(2, \mathbb{R}))$ (\Cref{lem: Lie -algebras-stabilizers0}),
it thus suffices to show that 
$$O:=\Big\{(A_1,A_2)\in\mathfrak{sl}(2,\RR):S_1'(A_1,A_2)\,\cap\,S_2'(A_1,A_2)\,\cap\,S_3'(A_1,A_2)=\{\textbf{0}_{2\times 2}\}\Big\}$$
is a dense open subset of $(\mathfrak{sl}(2,\RR))^2$.
By (P3) and (P4), there is a dense open subset of choices of $A_1,A_2$ such that
   $\langle A_1 , \text{Ad}(g)(A_2)\rangle \,\cap\, \langle \text{Ad}(g^{-1})(A_1), A_2\rangle$ 
   is the exactly the space consisting of scalar multiples of $D_g(A_1,A_2) :=N_1(A_1,A_2)\wedge N_2(A_1,A_2)$, where
    \begin{equation*}
      N_1(A_1,A_2) = A_1 \wedge \text{Ad}(g)(A_2)  \qquad N_2(A_1,A_2)=  \text{Ad}(g^{-1})(A_1) \wedge A_2  .
    \end{equation*}
    As proved in \Cref{inf_motions_edge}, there is a dense open subset of pairs $(A_1,A_2)$ where $D_g(A_1,A_2) \neq \mathbf{0}$. 
    Since $\mathcal{T}_{\langle g_i\rangle} = \langle H_{g_i} \rangle $ and since each map $(A_1,A_2) \mapsto D_{g_i}(A_1,A_2)$ is a polynomial map, 
    it now suffices to show that there exists a dense open subset of pairs $A_1,A_2 \in \mathfrak{sl}(2, \mathbb{R})$ such that
    \begin{equation*}
        \bigcap_{i=1}^3\Big\langle H_{g_i}, D_{g_i}(A_1,A_2) \Big\rangle  = \{\textbf{0}\}.
    \end{equation*}
    Using the Lorentzian cross product,
    this is equivalent to finding a dense open subset of pairs $A_1,A_2 \in \mathfrak{sl}(2, \mathbb{R})$ such that the vectors
    \begin{equation}\label{eq:detm replacement}
        H_{g_1} \wedge D_{g_1}(A_1,A_2), \qquad 
            H_{g_2} \wedge D_{g_2}(A_1,A_2), \qquad H_{g_3} \wedge D_{g_3}(A_1,A_2)
    \end{equation}
    are linearly independent.

    It is easy to check that $B(H_g, A_1 \wedge  \text{Ad}(g)(A_2)) \neq 0$ for a dense open subset of pairs $(A_1,A_2)$. 
    Thus, by \Cref{claim:hgdg} there exists a dense open set of pairs $(A_1,A_2)$ such that the vectors in \Cref{eq:detm replacement} are linearly independent if and only if there exists a dense open set of pairs $(A_1,A_2)$ such that the vectors
    \begin{equation}\label{eq:linwedgevecs}
        \Big(\text{Ad}(g_i^{-1})(A_1) \wedge A_2 \Big) - \text{Ad}(g_i)\Big( \text{Ad}(g_i^{-1})( A_1) \wedge A_2 \Big), \qquad 1\leq i\leq 3
    \end{equation}
    are linearly independent. 
    Moreover, by \Cref{claim:hgdg} and (P3), for each $1\leq i\leq 3$ and almost all choices of $A_1,A_2$, there is a scalar $\lambda\neq0$ such that
    \begin{align*}
        B \Big(H_{g_i} \, , \, \big( \text{Ad}(g_i^{-1})(A_1) \wedge A_2 \big) - \text{Ad}(g_i) \big( \text{Ad}(g_i^{-1})(A_1) \wedge A_2 \big) \Big)
        &=B \Big(H_{g_i} \, , \, \lambda^{-1}\big(H_{g_i}\wedge D_g(A_1,A_2) \big) \Big)\\
        &=\lambda^{-1}B \big(H_{g_i} \, , \, H_{g_i}\wedge D_g(A_1,A_2) \big)\\
        &=0,
    \end{align*}
    and so each of the vectors in \Cref{eq:linwedgevecs} lies in its corresponding plane $H_{g_i}^\perp$.

    One can easily check that each of the two maps $Z\mapsto Z - \text{Ad}(g_i)(Z)$ and $Z \mapsto H_{g_i} \wedge Z$ is a linear bijection on $H_{g_i}^{\perp}$: the first map as $H_{g_i}$ is the only eigenvector with eigenvalue $1$ for $g_i$, and the second map as $H_{g_i}$ is not contained in $H_{g_i}^{\perp}$ (\Cref{l: hg not orth to hg}).
    It follows that the map
    \begin{equation*}
        \prod_{i=1}^3 H_{g_i}^{\perp} \longrightarrow \prod_{i=1}^3 H_{g_i}^{\perp},
        ~
        (Z_i)_{i \in \{1,2,3\}} \longmapsto \Big( H_{g_i} \wedge \big(Z_i - \text{Ad}(g_i)(Z_i) \big) \Big)_{i \in \{1,2,3\}}
    \end{equation*}
    is a well-defined linear bijection.
    Now, for each $i \in \{1,2,3\}$,
    consider the map
    \begin{equation*}
        \widetilde{\varphi}_i: \mathfrak{sl}(2, \mathbb{R}) \times \mathfrak{sl}(2, \mathbb{R}) \longrightarrow H_{g_i}^{\perp}, ~ (A_1, A_2)\longmapsto H_{g_i} \wedge \big(\text{Ad}(g_i^{-1})(A_1)\wedge  A_2 \big).
    \end{equation*} 
    Using that $\wedge$ is Ad-equivariant, \Cref{claim:hgdg}, (P1), and (P3), we have for all $A_1,A_2$ that
    \begin{align*}
        & H_{g_i} \wedge \Big(\widetilde{\varphi}_i(A_1,A_2) - \text{Ad}(g_i)( \widetilde{\varphi}_i(A_1,A_2)) \Big)\\
        = ~& H_{g_i} \wedge \Big( H_{g_i} \wedge \big(\text{Ad}(g_i^{-1})(A_1)\wedge  A_2 \big) - \text{Ad}(g_i)(H_{g_i}) \wedge \text{Ad}(g_i)\big(\text{Ad}(g_i^{-1})(A_1)\wedge  A_2 \big) \Big)\\
        = ~& H_{g_i} \wedge \Big(H_{g_i} \wedge \Big( \big(\text{Ad}(g_i^{-1})(A_1)\wedge  A_2 \big) - \big(A_1\wedge  \text{Ad}(g_i)(A_2) \big) \Big) \Big) \\
        = ~& B(H_{g_i},H_{g_i}) \Big( \big(\text{Ad}(g_i^{-1})(A_1)\wedge  A_2 \big) - \big(A_1\wedge \text{Ad}(g_i)(A_2) \big) \Big).
    \end{align*}
    For $1\leq i\leq 3$, since $g_i$ is not parabolic, $B(H_{g_i},H_{g_i})\neq0$ by \Cref{l: hg not orth to hg}.
    So the vectors in \Cref{eq:linwedgevecs} are linearly independent if and only the vectors
    \begin{equation*}
        H_{g_1} \wedge \Big(\text{Ad}(g_1^{-1})(A_1)  \wedge A_2 \Big), \qquad H_{g_2} \wedge  \Big(\text{Ad}(g_2^{-1})(A_1)   \wedge A_2 \Big), \qquad H_{g_3} \wedge  \Big(\text{Ad}(g_3^{-1})(A_1)   \wedge A_2 \Big)
    \end{equation*}
    are linearly independent. 
    So it suffices to show that there exists a dense open set of pairs $(A_1,A_2)$ in $\mathfrak{sl}(2,\RR)^2$ such that the vectors in \Cref{eq:linwedgevecs2} are linearly independent. 
    As this is guaranteed by \Cref{lem:hard lemma}, we achieve our desired result.
\end{proof}

We can now prove \Cref{lem:2ext2loopsat1vertex}.

\begin{proof}[Proof of \Cref{lem:2ext2loopsat1vertex}]
    If $v_1 = v_2 =v_3 =v_4$, then the two rows in the rigidity matrix coming from the loops at $v_1$ are independent. Hence, it suffices to show that the four rows resulting from the new edges $v \xrightarrow{g_i} v_1$ are also linearly independent. This follows from \Cref{lem: case_x}, \Cref{prop:balgain} and \Cref{lem: case_x}.
\end{proof}

With this, we are now ready to prove \Cref{2-ext}.
\begin{proof}[Proof of \Cref{2-ext}]
    Since $\Gamma$ is a surface group, it contains no parabolic elements and no element of order 2 (\Cref{lem:noparabolic}).
    Depending on the type of 2-extension (see \Cref{fig:2ext}), the $\Gamma$-gain graph $(G',\psi')$ is $\Gamma$-isostatic by either \Cref{lem:easy2-exts1}, \Cref{lem:easy2-exts2}, \Cref{lem:nongeneric2-extcrossedgeimplies extension3}, or \Cref{lem:2ext2loopsat1vertex}.
\end{proof}

\section{Rigidity on compact Riemann surfaces}\label{sec:riemann}

In this section, we reinterpret our main results in the setting of compact Riemann surfaces, showing that they yield a combinatorial characterisation of generic rigidity for graphs on surfaces of genus $g \geq 2$.

Let $S$ be a Riemann surface.
Unlike with the hyperbolic plane $\HH$, there can be more than one geodesic $\gamma$ in $S$ with start $\gamma(0)=P$ and end $\gamma(1)=Q$ for distinct points $P,Q \in S$.
However, we can define a \emph{framework in $S$} as a triple $(G,z,p)$,
where $G=(V,E)$ is a (multi)graph, $p:V \rightarrow S$, and $z=(z(e))_{e \in E}$ is an assignment of geodesics to each edge $e$ connecting vertices $u,v$ such that (given a fixed arbitrary orientation of $e$) $z(e)(0)=p(u)$ and $z(e)(1) = p(v)$.
A \emph{continuous flex} of $(G,z,p)$ now consists of a continuous family $((G,p_t,z_t))_{t \in [0,1]}$ of frameworks in $S$ such that $(G,z_0,p_0)=(G,z,p)$ and for each $e \in E$ the map $t \mapsto z_t(e)$ is a homotopy where the length of $z_t(e)$ between $z_t(e)(0)$ and $z_t(e)(1)$ remains constant.
We then say that $(G,z,p)$ is \emph{continuously rigid} if every continuous flex is \emph{trivial}, in that the map $t \mapsto (G,z_t,p_t)$ corresponds to a continuous family of isometries of $S$.
If $S$ has finitely many isometries, then a continuous flex is trivial if and only if $(G,z_t,p_t)=(G,z,p)$ for all $t \in [0,1]$.

Now let $S_g$ be a compact Riemann surface of genus $g \geq 2$ with constant curvature.
In this specific case we can remove the need to keep track of geodesics by utilising the Uniformisation Theorem \cite{Unif2,Unif1}.
Let $\Gamma \leq \PSL(2,\RR)$ be the surface group for which $S_g = \HH/\Gamma$ and let $\varphi_\Gamma : \HH \rightarrow S_g$ be the quotient map.
For the framework $(G,z,p)$,
we construct the $\Gamma$-joint-configuration $(G,\psi,\tilde{p})$ by choosing, for each $v \in V(G)$, a point $\tilde{p}(v) \in \varphi_\Gamma^{-1}(p(v))$, and setting an edge $e$ from $u$ to $v$ to have label $h \in \Gamma$ if the image of the geodesic $\gamma_{\tilde{p}(u), h \cdot \tilde{p}(v)}$ is $z_e$.
(Since $\Gamma$ acts freely and properly discontinuously on $\HH$,
the choice of group element $\gamma$ is unique.)
We now say that $(G,\psi,\tilde{p})$ is a \emph{corresponding $\Gamma$-joint-configuration} for $(G,z,p)$.

\begin{proposition}\label{prop:surfacerigidity}
    Let $(G,z,p)$ be a framework in $S_g = \HH/\Gamma$ and let $(G,\psi,\tilde{p})$ be a corresponding $\Gamma$-joint-configuration. 
    Then $(G,z,p)$ is continuously rigid if and only if $(G,\psi,\tilde{p})$ is locally $\Gamma$-symmetrically rigid. 
\end{proposition}

\begin{proof}[Sketch proof]
    There is a one-to-one correspondence between the continuous flexes of $(G,z,p)$ and the continuous paths $(\tilde{p}_t)_{t\in [0,1]}$ in $\HH^{V(G)}$ with $\tilde{p}_0=\tilde{p}$ for which $f_{G,\psi}(\tilde{p}_t)=f_{G,\psi}(\tilde{p})$.
    Moreover, a continuous flex of $(G,z,p)$ is trivial if and only if the corresponding continuous path $(\tilde{p}_t)_{t\in [0,1]}$ is constant.
    Using that $f_{G,\psi}$ is a real analytic function,
    the proof of \cite[Proposition 1]{Asimow1978} can be adapted to show that $(G,\psi,\tilde{p})$ is locally $\Gamma$-symmetrically rigid if and only if any continuous path $(\tilde{p}_t)_{t\in [0,1]}$ with $\tilde{p}_0=\tilde{p}$ and $f_{G,\psi}(\tilde{p}_t)=f_{G,\psi}(\tilde{p})$ is the constant map.
\end{proof}

To discuss `generic' rigidity for $S_g$, we define the following concept. We define a framework $(G,z,p)$ in $S_g=\HH/\Gamma$ to be \emph{regular} if there exists $(G,\psi,\tilde{p})$ be a corresponding $\Gamma$-joint-configuration that is regular.
As the set of regular realisations of a $\Gamma$-gain graph is a dense open set in $\HH^{V(G)}$, 
it can be shown that there exists a comeagre\footnote{A subset of a measurable space is \emph{comeagre} if it contains the countable intersection of open dense sets. By the Baire Category Theorem, every comeagre subset of $S_g^{V(G)}$ is dense.} subset $U \subseteq S_g^{V(G)}$ for which $(G,z,p)$ is regular for any choice of geodesics $z$.

The following rigidity characterisation now follows immediately from \Cref{main_theorem}, \Cref{prop:asimowroth} and \Cref{prop:surfacerigidity}.

\begin{corollary}\label{cor:surfaces}
    Let $\Gamma \leq \PSL(2,\mathbb{R})$ be a surface group of genus $g \geq 2$, and let $(G,z,p)$ be a regular framework in $S_g=\HH/\Gamma$ with corresponding $\Gamma$-joint-configuration $(G,\psi,\tilde{p})$. 
    Then $(G,z,p)$ is continuously rigid if and only if $(G,\psi)$ contains a spanning $(2,3,1,0)$-gain tight $\Gamma$-gain graph.
\end{corollary}

In particular, this provides, to our knowledge, the first combinatorial characterisation of rigidity for generic frameworks on compact Riemann surfaces of genus 
$g\geq 2$. This extends existing intrinsic rigidity theories for surfaces with continuous isometry groups, such as the sphere and the torus, and complements extrinsic rigidity results for surfaces \cite{NOP1,NOP2}.

\begin{remark}\label{rem:noparabolic} 
A natural generalisation to \Cref{cor:surfaces} would consider quotients of the form $\HH/\Gamma$ where $\Gamma$ is a co-compact Fuchsian group that is not a surface group.
Such quotient spaces are examples of \emph{orbifolds}; roughly speaking, a generalisation of a smooth manifold that allows for singularities.
Extending \Cref{cor:surfaces} to such structures would require a generalisation of \Cref{main_theorem} to arbitrary co-compact Fuchsian groups.
Co-compact Fuchsian groups contain no parabolic elements \cite[Corollary 4.2.7]{katok1992fuchsian},
and so the main obstacles to this extension are:
(i) some of our technical lemmas require we have no finite-order group elements;
and (ii) \Cref{l:one non-cyclic} only applies to two types of co-compact Fuchsian groups, namely free groups and surface groups.
\end{remark} 

We close with a related conjecture. Alternatively one may define framework rigidity for $S_g$ using only shortest-length geodesics between points as follows.
Let $G$ be a simple graph and $p:V(G) \rightarrow S_g$ be a \emph{well-positioned realisation}; i.e.~a map where there exists a single shortest-length geodesic between any two vertices $p(u),p(v)$ if $uv$ is an edge of $G$.
Fix $z$ to be the family of shortest-length geodesics corresponding to each edge of $G$.
Define $(G,p)$ to be \emph{continuously rigid in $S_g$} if and only if $(G,z,p)$ is continuously rigid in $S_g$,
and we define $G$ to be \emph{rigid in $S_g$} if there exists a well-positioned realisation $p:V(G) \rightarrow S_g$ where $(G,z,p)$ is both regular and continuously rigid.
It is easy to check that a graph being rigid in this way guarantees the existence of a non-empty open set of well-positioned realisations that are continuously rigid.
The disadvantage of this concept of rigidity is that it is no longer a generic property, since different realisations can produce different gain maps. Hence we would arrive at a different family of graphs if the definition of rigid in $S_g$ required all well-positioned regular realisations to be rigid.
Nevertheless we conjecture the following combinatorial characterisation.

\begin{conjecture}
    For $g \geq 2$, a simple graph is rigid in $S_g$ if and only if it contains a $(2,0)$-tight spanning subgraph.
\end{conjecture}

\subsection*{Acknowledgements}
This project emerged from a Heilbronn Institute for Mathematical Research Focused Research Grant `Structural Stability: Combinatorics, Geometry and Topology'. 
S.\,D.\ and A.\,L.\,P.\ were supported by the KU Leuven grant iBOF/23/064 and the FWO grants G0F5921N (Odysseus) and G023721N.
S.\,D.\ was also supported by the Heilbronn Institute for Mathematical Research.
A.\,L.\,P.\ and A.\,N.\ were supported by UK Research and Innovation (grant number UKRI1112), under the EPSRC Mathematical Sciences Small Grant scheme.
A.\,N.\ was also supported by EPSRC grant EP/X036723/1.
J.\,V.\ was supported by the Wallenberg AI, autonomous systems, and software program (WASP) funded by the Knut and Alice Wallenberg foundation. 
For the purpose of open access, the authors have applied a Creative Commons Attribution (CC-BY) licence to any Author Accepted Manuscript version arising. 
\bibliographystyle{plainurl}
\bibliography{ref}

\appendix

\section{Recursive construction of tight graphs}
\label{appA}
In this paper, we employed two results from \cite{jzt2016}, or rather their proofs. Namely, we claimed that the proof of the following results, originally stated in terms of \emph{finite} groups, rather than \emph{infinite} groups, also hold in our setting. 
Note that these results are stated in terms of all groups in \cite{jzt2016} and there is no \textit{explicit} assumption that the group $\Gamma$ is finite. However, it is reasonable to assume that there is an \textit{implicit} assumption on the order of the group, since the paper is only concerned with finite symmetric graphs (and therefore finite groups). Hence, for completeness, we consider the infinite case in this appendix. Since the proof turns out to be essentially identical to that given in \cite{jzt2016} we provide only a brief summary.

\begin{theorem}[{\cite[Theorem 4.4]{jzt2016}}]
\label{app: recur. con. 1}
    Let $\Gamma$ be a group. A $\Gamma$-gain graph is $(2,3,1)$-gain tight if and only if it can be built up from a single vertex with an unbalanced loop by applying a series of 0-extensions, 1-extensions and loop-1-extensions.
\end{theorem}

\begin{theorem}[{\cite[Theorem 7.9]{jzt2016}}]
\label{app: recur. con. 2}
    Let $\Gamma$ be a group. A $\Gamma$-gain graph is  $(2,3,1,0)$-gain tight if and only if it can be built up from a disjoint union of base graphs by applying a series of $0$-extensions, $1$-extensions, loop-$1$-extensions, $2$-extensions, and loop-$2$-extensions. 
\end{theorem}

For the remainder of the appendix, we call any $(2,3,1)$-gain sparse (respectively, tight) graph and any $(2,3,1,0)$-gain sparse (respectively, tight) graph simply \textit{sparse} (respectively, \textit{tight}).
It is relatively straightforward to show that extension operations applied to tight graphs yield tight graphs. However the converse is highly non-trivial. For a standard induction argument to work one needs to guarantee the existence of a vertex that admits a reduction. 
The authors of \cite{jzt2016} start by showing the following statement. 

\begin{lemma}[{\cite[Lemma 2.4]{jzt2016}}]
\label{unions lemma}
    Let $\Gamma$ be a (finite) group, $(G,\psi)$ be a $\Gamma$-gain graph, and $H_1,H_2$ be connected subgraphs of $G$ such that $H_1\,\cap\, H_2$ is connected. Then, the following hold:
    \begin{itemize}
        \item[(i)] If $H_1,H_2$ are balanced, then so is $H_1\,\cup\, H_2$; and
        \item[(ii)] If $H_1$ is cyclic and $H_2$ is balanced, then $H_1\,\cup\, H_2$ is cyclic.
        \item[(iii)] If $H_1,H_2$ are cyclic and $H_1\,\cap\,H_2$ is unbalanced, then $H_1\,\cup\, H_2$ is cyclic.
    \end{itemize}
\end{lemma}

The general proof strategy for \Cref{unions lemma} adopted by the authors of \cite{jzt2016} is to take a spanning tree $T$ of $H_1\,\cup\, H_2$ such that $T\,\cap\,E(H_i)$ is a spanning tree of $H_i$ (for $i=1,2$) and such that $T\,\cap\,E(H_1\,\cap\,H_2)$ is a spanning tree of $H_1\,\cap\,H_2$. The existence of $T$ is guaranteed by the fact that $H_1,H_2,H_1\,\cap\,H_2$ are connected. 

By \Cref{lem:treegain}, all edges in $T$ can be assigned identity gain. For (i), then it suffices to note that, by \Cref{prop:balgain}, all edges in $E(H_i-T)$ also have identity gain, and so all edges in $E(H_1\,\cup\,H_2)$ have identity gain. Similarly, for (ii), the group $\langle H_1\,\cup\,H_2\rangle_{v,\psi}$ (for some $v\in V(H_1,\cup\,H_2)$) is
\begin{equation*}
    \langle\psi(e):e\in E(H_1\,\cup\,H_2-T)\rangle=\langle\psi(e):e\in E(H_2-T)\rangle,
\end{equation*}
which is cyclic. For (iii), the assumption that $H_1\,\cap\,H_2$ is unbalanced is crucial: given $e\in E(H_1\,\cap\,H_2-T)$ with unbalanced gain, the authors of \cite{jzt2016} consider the maximal cyclic group $\overline{\mathcal{C}}$ containing $\langle\psi(e)\rangle$ and note that $\langle H_1\,\cup\,H_2\rangle_{v,\psi}$ (for some $v\in V(H_1\,\cup\,H_2)$) is contained in $\overline{\mathcal{C}}$, proving the result. In our case, it suffices to note that $C_{\Gamma}(g)$ contains all gains of the elements in $E(H_1-T)$ and $E(H_2-T)$. Since $g\neq\text{id}$ and $\Gamma$ is a Fuchsian group (see Theorems 2.3.3 and 2.3.5 in \cite{katok1992fuchsian}), $C_{\Gamma}(g)$ is cyclic, proving case (iii).

\begin{lemma}[{\cite[Lemma 2.5]{jzt2016}}]
\label{lem: un. of bal.}
    Let $\Gamma$ be a (finite) group, $(G,\psi)$ be a $\Gamma$-gain graph, and $H_1,H_2$ be connected balanced subgraphs of $G$ such that $H_1\,\cap\, H_2$ has two connected components. Then, $H_1\,\cup\, H_2$ is cyclic.
\end{lemma}

The proof for \Cref{lem: un. of bal.} given in \cite{jzt2016} is similar. Taking a spanning tree $T$ of $H_1\,\cup\,H_2$ such that $E(H_1)\,\cap \,T$ is a spanning tree of $H_1$, we have that $E(H_2)\,\cap\,T$ has two connected components $T_1$ and $T_2$. This partitions $V(H_2)$ into $V_1:=V(H_2)\,\cap\,V(G[T_1]),V_2:=V(H_2)\,\cap\,V(G[T_2])$ and $V_3=V(H_2-V_1-V_2)$. Assuming that all edges in $T$ have identity gain, the balanced condition on $H_1,H_2$ implies that all edges of $G[V_1],G[V_2]$ and $H_1$ have identity gain and that all edges of $G[V_3]$ have the same gain (up to group inverse).

\begin{lemma}[{\cite[Lemma 7.2]{jzt2016}}]
\label{lem: un. of tight is tight}
    Let $\Gamma$ be a (finite dihedral) group and $(G,\psi)$ be a sparse $\Gamma$-gain graph. Let $H_1,H_2$ be tight subgraphs of $(G,\psi)$ such that $E(H_1\,\cap\, H_2)\neq\emptyset$. Then, $H_1\,\cup\, H_2$ is tight.
\end{lemma}

The authors of \cite{jzt2016} used \Cref{unions lemma,lem: un. of bal.} to show \Cref{lem: un. of tight is tight}. In turn, \Cref{lem: un. of tight is tight} was used to show that sparsity induces a matroid\footnote{They state the sparsity conditions in terms of \emph{submodular functions} which often induce matroids, for example via Edmonds theorem. Since we do not work with submodular functions in this article we chose not to introduce this viewpoint here.} whose independent sets are given by the edge sets of sparse graphs. This matroid structure was used to prove that tight graphs admit reductions. Aside from  \Cref{unions lemma,lem: un. of bal.}, the arguments given in \cite{jzt2016} use combinatorial and matroidal arguments that translate directly to our setting.

As mentioned, the technical difficulty in the proofs of \Cref{app: recur. con. 1,app: recur. con. 2} is to show a reduction is always possible. The fact that sparsity induces a matroid is the crucial component that allows the authors of \cite{jzt2016} to prove this. 
In brief, it is easy to see that any tight graph either has a vertex of degree two or three, or that it is 4-regular (in which case is it $(2,3,1,0)$-gain tight). The authors of \cite{jzt2016} showed that a graph with a degree three vertex always admits a reduction, proving the result for the case where the graph is not 4-regular.
The 4-regular case was more involved, as not every vertex admits a 2-reduction or a loop-2-reduction. However, each case can be considered separately and, if a given vertex does not admit a reduction, it was shown that there is always a different vertex that admits a reduction (unless the graph is a base graph).
\end{document}